\documentclass[a4paper,12pt,reqno]{amsart}
\usepackage{amssymb}
\usepackage{amsmath}
\usepackage{enumitem}
\usepackage{mathrsfs}
\usepackage[all]{xy}
\usepackage{mathtools}
\usepackage{hyperref}
\usepackage{url}
\usepackage{bbm}

\usepackage{tikz}

\usepackage[OT2,T1]{fontenc}
\DeclareSymbolFont{cyrletters}{OT2}{wncyr}{m}{n}
\usepackage[british]{babel}

\usepackage[margin=1.3in]{geometry}
\numberwithin{equation}{section} \numberwithin{figure}{section}

\DeclareMathOperator{\Pic}{Pic}

\DeclareMathOperator{\Spec}{Spec}
 \DeclareMathOperator{\rank}{rank}
\DeclareMathOperator{\Hom}{Hom}

\DeclareMathOperator{\Eff}{Eff}
\DeclareMathOperator{\Br}{Br}

\DeclareMathOperator{\lcm}{lcm}

\DeclareMathOperator{\meas}{meas}

\DeclareSymbolFont{cyrletters}{OT2}{wncyr}{m}{n}
\DeclareMathSymbol{\Sha}{\mathalpha}{cyrletters}{"58}
\DeclareMathSymbol{\Be}{\mathalpha}{cyrletters}{"42}

\newcommand{\bad}{\mathrm{Bad}}  
\newcommand{\ba}{\boldsymbol\alpha}
\newcommand{\bb}{\boldsymbol\beta}
\newcommand{\bg}{\boldsymbol\gamma}
\newcommand{\bd}{\boldsymbol\delta}

\newcommand{\Zprim}{\mathbb{Z}_{\mathrm{prim}}}

\newcommand{\ep}{\varepsilon}
\newcommand{\beql}[1]{\begin{equation}\label{#1}}
\newcommand{\eeq}{\end{equation}}
\renewcommand{\b}[1]{\mathbf{#1}}

\newcommand{\cl}[1]{\mathcal{#1}}

\newcommand{\card}{\#}
\newcommand{\z}{\mathbf{z}}
\newcommand{\x}{\mathbf{x}}
\newcommand{\w}{\mathbf{w}}
\newcommand{\y}{\mathbf{y}}
\renewcommand{\u}{\mathbf{u}}
\renewcommand{\v}{\mathbf{v}}
\newcommand{\sL}{\mathsf{\Lambda}}
\newcommand{\LL}{\mathcal{L}}
\newcommand{\lp}{\log^+}
\renewcommand{\1}{\mathbf{1}}
\renewcommand{\d}{\mathrm{d}}
\newcommand{\Main}{\mathrm{Main}}
\newcommand{\Err}{\mathrm{Err}}

\renewcommand{\O}{\mathcal{O}}
\newcommand\F{\mathbb{F}}
\renewcommand\P{\mathbb{P}}
\newcommand\A{\mathbb{A}}
\newcommand\Z{\mathbb{Z}}
\newcommand\ZZ{\mathbb{Z}}
\newcommand\N{\mathbb{N}}

\newcommand\Q{\mathbb{Q}}
\newcommand\R{\mathbb{R}}

\newcommand{\Adele}{\mathbf{A}}

\newcommand{\bx}{\x}

\newcommand{\Mod}[1]{\;(\operatorname{mod}\,#1)}

\newtheorem{lemma}{Lemma}

\newtheorem{theorem}[lemma]{Theorem}
\newtheorem{proposition}[lemma]{Proposition}

\theoremstyle{definition}
\newtheorem{example}[lemma]{Example}

\numberwithin{lemma}{section}

\usepackage{color}

\begin{document}

\title[Quintic del Pezzo surfaces with a
  conic bundle structure]{Manin's conjecture for quintic del Pezzo
  surfaces with a
  conic bundle structure}

\author[D.R. Heath-Brown] {Roger Heath-Brown}
\address{Roger Heath-Brown\\
  Mathematical Institute\\
  Radcliffe Observatory Quarter\\
  Woodstock Road\\
  Oxford\\
  OX2 6GG\\
  UK.}

\author[D. Loughran]{Daniel Loughran}
\address{Daniel Loughran\\
	Department of Mathematical Sciences \\
	University of Bath \\
	Claverton Down \\
	Bath \\
	BA2 7AY \\
	UK.}
\urladdr{https://sites.google.com/site/danielloughran/}


\begin{abstract}
We investigate Manin's conjecture for del Pezzo surfaces of degree
five with a conic bundle structure, proving matching upper and lower bounds,
and the full conjecture in the Galois general case.
\end{abstract}

\maketitle

\thispagestyle{empty}

\tableofcontents

\section{Introduction}

\subsection{Introduction} \label{sec:intro}
A conjecture of Manin \cite{FMT89} predicts an asymptotic formula for
the number of rational points 
of bounded height on Fano varieties. This conjecture has attracted
considerable interest, particularly 
in dimension two where it concerns \textit{del Pezzo surfaces}. The
most famous examples of del Pezzo surfaces are smooth cubic surfaces,
however here Manin's conjecture remains completely open. 

One of the earliest results in the area is transformational work of de
la Bret\'{e}che \cite{dlB02}, who proved Manin's conjecture for split
del quintic del Pezzo surfaces over $\Q$ (a del Pezzo surface is called
\emph{split} if it is the blow-up of $\P^2$ in a collection of
rational points).
This work was subsequently extended by de la Bret\`eche and 
Fouvry \cite{dlBF04} to cover 
surfaces given as the blow-up of $\P^2$ in two rational points and
two geometric points which are conjugate under $\Q(i)$. More recently, the work of
de la Bret\'{e}che \cite{dlB02} was revisited and generalised
to arbitrary number fields \cite{BD24,BD25} by Bernert and Derenthal. These
papers 
used the
\emph{universal torsor approach}. A different approach via conic bundles has also been used to
great success in Manin's conjecture, with Browning \cite{Bro22}
recovering de la Bret\'{e}che's result using the five distinct conic
bundles on the surface, and Browning and Sofos \cite{BS19} obtaining
correct upper and lower bounds for quartic del Pezzo surfaces whenever
there is a conic bundle structure. 

Our first result is a version of \cite[Thm.~1.1]{BS19} but for quintic
del Pezzo surfaces with a conic bundle structure. Every such surface
may be 
embedded as a surface of bidegree $(2,1)$ in $\P^2_\Q \times \P^1_\Q$
(see Lemma \ref{lem:equations}).
The $(-1)$-curves on the surface have an explicit description, given
in Lemma \ref{lem:lines}. We take the open set $U$ to be the
complement of these curves in the surface. Then
$$N(U,B) = \{(\x,\y)\in U(\Q) :H(\x)H(\y)\leq B\}$$
is the standard counting function for
rational points in $U$, where the height function is the biprojective height
\begin{equation} \label{eqn:Height}
	H(\x,\y) = H(\x)H(\y) =
        \max\{|x_0|,|x_1|,|x_2|\}\max\{|y_0|,|y_1|\},
\end{equation}
when $x_i,y_i \in \Z$ and $\gcd(x_0,x_1,x_2) = \gcd(y_0,y_1) =
1$.

\begin{theorem} \label{tub}
	Let 
	$$S: \quad y_0Q_0(x_0,x_1,x_2) + y_1Q_1(x_0,x_1,x_2) = 0 
	\quad \subset \P^2_\Q \times \P^1_\Q$$ be a smooth surface of bidegree
	$(1,2)$ and let $U$ be the complement of the $(-1)$-curves in
        $S$. Then
\[B (\log B)^{\rho -1}\ll N(U,B) \ll B (\log B)^{\rho -1}\]
as $B \to \infty$, where $\rho=\rho(S)=\rank \Pic S$ is the rank of
the Picard group of $S$.
\end{theorem}
Here, and throughout the paper, all implied constants may depend on $S$, and
on the quadratic forms $Q_0$ and $Q_1$ which define $S$.

We can obtain an asymptotic formula under additional assumptions.

\begin{theorem} \label{thm:main}
Let $S$ and $U$ be as in Theorem \ref{tub}. Assume that the
intersection $Q_0(\x)=Q_1(\x)=0$ has no rational points, and that no
non-trivial rational linear combination $a_0Q_0(\x)+a_1Q_1(\x)$ is singular.
Then $\rho=2$ and
	$$N(U,B) \sim c_{S} B\log B $$
	as $B \to \infty$, where $c_S$ is Peyre's constant for the variety $S$.
	Specifically, we have
	\[c_S=\tfrac23\tau_\infty\prod_p\tau_p,\]
	where $\tfrac23$ is Peyre's effective cone constant, and $\tau_\infty$ 
	and $\tau_p$ are the local densities.
\end{theorem}

Usually specific individual surfaces are treated in the literature, so
this verifies for the first time that Manin's conjecture holds for
infinitely many smooth del Pezzo surfaces of degree at most $5$ (in
higher degree they are amenable to harmonic analysis
techniques). Moreover usually only split, or close to split surfaces,
are treated, but our surfaces are far from split: we show in Lemma
\ref{r2} that the conditions in Theorem \ref{thm:main} are equivalent
to $S$ being the blow-up of $\P^2$ in a collection of four Galois
invariant points with Galois action $A_4$ or $S_4$. 

The best previously know result in this generality is a lower bound
of the correct order of magnitude 
$$N(U,B) \gg B(\log B)^{\rho-1}.$$
This follows from general results counting rational points on
conic bundle surfaces, given by Frei, Loughran and Sofos
\cite[Thm.~1.6]{FLS16}. In fact our
analysis provides an alternative proof of this lower bound, which does
not directly use the conic bundle structure (Theorem \ref{extra}).
Note that lower 
bounds are much easier than upper bounds in general. For example
\cite[\S5.1]{FLS16} shows the correct lower bound for Manin's
conjecture for the Fermat cubic surface 
$$x_0^3 + x_1^3 + x_2^3 + x_3^3 = 0$$
but the best known upper bound in this case is
$N(U,B)\ll_{\varepsilon}B^{4/3+\varepsilon}$, for any $\varepsilon>0$,
due to Heath-Brown \cite{HB4/3}.

To prove Theorem \ref{thm:main}, we use Dirichlet's hyperbola trick.
Thus we split $N(U,B)$ as $N_1(U,B) + N_2(U,B)$, where
\begin{align}
\begin{split} \label{def:N_1_N_2}
N_1(U,B) & = \{(\x,\y)\in U(\Q) :H(\x)H(\y)\leq B, H(\x)\leq H(\y)\} \\
N_2(U,B) & = \{(\x,\y) \in U(\Q) : H(\x)H(\y)\le B, H(\x) > H(\y)\}.
\end{split}
\end{align}
In the first case we observe that each $\x$ produces a unique
projective point $[\y]$,
unless $Q_0$ and $Q_1$ both vanish. The height condition on $(\x,\y)$
then produces a number of lattice point counting problems for complicated
regions.

For the second case we fix $\y$ and count points $\x$ lying on the
corresponding conic. This may be done using work of Heath-Brown
\cite{HBconic}, which allows us to count such points $\x$ with a good
uniformity with respect to $\y$. One of our motivations in writing this
paper is to showcase this process, which is crucial to the proof.
Having obtained an asymptotic formula for each individual conic one
then has to sum the leading terms as $\y$ varies. This turns out to be
surprisingly difficult. Indeed, a result of Serre \cite[Thm.~1]{Ser90}
applied to our setting says that at most $O(B/(\log B)^{1/2})$ of the
conics in our family have a  
rational point; we are therefore summing over a sparse set and need to
carefully keep track 
of those conics with a rational point. This is in contrast, for
example, with the quadric bundles  
considered in \cite{BBH, BHBQB}, where
a positive proportion 
of the quadrics in the family have a rational point.

An important feature of our work is that the ranges in which
$H(\x)$ and $H(\y)$ are roughly of order $\sqrt{B}$ make a negligible
contribution towards the main term. In handling $N_1(U,B)$ for
example, one may therefore assume that $H(\x)$
is smaller than $\sqrt{B}$ by a large power of $\log B$.
This turns out to be crucial for controlling our error terms.

We can estimate $N_1(U,B)$ for
all $S$, without needing the 
assumptions imposed in Theorem \ref{thm:main}.
\begin{theorem}\label{extra}
Let $S$ and $U$ be as in Theorem \ref{tub}. Then
\[N_1(U,B)\sim c_{S,1} B(\log B)^{\rho-1},\]
with 
\[c_{S,1}=\frac{2^{1-\rho}}{3(\rho-1)!}\tau_\infty\prod_p\tau_p.\]
\end{theorem}
To prove Theorem \ref{thm:main} we combine Theorem \ref{extra} with the 
following result.
\begin{theorem}\label{x+}
Let $S$ and $U$ be as in Theorem \ref{tub}. Assume that the
intersection $Q_0(\x)=Q_1(\x)=0$ has no rational points, and that no
non-trivial rational linear combination $a_0Q_0(\x)+a_1Q_1(\x)$ is singular.
Then
\[N_2(U,B)\sim c_{S,2} B \log B,\]
with 
\[c_{S,2}=\frac{1}{2}\tau_\infty\prod_p\tau_p.\]
\end{theorem}

We point out that $N_1(U,B)$ and $N_2(U,B)$ do not contribute equally
to $N(U,B)$. In the situation of Theorem \ref{thm:main} we have
\begin{equation} \label{eqn:different_alphas}
  \lim_{B\to \infty} \frac{N_1(U,B)}{N(U,B)} = \frac{1}{4},
  \quad \lim_{B\to \infty} \frac{N_2(U,B)}{N(U,B)} = \frac{3}{4}.
\end{equation}
This is in contrast, for example, with the quadric bundle considered
in the work of Browning and Heath-Brown \cite{BHBQB}, where the
corresponding factors arising from the
hyperbola method contribute asymptotically exactly the same
amount. We give a geometric
explanation of this phenomon in \S \ref{sec:effective_cone}; briefly
the hyperbola condition gives a decomposition of the effective cone
into two regions, with the corresponding volumes producing the factors
$\tfrac14$ and $\tfrac34$ respectively.

A more extreme example of this phenomenon is that there
is no contribution to $N(U,B)$ from points with $H(\x)\le cH(\y)^{1/2}$
if $c>0$ is a small enough constant, in contrast again to the 
situation in Browning and Heath-Brown \cite{BHBQB}. We will see in
\S\ref{sec:effective_cone} that 
this too is explained by the geometry of the effective cone.

Our method of proof allows us to count rational points satisfying real
conditions.  
Given that Theorem \ref{thm:main} also allows arbitrary
equations, one can use a suitable change of variables to impose
$p$-adic conditions. This allows us to prove that the rational points
on $S$ are equidistributed, 
in the sense of Peyre \cite[\S3]{Pey95} (we recall relevant background
on equidistribution in \S\ref{sec:equi}). 

\begin{theorem} \label{thm:equi}
	In the setting of Theorem \ref{thm:main}, the rational points
                of $U$ are equidistributed 
	with respect to Peyre's Tamagawa measure.
\end{theorem}

Theorem \ref{thm:equi} implies for example that Theorem \ref{thm:main}
holds with respect to arbitrary choices of adelic metric on the
anticanonical height (see \cite[Prop.~3.3]{Pey95}). 

\subsubsection*{Acknowledgements} 
We are grateful to P. Salberger for
discussions on the arithmetic and geometry of quintic del
Pezzo surfaces and E. Peyre for discussions on the effective cone. 
DL was supported by UKRI Future Leaders Fellowship \texttt{MR/V021362/1}.

\subsection{Notation and conventions} \label{sec:conventions}
In our proof,  we may multiply $Q_0$ and $Q_1$ in Theorem \ref{tub}
through by a suitable integer constant
to produce forms with integral matrices. Indeed as an artificial
manoeuvre for notational convenience we multiply through by a further
factor of 6 so as to ensure that the cubic form
$C(\y)=\det(y_0Q_0+y_1Q_1)$ vanishes
identically modulo $6^3$. Hence in particular we will automatically
have $6\mid \mathrm{Disc}(C)$. By abuse of notation,  we write $Q_0$
and $Q_1$ for the matrices of the forms $Q_0(\x)$ and $Q_1(\x)$.
Throughout the paper we allow any constants implied by the $\ll$ and 
$O(\ldots)$ notations to depend on the forms $Q_0$ and $Q_1$. 

Because of the length of this paper it has been necessary to give certain 
symbols different meanings in different parts of the paper. We hope this will 
not cause confusion. 

\section{Preliminaries}
In this section we consider the geometry of the surfaces under
consideration, study the leading constant and various arithmetic
functions related to the $p$-adic local densities, and include some
background on lattice point counting. 

\subsection{Geometry}

Let $k$ be a field of characteristic not equal to $2$.

\begin{lemma} \label{lem:equations}
	Let $S$ be a del Pezzo surface of degree $5$ 
	with a conic bundle structure, defined over a field $k$.
	Then $S$ may be embedded as a surface of bidegree $(2,1)$ in
        $\P^2 \times \P^1$,  
	i.e.~there exist quadratic forms $Q_0,Q_1$ defined over $k$, such that
\[S: \quad y_0Q_0(x_0,x_1,x_2) + y_1Q_1(x_0,x_1,x_2) = 0 \subset \P^2
	\times \P^1.\]
	If $k=\Q$, then a choice of anticanonical height is given by
        \eqref{eqn:Height}. 
\end{lemma}

\begin{proof}
  This is proved in the work of Frei, Loughran and Sofos
  \cite[Thm.~5.6]{FLS16}, but we recall
	the argument for completeness. The conic bundle map $\pi:S \to \P^1$
	has $3$ singular fibres over $\bar{k}$, giving $6$
	lines on $S_{\bar{k}}$.
	However the graph of lines of $S$ is the Peterson graph, and
        an inspection 
	of this graph reveals that the other $4$ lines in $S_{\bar{k}}$ are
	pairwise skew. The union of these lines is defined over $k$, and 
	blowing them down gives a map $S \to \P^2$. The product of
	these maps then gives $S \to \P^2 \times \P^1$ which is easily
	checked to be a closed immersion. A degree calculation shows
        that it has the defining equations given.
	By the adjunction formula, the anticanonical bundle for this
        embedding is $\O(1,1)$, and
	we therefore obtain the stated height function.
\end{proof}

We henceforth assume $S$ has the equation
\beql{ppp}
S: \quad y_0Q_0(x_0,x_1,x_2) + y_1Q_1(x_0,x_1,x_2) = 0 \subset \P^2
\times \P^1.
\eeq

Let $\pi_1$ and $\pi_2$ denote the projections onto $\P^1$ and
$\P^2$ respectively in (\ref{ppp}). Then $\pi_1$ is a conic bundle and
$\pi_2$ is a birational morphism.

We give a simple criterion for smoothness.
\begin{lemma} \label{lem:C} \hfill
The polynomial $C(\y)=\det(y_0Q_0+y_1Q_1)$ is non-zero of degree $3$.
%
It is separable if and only if $S$ is smooth, and in this 
case the zero locus of $C$ in $\P^1$ is exactly the locus of 
singular fibres of $\pi_1$, 
and every singular quadric in the pencil has rank $2$.
\end{lemma}
\begin{proof}
	That $C$ has degree $3$ is clear. The remaining properties
  are standard results 
	on pencils of conics and can be found for example in
        Wittenberg \cite[Prop.~3.26]{Wit07}.
\end{proof}

The singular fibres of the conic bundle are thus a pair of lines,
possibly permuted by Galois, meeting in a single point. We say that
such a conic is \emph{split} if each line is defined over $k$. 
We henceforth assume that $S$ is smooth.

\begin{lemma} \label{lem:lines} \hfill
\begin{enumerate}
\item The morphism $\pi_2:S \to \P^2$ is the blow-up 
  in the closed subscheme
  \beql{MDEF}
  M: Q_0(\x) =  Q_1(\x) = 0,
  \eeq
  which is finite \'etale of degree $4$.
\item The complement $U$ of the lines in $S$ is given by
$$U =\{(\x,\y) \in S :(Q_0(\x), Q_1(\x)) \neq (0,0),\quad C(\y) \neq 0\}.$$
	\end{enumerate}
\end{lemma}
\begin{proof}
	(1) The morphism $\pi_2$ is clearly an isomorphism outside of
        $M$, hence it is a birational 
	morphism. Any birational morphism of smooth projective
        surfaces is the composition 
	of blow-ups in smooth points, hence $\pi_2$ is the blow-up of
        $M$ and $M$ is smooth. 
	As $M$ has degree $4$ it is thus finite \'etale of degree $4$.
	
	(2) The singular fibres of $\pi_1$ give $6$ lines over
        $\bar{k}$; the inverse  
	image of $M$ via $\pi_2$ gives $4$ lines over $\bar{k}$. These are all
	the lines, since a del Pezzo surface of degree $5$ has $10$
        lines over $\bar{k}$. 
\end{proof}

A good way to visualise the lines on $\bar{S}$ is as in  Figure
\ref{fig:blow-up}. The surface is the blow up of $\P^2$ in $4$ points
in general position. These give $4$ lines. The conic bundle is given
by the strict transform of the pencil of conics passing through the
$4$ blown-up points. The remaining $6$ lines, which make up the
singular fibres of the conic bundle, arise from the strict transforms
of the lines through the pairs of the blown up points. 

\begin{figure}[htb] 
\begin{center}
\begin{tikzpicture}
\draw (-1,0) -- (3,0);
\draw (-1,2) -- (3,2);
\draw (0,-1) -- (0,3);
\draw (2,-1) -- (2,3);
\draw (-0.75,-0.75) -- (2.75,2.75);
\draw (-0.75,2.75) -- (2.75,-0.75);
\node at (0,0) [circle,fill,inner sep=1.5pt]{};
\node at (0,2) [circle,fill,inner sep=1.5pt]{};
\node at (2,0) [circle,fill,inner sep=1.5pt]{};
\node at (2,2) [circle,fill,inner sep=1.5pt]{};
\end{tikzpicture}
\end{center}
\caption{Blow-up model for $\bar{S}$} \label{fig:blow-up}
\end{figure}
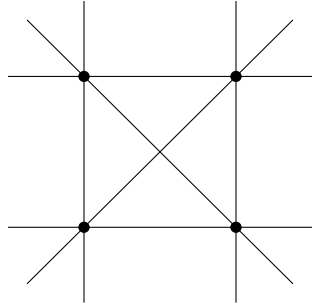 


\begin{lemma} \label{lem:Pic}
	We have the following formulas for $\rank \Pic S$.
	
	\noindent (1) Let $\card M$ be the number of closed points of the subscheme $M$
from \eqref{MDEF}. Then
	$$\rank \Pic S = \card M + 1.$$
	(2) We have
  $$\rank \Pic S = 2 + \#\{\text{closed points $P \in \P^1 : \pi_1^{-1}(P)$
	is singular and split}\}.$$	
\end{lemma}
\begin{proof}
The surface $S$ is obtained by blowing up $\P^2$ in  the irreducible
  components of $M$. Each blow up increases the rank of the Picard group
	by $1$. The result then follows from the fact that $\Pic \P^2 \cong \Z$.
For the second part see Frei, Loughran and Sofos \cite[Lem.~2.1]{FLS16}.
\end{proof}

We give a geometric characterisation of the conditions in Theorem
\ref{thm:main}. 

\begin{lemma} \label{r2}
	The following are equivalent.
	\begin{enumerate}
	\item Both $M$ and the zero locus of $C$ admit no rational point.
\item The surface $S$ is a blow-up of $\P^2$ in a closed point of degree $4$
		with splitting field either $A_4$ or $S_4$.
	\end{enumerate}
If these equivalent conditions hold then $\rho(S)=2$. 
\end{lemma}
\begin{proof}
	We use Figure \ref{fig:blow-up}.
	Let $G \subseteq S_4$ be the Galois group acting on the four
        blow-up points. 
If $G \subseteq D_4$ then the diagonal lines are Galois invariant, whence
	the zero locus of $C$ admits a rational point. If $G \subseteq S_3$
	then $M$ has a rational point. The only remaining cases are $G = A_4$ or
$G=S_4$, where it is easy to see that the zero locus of $C$ has no rational
	point. The last part is Lemma \ref{lem:Pic}.
\end{proof}

Under the conditions of Lemma \ref{r2} the cubic polynomial $C$ is
irreducible, and we are in the $A_4$ case if and only if $C$ has
cyclic splitting field. 

Our next steps are to calculate Peyre's constant
\cite[Def.~2.5]{Pey95} in our setting. 
Our surfaces $S$ are rational, as one can see by solving
the equation for $\y$. It follows that $S$ will satisfy weak
approximation over $\Q$, and that the cohomological constant is equal
to $1$. Hence 
Peyre's constant  takes the form $\alpha \tau_\infty \prod_p \tau_p$
where $\alpha$ is the cone constant and $\tau_v$
are the corresponding local densities.

\subsection{The effective cone constant} \label{sec:effective_cone}
\subsubsection{Definition and calculation}
First we recall the definition of this constant. We let $\Eff(S)$
denote the pseudo-effective cone of $S$, i.e. the closure of the cone
of effective divisors. We let $\Pic(S)^{\wedge} = \Hom(\Pic S, \Z)$
and consider the dual cone $\Eff(S)^{\wedge} \subset
\Pic(S)^{\wedge}$. We equip the real vector space $\Pic(S)^{\wedge}_\R
:= \Pic(S)^{\wedge}$ with the Haar measure such that the lattice
$\Pic(S)^{\wedge}$ has covolume $1$. Then the effective cone constant
is defined by 
\begin{equation} \label{def:alpha}
  \alpha(S) = \frac{1}{(\rho - 1)!}\int_{\Eff(S)^{\wedge}}
  \exp(-\langle -K_S, v\rangle) \d v.
\end{equation}
We define the
\emph{blow-up type} of $S$ as follows. Let $M$ be the closed subscheme
$Q_0(\x) = Q_1(\x) = 0$ of $\P^2$. This is a union of closed points
whose total degree sums to $4$. The blow-up type is then defined to be
the unordered list of the degrees of these closed points. For
example, a split surface corresponds to $(1,1,1,1)$. 

\begin{lemma} \label{lem:alpha}
  \[	\begin{array}{c|c|c|c|c|c|} 
	\textrm{Blow-up type} & (1,1,1,1) & (1,1,2) & (1,3) & (2,2) & (4) \\
		\hline
		\alpha(S) & 1/144 &  1/24 & 5/24  & 1/6 & 2/3
	\end{array}\]
\end{lemma}
\begin{proof}
	This follows from Derenthal, Joyce and Teitler \cite[Table 8]{DJT}.
	(Note that there is a typographical error in the caption: it should 
	say ``degree 5''  and not ``degree 7''). 	
	Briefly, the blow-up types correspond to the following
	rows (counted from the bottom): $(1,1,1,1)$ --- 1; $(1,1,2)$ --- 2; 
	$(1,3)$ --- 8; $(2,2)$ --- 3, 6; and $(4) $--- 4,5,10.
\end{proof}

Note that $\alpha(S)$ is independent of the Galois type.
(For example, there are two possible Galois actions for (2,2)). 

We now study in more detail the effective cone constant in the setting of
Theorem \ref{thm:main}. 

\begin{lemma} \label{lem:effective_cone}
Assume that $S$ is the blow-up of $\P^2$ in a closed point of degree
$4$. Let $E$ 
be the exceptional divisor and $F$ the fibre of the conic bundle. Then
	\begin{enumerate}
	\item $E$ and $F$ generate an index $2$ subgroup of $\Pic(S)$.
	\item The effective cone $\Eff(S)$ is generated by $E$ and $F$.
	\end{enumerate}
\end{lemma}
\begin{proof}
	A basis for the Picard group is given by $E$ and $L$, where $L$ denotes
	the pull-back of a line from $\P^2$. However we have $2L = E+F$
	whence (1) easily follows.
	For (2), recall that the big cone is the interior of the effective cone.
	However neither $E$ nor $F$ is big, so they lie on the
	boundary of the effective cone. Since this is a $2$-dimensional cone
	and $E$ and $F$ are linearly independent, they
	are thus generators.
\end{proof}

We explain for completeness how to show directly from the definition
that $\alpha(S) = 2/3$. We have $-K_X = (E + 3F)/2$. From Lemma
\ref{lem:effective_cone} and  \eqref{def:alpha} we thus obtain 
\begin{equation} \label{eqn:alpha_2/3}
\alpha(S) = \frac{1}{2}\int_{e,f \geq 0} \exp(-(e + 3f)/2)
\d e\, \d f = \frac{2}{3}, 
\end{equation}
in agreement with Lemma \ref{lem:alpha}. (The factor $1/2$ comes from
the fact that $E$ and $F$ only generate an index $2$ subgroup). 

\subsubsection{Hyperbola method}
We now use the effective cone to describe the phenomenon mentioned in
\eqref{eqn:different_alphas}, namely that the two parts of the
hyperbola method contribute different amounts. We achieve this by
using a perspective of Peyre's \cite[\S 4.1]{Pey21}, though our
framework is slightly different (in particular we use a different
measure). 

Let $X$ be a smooth Fano variety over $\Q$. We choose a basis for
$\Pic X$ as well as adelic metrics on the basis. This yields a map 
$$
	\mathbf{h}: X(\Q) \to \Pic(S)^{\wedge}_\R, \quad x \mapsto (L
        \mapsto \log H_L(x)). 
$$
We equip the space $\Pic(S)^{\wedge}_\R$ with the measure
$$\mu(\mathscr{D}) = \frac{1}{(\rho -
  1)!}\int_{\mathscr{D}}\exp(-\langle -K_S,v \rangle)\mathrm{d}v$$ 
for any measurable subset $\mathscr{D} \subset
\Pic(X)^{\wedge}_\R$. Note that the total measure is $\alpha(X)$ by
\eqref{def:alpha}.

One anticipates that the image of $X(\Q)$ under $\mathbf{h}$ should be
equidistributed with respect to $\mu$. For example, let $C \subset
\Eff(X)^{\wedge}$ be a closed connected subcone. Then one expects that
there is a thin subset $\Omega \subset X(\Q)$ such that 
\begin{align} \label{eqn:Peyre_equidistribution}
\begin{split}
	&\lim_{B \to \infty} \frac{\#\{ x \in X(\Q) \setminus \Omega : 
	H_{-K_X}(x) \leq B, \mathbf{h}(x) \in C\}}{\#\{ x \in X(\Q)
    \setminus \Omega :  
	H_{-K_X}(x) \leq B\}} 
	= \frac{\mu(C \cap \Eff(X)^{\wedge})}{\alpha(X)}.
\end{split}
\end{align}

We now apply this framework to our setting. We let $L$ denote the class 
of the pull-back of a line from $\P^2$. By Lemma
\ref{lem:effective_cone} the Picard group is generated by $L$ and $F$
(using the relation $2L = E+F)$. The map $\mathbf{h}$ is given by 
$$
	\mathbf{h}: S(\Q) \to \Pic(S)^{\wedge}_\R, \quad (x,y) \mapsto
        (aL + bF \mapsto a\log H(x) + b \log H(y) ). 
$$
We consider the cone
$$C = \{ v \in \Eff(S)^{\wedge} : \langle v, L \rangle \leq  \langle
v, F \rangle \}.$$ 
We have  $\mathbf{h}(x,y) \in C$ if and only if $H(x) \leq H(y)$. Thus
the following lemma shows that \eqref{eqn:different_alphas} is
compatible with \eqref{eqn:Peyre_equidistribution}. 

\begin{lemma}
	$$\frac{\mu(C)}{\alpha(S)} = \frac{1}{4}, \quad
  \frac{\mu(\Eff(S)^{\wedge} \setminus C)}{\alpha(S)} = \frac{3}{4}.$$ 
\end{lemma}
\begin{proof}
	We use a similar argument to the calculation \eqref{eqn:alpha_2/3}.
	The cone $\Eff(S)$ is self dual and we have the relation $2L = E + F$.
	Therefore in the integral we restrict to the region $(e + f)/2 \leq f$, 
	which is equivalent to $e \leq f$. 
	We obtain
	$$\frac{1}{2}\int_{f \geq e \geq 0} \exp(-((e + 3f)/2))
        \mathrm{d}e\, \mathrm{d}f 
	= \frac{1}{3}\int_{e \geq 0} \exp(-2e) \mathrm{d}f= \frac{1}{6}$$
	which gives the result on noting that $(1/6)/(2/3) = 1/4$.
\end{proof}

In the same way, the condition $H(\x)\le H(\y)^{1/2}$ corresponds
to the subset of the effective cone in which $(e+f)/2\le f/2$. Since 
this has measure zero, the expression \eqref{eqn:Peyre_equidistribution}
predicts that $0\%$ of rational points satisfy this height condition.

\subsection{The real density}\label{trd}
We next consider the real density $\tau_\infty$. We write
\[h_Q(\x)=||(Q_0(\x),Q_1(\x))||_\infty\]
for typographical convenience.
\begin{lemma}\label{lem:rd}
  We have
\begin{equation}\label{tinf}
  \tau_\infty = \int_{\R^2}
  \frac{\d u_1\, \d u_2}{h_Q(1,u_1,u_2)||(1,u_1,u_2)||_\infty}
    =\frac12\int_{||\v||_{\infty}\le 1}\frac{\d v_0\,\d v_1\,\d v_2}{h_Q(\v)}.
\end{equation}
\end{lemma}
\begin{proof}
  The surface $S$ is given by the equation
\[y_0Q_0(x_0,x_1,x_2) + y_1Q_1(x_0,x_1,x_2) = 0,\]
so that there is a real parametrisation of a dense open set, given by
\[\R^2 \to S(\R),
\quad (u_1,u_2) \mapsto (1,u_1,u_2; -Q_1(1,u_1,u_2), Q_0(1,u_1,u_2)).\]
Peyre's formula \cite[\S2.2.1]{Pey95} for the Tamagawa measure 
therefore gives the first expression for $\tau_\infty$.

To establish the second formula we write
\[\int_{||\v||_{\infty}\le 1}\frac{\d v_0\d v_1\d v_2}{h_Q(\v)}=
2\int_0^1\int_{[-1,1]^2}\frac{\d v_1\d v_2}{h_Q(\v)}\d v_0.\]
One may then substitute $v_1=v_0u_1$ and $v_2=v_0u_2$ to produce
\begin{eqnarray*}
  \int_0^1\int_{[-1,1]^2}\frac{\d v_1\d v_2}{h_Q(\v)}\d v_0&=&
  \int_0^1\int_{[-1/v_0,1/v_0]^2}\frac{\d u_1\d u_2}{h_Q(1,u_1,u_2)}\d v_0\\
  &=&\int_{\R^2}\left\{\int_0^{\min(1,1/|u_1|,1/|u_2|)}\d v_0\right\}
    \frac{\d u_1\d u_2}{h_Q(1,u_1,u_2)}\\
&=&\int_{\R^2}\frac{\d u_1\,\d u_2}{h_Q(1,u_1,u_2)||(1,u_1,u_2)||_\infty},
\end{eqnarray*}
as required for the lemma.
\end{proof}

The following shows directly that the second integral in Lemma
\ref{tinf} is finite. We handle more general integrals which will
appear later in our work. 

 \begin{lemma}\label{QB}
   The set
   \begin{equation}\label{ms}
     \{\x\in\R^3:\max(|Q_0(\x)|,|Q_1(\x)|)\le \lambda\}
     \end{equation}
   has measure $O(\lambda^{3/2})$.  Moreover we have
   \[ \int_{[-1,1]^3}\frac{\bigl|\log\max(|Q_0(\x)|,|Q_1(\x)|)\bigr|^k}
   {\max(|Q_0(\x)|,|Q_1(\x)|)}\d x_0\d x_1\d x_2\ll_k 1 \]
  for any $k\ge 0$.
 \end{lemma}
 \begin{proof}
   In the proof we shall replace
$Q_0$ and $Q_1$ by two other forms $Q_0'=aQ_0+bQ_1$ and
$Q_1'=cQ_0+dQ_1$, say, spanning the same pencil
$\langle Q_0,Q_1\rangle_{\R}$. The choice of these will be determined
by $Q_0$ and $Q_1$. The region $\max(|Q_0(\x)|,|Q_1(\x)|)\le \lambda$ will
then be included in the set $\max(|Q_0'(\x)|,|Q_1'(\x)|)\le \kappa$ for
some $\kappa\ll \lambda$, where the order constant depends only on $a,b,c,d$,
and hence only on our original choice of $Q_0$ and $Q_1$. As the
reader will recall our convention is to allow such dependence in our
order constants.  Similarly, we shall make invertible real linear
transformations 
on the vector $\x$, the choice of which will depend only on $Q_0$ and
$Q_1$. These will multiply the measure of the set (\ref{ms}) by a
factor depending only on $Q_0$ and
$Q_1$. Moreover, given such a linear transformation $T$, a region
$\mu\ll ||\x||_{\infty}\ll \mu$ will be included in a
corresponding region $\mu\ll ||T\x||_{\infty}\ll\mu$, with 
new implied constants depending on $T$.
Such transformations too will
have no effect on the claim in Lemma \ref{QB}.

We recall that $C(a_0,a_1)=\det(a_0Q_0+a_1Q_1)$. By Lemma \ref{lem:C}
this is non-zero and separable. Moreover $a_0Q_0+a_1Q_1$
will have rank at least two unless $a_0=a_1=0$. The form
$C(a_0,a_1)$ will have at least
one real linear factor, so that we may assume without loss of
generality that $Q_0$ has rank 2, after choosing a new basis for the
real pencil $\langle Q_0,Q_1\rangle_{\R}$. Moreover, by making a
suitable real linear transformation on the vector $\x$ we may assume
either that $Q_0(\x)=x_0^2+x_1^2$ or that $Q_0(\x)=x_0x_1$. Then
$Q_1(0,0,1)=\xi$ cannot vanish, since the pencil is nonsingular, so that a
further real linear transformation puts $Q_1(\x)$ in the shape
$q(x_0,x_1)+\xi x_2^2$. After these transformations it still suffices to show
that the set on which $Q_0$ and $Q_1$ are $O(\lambda)$ has 
measure $O(\lambda^{3/2})$.
When $Q_0(\x)=x_0^2+x_1^2$ this is now clear.
When $Q_0(\x)=x_0x_1$ we replace $Q_1$ by $Q_1-\tau Q_0$ for suitable
$\tau$ so as to make $Q_1$ diagonal. We may then assume, after a rescaling of
variables, that $Q_1(\x)=\pm x_0^2\pm x_1^2\pm x_2^2$, say. Our task
is now to bound the measure of the set on which $|x_0x_1|\le\kappa$ and
$|\pm x_0^2\pm x_1^2+x_2^2|\le \kappa$, where $\kappa$ is a suitable constant 
multiple of $\lambda$. When $|x_0|$ and
$|x_1|$ are at most $\sqrt{\kappa}$ we must have $|x_2|\le\sqrt{3\kappa}$
and the contribution to the measure is $O(\kappa^{3/2})$. Alternatively,
if $|x_0|\ge\sqrt{\kappa}$ say, then
$|x_1|\le \kappa/|x_0|\le\sqrt{\kappa}$, and
$|\pm x_0^2+x_2^2|\le 2\kappa$. If the coefficient of $x_0^2$ is $+1$ we
get a set of measure $O(\kappa^{3/2})$ again, and otherwise we find that
\[\big| |x_0|-|x_2|\big|\le\frac{2\kappa}{\big| |x_0|+|x_2|\big|}
\le\frac{2\kappa}{|x_0|}.\]
Thus for any $|x_0|\ge\sqrt{\kappa}$ each of $x_1$ and $x_2$ are
restricted to sets of measure $\ll \kappa/|x_0|$.  The total measure
is therefore
\[\ll\int_{\sqrt{\kappa}}^\infty\frac{\kappa^2}{x_0^2}\d x_0\ll
\kappa^{3/2},\]
completing the proof of the first claim in the lemma.

For the second statement we observe as above that for each
$\lambda\ll 1$ we have
\[\mathrm{meas}\left\{\x\in\R^3:
\tfrac12\lambda<\max(|Q_0(\x)|,|Q_1(\x)|)\le\lambda\right\}\ll\lambda^{3/2}.\]
The result then follows, on summing over values of $\lambda=2^{-n}\ll 1$, 
since the sum
\[\sum_{\lambda=2^{-n}\ll 1} \lambda^{3/2}\frac{|\log\lambda|^k}{\lambda}\]
converges.
\end{proof}

\subsection{$p$-adic densities}\label{padicdens}

Our arguments will encounter arithmetic functions associated to
\begin{eqnarray*}
\mathcal{S}:  &y_0 Q_0(\x) +y_1 Q_1(\x) = 0  &\subset \P^2_\Z\times\P^1_\Z,\\
\mathcal{M}:  &Q_0(\x) = Q_1(\x) = 0 &\subset \P^2_\Z,\\
\mathcal{C}:  &C(\y)=\det(y_0 Q_0 + y_1 Q_1) = 0  &\subset \P^1_\Z,\\
\mathcal{D}:  &y_0Q_0(\x) +y_1 Q_1(\x) = 0,\; C(\y)=0 & \subset \P^2_\Z\times\P^1_\Z.
\end{eqnarray*}

For each of these we define a corresponding counting function,
\[f_S(m)=\card\cl{S}(\Z/m\Z),\;\;\; f_M(m)=\card \cl{M}(\Z/m\Z),\]
\[f_C(m)=\card \cl{C}(\Z/m\Z),\;\;\; f_D(m)=\card \cl{D}(\Z/m\Z), \]
so that
\begin{align}
f_S(m)&=\phi(m)^{-2}\card\left\{(\x,\y)\Mod{m}:
\begin{array}{c}
\gcd(\x,m)=\gcd(\y,m)=1,\\
y_0 Q_0(\x) +y_1 Q_1(\x)\equiv 0\Mod{m}
\end{array}\right\}, \nonumber \\
f_M(m)&=\rule{0mm}{7mm}\phi(m)^{-1}
\card\left\{\x\Mod{m}:
\begin{array}{c}
\gcd(\x,m)=1,\\
Q_0(\x) \equiv Q_1(\x)\equiv 0\Mod{m}
\end{array}\right\} \label{af1} \\
f_C(m)&=\rule{0mm}{7mm}\phi(m)^{-1}\card\{\y\Mod{m}:\,\gcd(\y,m)=1,\,
C(\y)\equiv 0\Mod{m}\}, \nonumber
\end{align}
and
$$
f_D(m)=\phi(m)^{-2}
\card\left\{(\x,\y)\Mod{m}:
\begin{array}{c}
\gcd(\x,m)=\gcd(\y,m)=1,\\
y_0 Q_0(\x) +y_1 Q_1(\x)\equiv 0\Mod{m},\\
C(\y)\equiv 0\Mod{m}
\end{array}\right\}.
$$
These functions are all multiplicative. 

The $p$-adic factor in  Peyre's constant includes both a $p$-adic
density and a convergence factor. The proof of Lemma \ref{lem:Pic}
shows that  
$$L(s, \Pic \bar{S}) = \zeta(s) L(s,\Z[\mathcal{M}])$$
where $\Z[\mathcal{M}]$ denotes the Galois representation given by the
free $\Z$-module on the $\bar{\Q}$-points of $\mathcal{M}$. Standard
properties of Artin $L$-functions therefore imply that we may take
$(1-1/p)^{\rho}$ as the convergence factors. Therefore we deduce that 
\[\tau_p =\left(1-\frac{1}{p}\right)^{\rho}\lim_{k\to\infty}
\frac{\#\mathcal{S}(\Z/p^k\Z)}{p^{2k}}=
\left(1-\frac{1}{p}\right)^{\rho}\lim_{k\to\infty}f_S(p^k)p^{-2k}.\] 
We shall next examine the above limit and prove that 
$\prod_p\tau_p$ is convergent.


Our first goal is to calculate $f_S$  in terms of $f_M$. 
Before stating our result we give the following useful lemma
\begin{lemma}\label{ull}
When $(A_0,A_1,p^k)=p^j$ with $0\le j\le k-1$ the congruence
\[A_0y_0+A_1y_1\equiv 0\Mod{p^k}, \quad \gcd(y_0,y_1,p)=1\]
has $p^j\phi(p^k)$ solutions $\y\Mod{p^k}$.
\end{lemma}
\begin{proof}
Let $A_0=B_0p^j$ and $A_1=B_1p^j$. At least one of $B_0$ and $B_1$ is
coprime to $p$,  
and we suppose without loss of generality that $p\nmid B_0$. Then
\[B_0y_0+B_1y_1\equiv 0\Mod{p^{k-j}}\]
so that 
$y_0\equiv -B_0^{-1}B_1y_1\Mod{p^{k-j}}$. Since $k-j\ge 1$ we see that
$\gcd(y_0,y_1,p)=1$ if and only if $p\nmid y_1$. Thus there are
$\phi(p^k)$ admissible values of  
$y_1$ modulo $p^k$, and each of these determines exactly one value of
$y_0$ modulo $p^{k-j}$. The result then follows. 
\end{proof}

We now give our formula relating $f_S$ and $f_M$.
\begin{lemma}\label{afl1}
For any prime $p$ and any $k\ge 1$ we have
\[\frac{f_S(p^k)}{p^{2k}}=1+p^{-1}+p^{-2}+\frac{f_M(p^k)}{p^{k+1}}+
(1-p^{-1})\sum_{j=1}^k\frac{f_M(p^j)}{p^j}.\]
In particular $f_S(p) = p^2 + (f_M(p) + 1)p + 1.$
\end{lemma}
\begin{proof}
For the proof we start from the expression (\ref{af1}) and classify
$\x$ according to the value of the highest common factor
$\gcd(Q_0(\x),Q_1(\x),p^k)$. The number of choices for $\x$ with
$\gcd(Q_0(\x),Q_1(\x),p^k)=1$ will be  
\begin{eqnarray*}
\lefteqn{\card\{\x\Mod{p^k}:\,\gcd(\x,p)=1\}}\\
&&-\card\{\x\Mod{p^k}:\,\gcd(\x,p)=1,\,
Q_0(\x) \equiv Q_1(\x)\equiv 0\Mod{p}\}\\
&=&\left(p^{3k}-p^{3k-3}\right)-p^{3k-3}\phi(p)f_M(p),
\end{eqnarray*}
and for each such $\x$ there will be $\phi(p^k)$ possible $\y$, by
Lemma \ref{ull}.  
This case therefore contributes
\[\frac{\left(p^{3k}-p^{3k-3}\right)-p^{3k-3}\phi(p)f_M(p)}{\phi(p^k)p^{2k}}=
1+p^{-1}+p^{-2}-p^{-2}f_M(p)\]
to $f_S(p^k)p^{-2k}$.

There are $\phi(p^k)f_M(p^k)$ choices of $\x$ with
$\gcd(Q_0(\x),Q_1(\x),p^k)=p^k$. This time there are $p^{2k}-p^{2k-2}$
possible $\y$ for each such $\x$. The corresponding contribution to
$f_S(p^k)p^{-2k}$ is therefore 
\[(1+p^{-1})p^{-k}f_M(p^k).\]

When $\gcd(Q_0(\x),Q_1(\x),p^k)=p^j$ with $j$ in the remaining range $1\le
j\le k-1$ the number of relevant $\x$ is 
\begin{eqnarray*}
  \lefteqn{\card\{\x\Mod{p^k}:\,\gcd(\x,p)=1,\, Q_0(\x)
    \equiv Q_1(\x)\equiv 0\Mod{p^j}\}}\\
&&-\card\{\x\Mod{p^k}:\,\gcd(\x,p)=1,\,
Q_0(\x) \equiv Q_1(\x)\equiv 0\Mod{p^{j+1}}\}\\
&=& p^{3k-3j}\phi(p^j)f_M(p^j)-p^{3k-3j-3}\phi(p^{j+1})f_M(p^{j+1}),
\end{eqnarray*}
and for each such $\x$ there are $p^j\phi(p^k)$ corresponding values
of $\y$, by Lemma \ref{ull}. We therefore get a contribution 
\[\frac{\left(p^{3k-3j}\phi(p^j)f_M(p^j)-p^{3k-3j-3}\phi(p^{j+1})
  f_M(p^{j+1})\right)p^j} 
{\phi(p^k)p^{2k}}=p^{-j}f_M(p^j)-p^{-j-2}f_M(p^{j+1})\]
to $f_S(p^k)p^{-2k}$, for each $j$.

In total we then find that
\begin{eqnarray*}
\frac{f_S(p^k)}{p^{2k}}&=&1+p^{-1}+p^{-2}-p^{-2}f_M(p)+(1+p^{-1})p^{-k}f_M(p^k)\\
&&\hspace{1cm}+\sum_{j=1}^{k-1}\left(p^{-j}f_M(p^j)-p^{-j-2}f_M(p^{j+1})\right),
\end{eqnarray*}
and the lemma follows.
\end{proof}

We can also express $f_S(p)$ in terms of $f_C(p)$ and $f_D(p)$.
\begin{lemma}\label{afl2}
For any prime $p$ we have
\[f_S(p)=p^2+2p+1-(p+1)f_C(p)+f_D(p)\]
and hence
\[pf_M(p)+(p+1)f_C(p)=f_D(p)+p.\]
\end{lemma}
\begin{proof}
One can check this directly in the case $p=2$, using the fact
that $Q_0,Q_1$ and $C$ have been chosen to vanish modulo 2.
In general we work from (\ref{af1}) and first consider points with $p\nmid C(\y)$. There
will be $p^2-1-\phi(p)f_C(p)$ such vectors $\y$. The corresponding
conics are non-singular modulo $p$, so that each produces $p^2-1$
choices of $\x$. When $p\mid C(\y)$ there are $\phi(p)^2f_D(p)$ pairs
$(\x,\y)$ by the definition of $f_D$. It follows that 
\[\phi(p)^2f_S(p)=\left(p^2-1-\phi(p)f_C(p)\right)(p^2-1)+\phi(p)^2f_D(p),\]
yielding the first assertion of the lemma. The second assertion then
follows via  
the case $k=1$ of Lemma \ref{afl1}.
\end{proof}

For primes of good reduction there is a simple formula for
$f_D(p)$. It is convenient to introduce the notation 
\[Q_\y(\x)=y_0 Q_0(\x)+y_1 Q_1(\x).\]
When $Q_\y$ has rank 2 over $\F_p$ we write $\chi(p;\y)=+1$ when the
form factors over $\F_p$ and $\chi(p;\y)=-1$ otherwise; and we set
$\chi(p;\y)=0$ when the rank is different from 2. If
$a=[\y]\in\P^1_{\F_p}$ we write $\chi(p;a)=\chi(p;\y)$, which is
clearly well defined. 
\begin{lemma}\label{LC}
Suppose that $p\nmid\mathrm{Disc}(C)$. Then
\[f_D(p)=f_C(p)+p\sum_{\substack{a\in\P_{\F_p}^1\\ C(a)=0}}(1+\chi(p;a)),\]
\[f_M(p)+f_C(p)=1+\sum_{\substack{a\in\P_{\F_p}^1\\ C(a)=0}}(1+\chi(p;a)),\]
and
\[f_M(p)=1+\sum_{\substack{a\in\P_{\F_p}^1\\ C(a)=0}}\chi(p;a).\]
\end{lemma}
\begin{proof}
If $Q_\y$ were to vanish identically over $\F_p$ for some choice of
$\y$ coprime to $p$ then the cubic form $C$ would have a linear factor
of multiplicity 3 modulo
$p$, contradicting our assumption that $p\nmid\mathrm{Disc}(C)$.
We can therefore assume that $Q(\y)$ always has rank 1 or 2 for
the points $a=[\y]$ under consideration. When $Q_\y$ has rank 1 over
$\F_p$ it has $p+1$ projective zeros modulo $p$. In the rank 2 case
the number of zeros is $2p+1$ if $Q_\y$ factors, and 1 if not. This
produces the first assertion of the lemma. The second part is then an
immediate consequence of  Lemma \ref{afl2}, and the final assertion 
follows from the definition of $f_C$.
\end{proof}

We next consider the size of $f_C(m)$ and $f_M(m)$. 
\begin{lemma}\label{fests}
If $p\nmid\mathrm{Disc}(C)$ we have $f_C(p^e)=f_C(p)\le 3$ for all
$e\ge 1$, while if  
$p\mid\mathrm{Disc}(C)$ we have $f_C(p^e)\ll_p 1$. If $\ep>0$ is fixed
we have $f_C(m)\ll\tau_3(m)\ll_\ep m^{\ep}$. 

There is a finite set $B$ of primes of bad reduction for $\cl{M}$ such
that $f_M(p^e)=f_M(p)\le 4$ for all $p\not\in B$ and all $e\ge 1$.
When $p\in B$ we have $f_M(p^e)\ll_p 1$ for all $e\ge 1$. Moreover if
$\ep>0$ is fixed we have $f_M(m)\ll\tau_4(m)\ll_\ep m^{\ep}$. 
\end{lemma}
Here $\tau_k(m)$ is the $k$-fold divisor function. We hope the reader
will not confuse this with the notation $\tau_p$ used for the local densities
for our del Pezzo surface $S$. We should also remind the reader that our 
order constants are always allowed to depend on
$Q_0$ and $Q_1$, but not on other parameters unless explicitly
mentioned. 

The above result is part of a much more general phenomenon. 
However we have chosen to give an argument in terms of the explicit 
equations for $\mathcal{C}$ and $\mathcal{M}$.

\begin{proof}
If $p\nmid\mathrm{Disc}(C)$ then $C$ cannot vanish identically over
$\F_p$, so that it has at most 3 projective zeros. Moreover Hensel's
Lemma shows that $f_C(p^e)=f_C(p^{e+1})$ for $e\ge 1$. When
$p\mid\mathrm{Disc}(C)$ the $p$-adic valuation of $\nabla C(\y)$ is
bounded below as $\y$ runs over vectors coprime to $p$. Hence Hensel's
Lemma shows that $f_C(p^e)$ is bounded for each such $p$. To estimate
$f_C(m)$ in general we can ignore the boundedly many prime factors of
$m$ which divide 
$\mathrm{Disc}(C)$, and deduce that $f_C(m)\ll \tau_3(m)$, whence
$f_C(m)\ll_\ep m^\ep$ for any fixed $\ep>0$.

To treat $f_M$ we begin by defining quadratic forms
\[\Delta_{ij}=\det\left(\begin{array}{cc} \frac{\partial Q_0}{\partial x_i} &
\frac{\partial Q_0}{\partial x_j} \\
\frac{\partial Q_1}{\partial x_i} & \frac{\partial Q_1}{\partial x_j} 
\end{array}\right),\quad (i,j)=(0,1), (1,2),\mbox{ or } (2,0).\]
These vanish simultaneously if and only if the Jacobian of $Q_0$ and
$Q_1$ has rank less than 2. Since our surface $S$ is smooth it follows
that $Q_0,Q_1,\Delta_{01},\Delta_{12}$ 
and $\Delta_{20}$ vanish simultaneously only at $\x=\mathbf{0}$. The
Nullstellensatz then implies that there are positive integers $R$ and
$d$, and forms  
\[F_{0k}, F_{1k}, F_{01k},F_{12k},F_{20k}\in\Z[x_0,x_1,x_2]\]
such that
\[F_{0k}(\x)Q_0(\x)+F_{1k}(\x)Q_1(\x)+F_{01k}(\x)\Delta_{01}(\x)+
F_{12k}(\x)\Delta_{12}(\x)+F_{20k}(\x)\Delta_{20}(\x)\]
\beql{elim}
=Rx_k^d,\quad k=0,1,2.
\eeq
Since $R$ may be considered fixed once $Q_0$ and $Q_1$ are given, we
may take $B$ as the set of primes dividing $R$. For a prime $p\nmid R$
the relation (\ref{elim}) shows that the intersection $Q_0=Q_1=0$ is
smooth over $\F_p$, and so has at most 4 points. Moreover the Jacobian
of $Q_0$ and $Q_1$ has full rank modulo $p$ at these points, so that any projective
point modulo $p^e$ lifts uniquely to a point modulo $p^{e+1}$, by
Hensel's Lemma. Thus $f_M(p^e)=f_M(p)\le 4$ for $e\ge 1$ whenever
$p\nmid R$. 

Now suppose that $p^r|| R$ with $r\ge 1$. For each $\x_0\in\Z^3$
coprime to $p$ and any $e\ge r$ we will estimate 
\beql{CF}
\card\{\u\Mod{p^e}:\, Q_0(\x_0+p^{r+1}\u)\equiv 
Q_1(\x_0+p^{r+1}\u)\equiv 0\Mod{p^{r+1+e}}\},
\eeq
and it will be sufficient to show that the number of $\u$ is
$O_p(p^e)$, uniformly in $\x_0$. 
We have 
\[Q_i(\x_0+p^{r+1}\u)=Q_i(\x_0)+p^{r+1}\nabla Q_i(\x_0).\u+p^{2r+2}Q_i(\u).\]
Thus there are no possible $\u$ unless $p^{r+1}\mid Q_i(\x_0)$ for $i=1$
and 2, as we now assume. The relation (\ref{elim}) then shows that  
$\gcd(\Delta_{01}(\x_0),\Delta_{12}(\x_0),\Delta_{20}(\x_0),p^{r+1})=p^s$
for some exponent $s\le r$, and we may suppose without loss of generality
that  
$p^s|| \Delta_{01}(\x_0)$. If we now set
\[Q_0'(\x)=\frac{\partial Q_1}{\partial x_1}(\x_0)Q_0(\x)-
\frac{\partial Q_0}{\partial x_1}(\x_0)Q_1(\x)\]
and
\[Q_1'(\x)=-\frac{\partial Q_1}{\partial x_0}(\x_0)Q_0(\x)+
\frac{\partial Q_0}{\partial x_0}(\x_0)Q_1(\x)\]
we find that
\[Q'_0(\x_0+p^{r+1}\u)=Q'_0(\x_0)
+p^{r+1}\{\Delta_{01}(\x_0)u_0-\Delta_{12}(\x_0)u_2\}+p^{2r+2}Q'_0(\u)\]
and
\[Q'_1(\x_0+p^{r+1}\u)=Q'_1(\x_0)
+p^{r+1}\{\Delta_{01}(\x_0)u_1-\Delta_{20}(\x_0)u_2\}+p^{2r+2}Q'_1(\u).\]
If $p^{r+1+e}$ divides both $Q_0(\x_0+p^{r+1}\u)$ and $Q_0(\x_0+p^{r+1}\u)$ then
$p^{r+1+e}$ divides both $Q_0'(\x_0+p^{r+1}\u)$ and $Q_1'(\x_0+p^{r+1}\u)$, and
since $e\ge r$ we now see that there can be no solution $\u$ unless
$p^{r+s+1}$ divides both $Q'_0(\x_0)$ and $Q'_1(\x_0)$. In the latter
case we find that  
$q_0(\u)\equiv q_1(\u)\equiv 0\Mod{p^{e-s}}$, with
\[q_0(\u)=a_0+\lambda u_0-b_0u_2+p^{r-s+1}Q_0'(\u),\quad
q_1(\u)=a_1+\lambda u_1-b_1 u_2+p^{r-s+1}Q_1'(\u),\]
where
\[a_0=p^{-r-s-1}Q_0'(\x_0),\quad a_1=p^{-r-s-1}Q_1'(\x_0),\quad 
\lambda= p^{-s}\Delta_{01}(\x_0),\]
\[b_0=p^{-s}\Delta_{12}(\x_0),\quad\mbox{and}\quad
b_1=p^{-s}\Delta_{20}(\x_0).\]
By assumption $p^s||\Delta_{01}(\x_0)$, so that $\lambda$ is coprime to
$p$. Thus for any fixed value of $u_2$ the congruences $q_0(\u)\equiv
q_1(\u)\equiv 0\Mod{p}$ have a unique solution $(u_0,u_1)\Mod{p}$,
which lifts to a unique solution modulo $p^{e-s}$. We deduce that
the congruences $q_0(\u)\equiv q_1(\u)\equiv 0\Mod{p^{e-s}}$ have
exactly $p^{e-s}$ solutions $\u$ modulo $p^{e-s}$, whence
the counting function (\ref{CF}) is $O_p(p^e)$ as required. This is
enough to show that $f_M(p^e)\ll_p 1$, and the bound
$f_M(m)\ll\tau_4(m)\ll_\ep m^\ep$ then follows by the same argument
that we used for $f_C(m)$.  
\end{proof}

Lemmas \ref{afl1} and \ref{fests} produce the following corollary.
\begin{lemma}\label{tl}
  The sum $\sum_{k=1}^\infty f_M(p^k)p^{-k}$ converges, and the limit
 \[\varpi_p=\lim_{k\to\infty}\frac{f_S(p^k)}{p^{2k}}\]
 exists. Indeed
 \[\varpi_p=1+p^{-1}+p^{-2}+(1-p^{-1})\sum_{j=1}^\infty\frac{f_M(p^j)}{p^j}\]
and
\[\frac{f_S(p^k)}{p^{2k}}=\varpi_p+O(p^{-k-1}).\] 
Finally, in the notation of \S \ref{padicdens} we have
\[\tau_p=(1-p^{-1})^{\rho}\varpi_p=1+\frac{f_M(p)-\rho+1}{p}+O(p^{-2}).\]
\end{lemma}


We have shown that $\varpi_p=\lim_{k\to\infty}f_S(p^k)p^{-2k}$, and it will be convenient to
have also the following alternative expression.
\begin{lemma}\label{varpialt}
Let
\[\widehat{\cl{S}}(q)=\card\left\{(\x,\y)\Mod{q}: y_0 Q_0(\x) +y_1 Q_1(\x)\equiv 0\Mod{q}
\right\}.\]
Then
\[\varpi_p=\lim_{k\to\infty}p^{-4k}\widehat{\cl{S}}(p^k).\]
\end{lemma}
\begin{proof}
We set
\[\cl{N}(q)=\card\left\{(\x,\y)\Mod{q}:\begin{array}{c}
\gcd(\x,q)=\gcd(\y,q)=1,\\ y_0 Q_0(\x) +y_1 Q_1(\x)\equiv 0\Mod{q}\end{array}\right\},\]
so that $f_S(q)=\phi(q)^{-2}\cl{N}(q)$. We can express $\widehat{\cl{S}}(p^k)$ in terms of $\cl{N}(q)$
by classifying the pairs $(\x,\y)$ according to the values $\gcd(\x,p^k)=p^a$ and
$\gcd(\y,p^k)=p^b$. Suppose the contribution to $\widehat{\cl{S}}(p^k)$ from terms with given 
values of $a$ and $b$ is denoted $\widehat{\cl{S}}(p^k;a,b)$. When $2a+b<n$ we write $\x=p^a\u$
and $\y=p^b\v$, so that $\u$ and $\v$ are coprime to $p$ and run modulo $p^{k-a}$ and 
$p^{k-b}$ respectively. Moreover we have $v_0Q_0(\u)+v_1Q_1(\u)\equiv 0\Mod{p^{k-2a-b}}$
It follows that $\widehat{\cl{S}}(p^k;a,b)=p^{7a+3b}\cl{N}(p^{k-2a-b})$. We may also derive a trivial upper 
bound by noting that there are at most $p^{3k-3a}$ choices for $\u$ and at most $p^{2k-2b}$
choices for $\v$. This leads to the estimate $\widehat{\cl{S}}(p^k;a,b)\le p^{5k-3a-2b}$.  We will use this 
trivial estimate when $2a+b\ge\tfrac34 k$, so that $3a+2b\ge \tfrac32(2a+b)\ge\tfrac98 k$. 
Such terms $a,b$ contribute a total $O(k^2p^{4k-k/8})$ to $\widehat{\cl{S}}(p^k)$, so that
\[p^{-4k}\widehat{\cl{S}}(p^k)=\sum_{2a+b<\tfrac34 k} p^{-a-b}\cdot p^{-4(k-2a-b)}\cl{N}(p^{k-2a-b})
+O(k^2p^{-k/8}).\]
However  $p^{-4(k-2a-b)}\cl{N}(p^{k-2a-b})=(1-1/p)^2f_S(p^{k-2a-b})p^{-2(k-2a-b)}$,
which tends to $(1-1/p)^2\varpi_p$ as $k\to\infty$, uniformly for $2a+b\le\tfrac34 k$.
We therefore deduce that
\[p^{-4k}\widehat{\cl{S}}(p^k)\to\sum_{a,b=0}^\infty p^{-a-b}(1-1/p)^2\varpi_p \]
as $k\to\infty$, and the lemma follows on evaluating the infinite sum.
\end{proof}

We next give a more precise description of $f_M(p)$, for ``good'' primes.

\begin{lemma}\label{lem:f1}
Let $m=\card M$ be the number of closed points of the subscheme $M$ from (\ref{MDEF}),
as in Lemma \ref{lem:Pic}, so that $m=\rho-1$. Then
there are number fields $K_1,...,K_m$, not
necessarily distinct, whose degrees sum to 4, and a finite set $B$, such that
\[f_M(p) = \sum_{i=1}^m\#\{\mbox{primes ideals }\mathfrak{p}\subset\O_{K_i}
: N(\mathfrak{p}) = p\}\]
for all primes $p\not\in B$.
\end{lemma}

\begin{proof}
  The scheme
  \[Q_0(\x) = Q_1(\x) = 0 \quad \subset \P^2_\Q\]
is reduced, has degree $4$ and dimension $0$. Thus it may be written
as the disjoint union $\sqcup_{i=1}^m \Spec K_i$ where the $K_i$ are
number fields whose degrees sum to $4$. The assertion of the lemma
then follows.
\end{proof}

We can now examine the asymptotic behaviour of sums involving $f_M(n)$.
\begin{lemma}\label{lem:f2}
  The Dirichlet series
  \[F(s)=\sum_{n=1}^\infty f_M(n)n^{-s}\]
  has an analytic continuation to $\sigma>\tfrac12$ as a meromorphic
  function whose only singularity is a pole of order $\rho-1$ at
  $s=1$. Moreover there is a constant $A$ such that
  $F(\sigma+it)=O(|t|^A)$ in the region $\sigma\ge\tfrac34$, $|t|\ge 1$.

 It follows that there is an exponent $\alpha<1$ and a polynomial $P$ of
  degree $\rho-2$ such that
  \[\sum_{n\le x}f_M(n)=xP(\log x)+O(x^{\alpha}).\]
  Moreover, there is a constant $A'$ depending only on $Q_0$ and $Q_1$
  such that
  \beql{dded}
  \sum_{p\le x}f_M(p)p^{-1}=(\rho-1) \log\log x+A'+O\big((\log x)^{-1}\big).
  \eeq
  Finally, the product
  \[\mathfrak{S}_S:=\prod_p\tau_p\]
  is convergent, with
  \[\prod_{p>x}\tau_p=1+O((\log x)^{-1}).\]
  \end{lemma}
\begin{proof}
  We begin by writing
  \[F(s)=G(s)\prod_{i=1}^m \zeta_{K_i}(s)\]
for an appropriate Euler product $G(s)$, where $m=\rho-1$ as before. For primes $p\not\in B$ Lemma
\ref{lem:f1} shows that the corresponding Euler factor in $G(s)$ will
be $1+O(p^{-2\sigma})$ when $\sigma>0$, while for primes $p\in B$ it is
$1+O(p^{-\sigma})$. Here we use Lemma \ref{fests} to control the size
of $f_M(p^e)$ for $e\ge 2$.

The above estimates for the Euler factors of $G(s)$ show that $G(s)$
is holomorphic for $\sigma>\tfrac12$ and bounded for
$\sigma\ge\tfrac34$, say, and the first assertion of the lemma
follows.

The second assertion is then an easy exercise with Perron's formula.

For the next statement we observe that
\[\sum_{p\le x}f_M(p)=\sum_{i=1}^m\pi(x;K_i)+O(x^{1/2})\]
by Lemma \ref{lem:f1}, the error term accounting for primes $p\in B$
along with any prime ideals of degree 2 or more. The Prime Ideal Theorem
then shows that
\[\sum_{p\le x}f_M(p)=m\int_2^x\frac{\d t}{\log t}+ 
O\big(x\exp(-c\sqrt{\log x})\big)\]
with a constant $c>0$ depending only on the $K_i$ (and hence only on
$Q_0$ and $Q_1$). The claim then follows by partial summation.

Finally, we see from Lemma \ref{tl} that
\[\sum_{x<p\le y}\log\tau_p=
\sum_{x<p\le y}\left\{\frac{f_M(p)-m}{p}+O(p^{-2})\right\}
\ll (\log x)^{-1},\]
by (\ref{dded}). The claimed result then follows.
\end{proof}

We also need information about another function which will occur much
later in the paper. 
\begin{lemma}\label{earlier}
For any integer $N\ge 3$ we define a multiplicative function $f(*;N)$ by setting
  \[f(p;N)=\big(f_M(p)-1\big)\frac{p^2-p}{p(p+1)-f_C(p)},\;\;\;(p>N)\]
  and $f(p^e;N)=0$ if $p\le N$ or $e\ge 2$.  Then if $\rho=2$  we have
  \[\sum_{d\le D}f(d;N)/d=1+O((\log N)^{-1})\]
  for $N\le(\log D)^2$.
\end{lemma}
\begin{proof}
  Since $\rho=2$ we write the Dirichlet series 
\[F(s;N)=\sum_{d=1}^\infty f(d;N)d^{-s}\] 
as
\[F(s;N)=G_1(s)\frac{\zeta_K(s)}{\zeta(s)},\]
where $K$ is the quartic field over which the points of $Q_0=Q_1=0$
are defined. When $p>N$ 
the Euler factors 
of $F(s;N)$ are $1+\big(f_M(p)-1\big)p^{-s}+O(p^{-\sigma-1})$ for $\sigma>0$.
On the other hand, if $B$ is the finite set occurring in Lemma
\ref{lem:f1} and $\sigma>0$, the Euler factor of
$\zeta(s)^{-1}\zeta_{K}(s)$ will be 
$1+\big(f_M(p)-1\big)p^{-s}+O(p^{-2\sigma})$ whenever
$p\not\in B$.  It follows that the Euler factors of $G_1(s)$ are
$1+O(p^{-\sigma-1})+O(p^{-2\sigma})$ for $p>N$, while for $p\le N$
they are trivially $1+O(p^{-\sigma})$.  We then see that $G_1(s)$ will
be holomorphic for $\sigma\ge \tfrac34$, say, and will satisfy the
bound 
\[ G_1(s)\ll\exp\{\sum_{p\le N}O(p^{-3/4})\}\ll\exp(N^{1/4})\]
there.
The normal closure of $K/\mathbb{Q}$ has Galois group $A_4$ or $S_4$ 
which are both solvable. It then follows from work of Uchida 
\cite{Uch} and van der Waall \cite{vdW}, independently, that 
$\zeta_{K}(s)/\zeta(s)$ is holomorphic of finite order. 
A standard application of Perron's formula then shows that
 \[ \sum_{d\le D}f(d;N)/d=\mathrm{Res}
 \left(F(s+1;N)\frac{D^s}{s};\, s=0\right) +O(D^{-c}e^{N^{1/4}})\]
for an appropriate numerical constant $c>0$. The error term is
suitably small when  
$N\le(\log D)^2$. The function $F(s;N)$ has a removable singularity 
at $s=1$ so that the residue is just $\lim_{s\to 0} F(s+1;N)$
We now define temporarily
\[\alpha(s,p)=\log\left(1+f(p;N)p^{-s}\right)\]
for real $s\in[1,2]$. These functions will be continuous, with
\[\alpha(s,p)=(f_M(p)-1)p^{-s}+O(p^{-2})\]
for $p>N$.  It follows from Lemma \ref{lem:f2} coupled with partial
summation that $\sum_p\alpha(s,p)$ is uniformly convergent on $[1,2]$ 
and hence represents a continuous function. Indeed the sum for $p>N$ 
will be $O((\log N)^{-1})$.  We deduce that
\begin{eqnarray*}
\lim_{s\to 0} F(s+1;N)
&=&\exp\left\{\lim_{s\to 0}\sum_p\alpha(s+1,p)\right\}\\
&=&\exp\left\{\sum_p\alpha(1,p)\right\}\\
&=&1+O\left((\log N)^{-1}\right),
\end{eqnarray*}
and the lemma follows
\end{proof}

Finally we record the following 
simple facts about the $k$-fold divisor function $\tau_k(n)$,
which we will use repeatedly in some of our later estimations.
\begin{lemma} \label{lem:tau_k}
For any $x\ge 2$ we have
\[\sum_{n\le x}\tau_k(n)\ll_k x(\log x)^{k-1},\]
and hence
\begin{align*}
	\sum_{n \leq x} \frac{\tau_k(n)}{n^\alpha} &\ll_\alpha
	\begin{cases}
		 x^{1-\alpha} (\log x)^{k-1}, & \alpha < 1, \\
		 (\log x)^k, & \alpha = 1,
	\end{cases} \\
	\sum_{n \geq x} \frac{\tau_k(n)}{n^\alpha} &\ll_\alpha
	x^{1-\alpha} (\log x)^{k-1}, \quad \,\,\, \alpha > 1,
\end{align*}
for any $\alpha \in \R$.

We also have
\[\tau_{k_1}(n)\tau_{k_2}(n) \leq \tau_{k_1 k_2}(n)\]
\end{lemma}
\begin{proof}
The first estimate is a standard application of Perron's formula and the next 
results follow by partial summation. For the final bound it is enough to show
that
\[ \tau_{k_1}(p^{a})\tau_{k_2}(p^{a}) =
\binom{a + k_1 -1}{k_1 -1}  \binom{a + k_2 -1}{k_2 -1} \leq
\binom{a + k_1k_2 -1}{k_1k_2 - 1} = \tau_{k_1k_2}(p^{a}),\]
where the central inequality is easily established by induction on $a$.
\end{proof}

\subsection{Lattices}\label{seclat}

We complete our preliminaries by giving two results on counting points in
lattices. The first is well-known; we include a proof for
completeness.  

\begin{lemma}\label{easyL}
Let $\mathcal{R}\subset\R^2$ be a bounded closed set, whose boundary
is a rectifiable closed curve of length
$L$.  Then if $\sL$ is a 
2-dimensional lattice in $\R^2$ we have
\[\card(\sL\cap\mathcal{R})=\frac{\meas(\mathcal{R})}{\det(\sL)}
  +O(L/\lambda_1)+O(1),\]
  where $\lambda_1$ is the first successive minimum of $\sL$, with
  respect to the $||\cdot||_\infty$-norm.
\end{lemma}

We recall that a curve is said to be rectifiable if it has finite 
arc-length. Any piece-wise continuously differentiable curve will 
be rectifiable. 

\begin{proof}
  This is a version of the well-known Lipschitz Principle. When
  $\sL=\Z^2$ the above result was proved by Landau \cite[page
    186]{landau}. We proceed to deduce the general case. Let
  $\mathbf{u},\mathbf{v}$ be a basis for $\sL$ with
  $||\mathbf{u}||_2=\lambda_1$ and
  $||\mathbf{v}||_2=\lambda_2$. We then want to count pairs
  $(n_1,n_2)\in\Z^2$ with $n_1\mathbf{u}+n_2\mathbf{v}\in\mathcal{R}$. If we
  write $M$ for the matrix with columns $\mathbf{u}$ and $\mathbf{v}$ then
  $|\det(M)|=\det(\sL)$ and
  \[||M^{-1}||_2\ll \lambda_2/\det(\sL)\ll 1/\lambda_1.\]
  Moreover
$n_1\mathbf{u}+n_2\mathbf{v}\in\mathcal{R}$ if and only if
$(n_1,n_2)^T\in M^{-1}\mathcal{R}$. However $M^{-1}\mathcal{R}$ has
 measure $\meas(\mathcal{R})/\det(\sL)$ and boundary length of order
 $||M^{-1}||_2L\ll L/\lambda_1$, and the result follows.
 \end{proof}
 
Given a positive integer $m$ and an element $a \in \P^1(\Z/m\Z)$ we
define a lattice 
\[\sL(a,m) = \{ \y \in \Z^2 : \exists k \in \Z
\mbox{ such that } \y \equiv k a \Mod m\}\] 
For each $m$ there are $m\prod_{p|m}(1+p^{-1})$ distinct lattices of
this sort, which will partition $\Zprim^2$. We will have $\det(\sL(a,m))=m$, 
and hence $\lambda_1\ll m^{1/2}$, where $\lambda_1$ is the first 
successive minimum of $\sL(a,m)$.
  
 Now we adapt  Lemma \ref{easyL} to handle primitive lattice points in $\sL(a,m)$.
\begin{lemma}\label{easyL2}
  Let $\mathcal{R}$ be as in Lemma \ref{easyL}, and suppose that
  \[\mathcal{R}\subseteq[-R,R]\times[-R,R]\]
for some $R\ge 2$.  Then
\begin{eqnarray*}
  \card\{\y\in\Zprim^2\cap\sL(a,m)\cap\mathcal{R}\}&=&
  \frac{6}{\pi^2}\prod_{p\mid
    m}\left(\frac{p}{p+1}\right)\frac{\meas(\mathcal{R})}{m}\\
&&\hspace{1cm}\mbox{}+O\left(\frac{\tau(m)(R+L)}{\lambda_1}\log R\right),
  \end{eqnarray*}
  where $\lambda_1$ is the first successive minimum of $\sL(a,m)$.
\end{lemma}
\begin{proof}
  We begin with the fact that
  \[\card\{\y\in\Zprim^2\cap\sL(a,m)\cap\mathcal{R}\}=
    \card\{\y\in\Zprim^2\cap\sL(a,m)\cap\mathcal{R}:\,\y\not=\mathbf{0}\}\]
\begin{equation}\label{pf1}
=\sum_{d=1}^\infty\mu(d)
\card\{\y\in\Z^2\cap\sL(a,m)\cap\mathcal{R}:\,\y\not=\mathbf{0},\, d\mid\y\}.
\end{equation}
  Here we will have $\y=d\z\in\sL(a,m)$ if and only if
  $\z\in\sL(a,m/\gcd(d,m))$. We write $\lambda_1(d)$ for
  the first successive minimum of $\sL(a,m/\gcd(d,m))$. Then if
  $\y\in\sL(a,m)$ has length $\lambda_1$, we see that $\y$ is also
  in $\sL(a,m/\gcd(d,m))$, whence $\lambda_1(d)\le\lambda_1$. On the
  other hand, if $\z\in\sL(a,m/\gcd(d,m))$ has length $\lambda_1(d)$,
  we see that $\gcd(d,m)\z$ is in $\sL(a,m)$, so that
  $\lambda_1\le\gcd(d,m)\lambda_1(d)$.

  We now observe that
 \begin{eqnarray*}
   \lefteqn{\{\y\in\Z^2\cap\sL(a,m)\cap
     \mathcal{R}:\,\y\not=\mathbf{0},\, d\mid\y\}}\\
   &=&\{d\z:\, \z\in\Z^2\cap\sL(a,m/\gcd(d,m))\cap d^{-1}\mathcal{R},\,
   \z\not=\mathbf{0}\}.
\end{eqnarray*}
This set will be empty if $\lambda_1(d)>R/d$, and hence if
$\lambda_1>R\gcd(d,m)/d$. In general, its cardinality will be
\[\frac{\meas(\mathcal{R})/d^2}{\det(\sL(a,m/(d,m)))}
+O\left(\frac{L/d}{\lambda_1(d)}\right)+O(1),\]
by Lemma \ref{easyL}. It then follows from (\ref{pf1}) that
\begin{eqnarray*}
\lefteqn{\card\{\y\in\Zprim^2\cap\sL(a,m)\cap\mathcal{R}\}}\\
&=&
\sum_{\substack{d\\ d/(d,m)\le 2R/\lambda_1}}\mu(d)\left\{
\frac{\meas(\mathcal{R})/d^2}{m/\gcd(d,m)}
+O\left(\frac{\gcd(d,m)L}{d\lambda_1}\right)+O(1)\right\}.
\end{eqnarray*}
To complete the proof it therefore remains to show that
\begin{equation}\label{pf2}
  \sum_{\substack{d\\ d/\gcd(d,m)\le 2R/\lambda_1}}\mu(d)\frac{(d,m)}{d^2}=
  \frac{6}{\pi^2}\prod_{p\mid m}\left(\frac{p}{p+1}\right)
+O(\tau(m)\lambda_1/R),
  \end{equation}
that
\begin{equation}\label{pf3}
\sum_{\substack{d\\ d/\gcd(d,m)\le 2R/\lambda_1}}\frac{\gcd(d,m)}{d}\ll\tau(m)\log R,
  \end{equation}
and that
\begin{equation}\label{pf4}
\sum_{\substack{d\\ d/\gcd(d,m)\le 2R/\lambda_1}}1\ll\tau(m)R/\lambda_1.  
\end{equation}
These suffice since
$\meas(\mathcal{R})\ll R^2$ and $\lambda_1^2\ll m$.

For (\ref{pf2}) we have
\begin{eqnarray*}
\sum_{\substack{d\\ d/\gcd(d,m)\le 2R/\lambda_1}}\mu(d)\frac{\gcd(d,m)}{d^2}&=&
\sum_{k\mid m}\sum_{\substack{e\le 2R/\lambda_1\\ \gcd(e,m/k)=1}}
\mu(ek)\frac{k}{(ek)^2}\\
&=&\sum_{k\mid m}\frac{\mu(k)}{k}
\sum_{\substack{e\le 2R/\lambda_1\\ \gcd(e,m)=1}}\mu(e)e^{-2}\\
&=&\sum_{k\mid m}\frac{\mu(k)}{k}\left\{
\sum_{\substack{e=1\\ \gcd(e,m)=1}}^\infty\mu(e)e^{-2}+O(\lambda_1/R)\right\}\\
&=&\frac{6}{\pi^2}\prod_{p\mid m}\left(\frac{p}{p+1}\right)
+O(\tau(m)\lambda_1/R).
\end{eqnarray*}
Similarly for (\ref{pf3}) we have
\begin{eqnarray*}
\sum_{\substack{d\\ d/\gcd(d,m)\le 2R/\lambda_1}}\frac{\gcd(d,m)}{d}&=&
\sum_{k\mid m}\sum_{\substack{e\le 2R/\lambda_1\\ \gcd(e,m/k)=1}}e^{-1}\\
&\ll& \tau(m)\log R.
\end{eqnarray*}
Finally for (\ref{pf4}) we note that
\begin{eqnarray*}
\sum_{\substack{d\\ d/\gcd(d,m)\le 2R/\lambda_1}}1&=&
\sum_{k\mid m}\sum_{\substack{e\le 2R/\lambda_1\\ \gcd(e,m/k)=1}}1\\
&\ll& \tau(m)R/\lambda_1.
\end{eqnarray*}
This establishes the required estimates, and so completes the proof.
\end{proof}

\section{Upper bounds via lattices}\label{CVL}

In this section we give an upper bound for the number of solutions of
\[y_0Q_0(\x)+y_1Q_1(\x)=0,\;\;\;(Q_0(\x),Q_1(\x))\not=(0,0)\]
with $\x\in\Z^3_{\mathrm{prim}}$ and $\y\in\Z^2_{\mathrm{prim}}$, subject to
$X<||\x||_\infty\le 2X$ and $Y<||\y||_\infty\le 2Y$.
Later, we will give a more refined version of
this approach, producing an asymptotic formula. 
The parameters $X$ and $Y$ will be used extensively in 
this paper, and we assume throughout that $X,Y\ge 1$.

When $Q_0(\x)$ and $Q_1(\x)$ are not both zero 
we will have
\[\y=\pm\frac{(-Q_1(\x),Q_0(\x))}{\gcd(Q_0(\x),Q_1(\x))}.\]
Hence it suffices to count the $\x$ for which the corresponding 
$\y$ is of the right size. 
The argument will rely critically on counting vectors $\x$ lying in
certain lattices, see Lemma \ref{rl1}. This will be easiest when $\x$
is (roughly) of size $B^{1/2}$ or smaller, as will be the case, in
effect, in the present section.

We point out that if $\x$ and $\y$ are primitive integer vectors, a point 
$(\x,\y)\in S$ will be in the open set $U$ unless
one, or both, of $||\x||_\infty$ or $||\y||_\infty$ is $O(1)$, where the implied
constant may depend on $Q_0$ and $Q_1$, as already explained.

\subsection{Preliminary estimates}\label{S6}
It is convenient to
attach non-negative weights $W_3(X^{-1}\x)$ and $W_2(Y^{-1}\y)$ to
$\x$ and $\y$,
where the subscripts 3 and 2 are reminders that the arguments are in
$\R^3$ and $\R^2$ respectively. 
We assume that $W_3(\u)$ and $W_2(\v)$ are infinitely differentiable,
are even, and are supported on the sets 
\[\tfrac12\le||\u||_\infty\le
\tfrac52\;\;\;\mbox{and}\;\;\; \tfrac12\le||\v||_\infty\le \tfrac52\]
respectively.  
In what follows we will have many error terms that depend on the
choice of the weights $W_3$ and $W_2$. Since these functions
are regarded as fixed, we shall not mention this dependence. 
\bigskip

Our primary object of study will now be
\begin{equation} \label{def:S(X,Y)}
S(X,Y)=\sum_{\substack{\x\in\Z^3_{\mathrm{prim}}\\ (Q_0(\x),Q_1(\x))\not=(0,0)}}
\;\;\;\;\;\sum_{\substack{\y\in\Z^2_{\mathrm{prim}}\\ y_0Q_0(\x)+y_1Q_1(\x)=0}}
W_3(X^{-1}\x)W_2(Y^{-1}\y).
\end{equation}
As noted above we have
\[\y=\pm\frac{(-Q_1(\x),Q_0(\x))}{\gcd(Q_0(\x),Q_1(\x))}\]
when $(Q_0(\x),Q_1(\x))\not=(0,0)$.
In particular one sees that $S(X,Y)$ vanishes unless $Y\ll X^2$, and
that $S(X,Y)\ll 1$ when $X\ll 1$. When
$X\gg 1$ we have $(Q_0(\x),Q_1(\x))\not=(0,0)$, since
$Q_0(\x)=Q_1(\x)=0$ consists of at most 4 projective 
points. Since $W_2$ is even we then deduce that $S(X,Y)=2S_1(X,Y)$
for $X\gg 1$, with
\beql{S1def}
S_1(X,Y)=\sum_{\x\in\Z^3_{\mathrm{prim}}}W_3(X^{-1}\x)
W_2\left(Y^{-1}\frac{(-Q_1(\x),Q_0(\x))}{\gcd(Q_0(\x),Q_1(\x))}\right).
\eeq
Our eventual goal for the current section is now the following, in 
which we introduce the notation $\LL=3+\log XY$ for typographical convenience.
\begin{proposition}\label{P1}
  The sum $S(X,Y)$ vanishes when $Y\gg X^2$, and $S(X,Y)\ll 1$ for
  $X\ll 1$. Moreover, when $1\le X\le Y$ we have
  \begin{equation}\label{esy}
  S(X,Y)\ll XY\LL^{\rho-2}.
  \end{equation}
  \end{proposition}

We have already established the first two claims of the proposition,
and the bound (\ref{esy}) is trivial when $X\ll 1$. We may therefore
assume that $X$ is large enough that $(Q_0(\x),Q_1(\x))\not=(0,0)$ on
the support of $W_3(X^{-1}\x)$, so that we may focus our attention on
$S_1(X,Y)$.  We proceed by classifying our points $\x$ according
to the value $d=\gcd(Q_0(\x),Q_1(\x))$, so that
\beql{sxy*}
S(X,Y)\ll S_1(X,Y)\le \sum_{d=1}^\infty S(d),
\eeq
with
\beql{23a}
S(d)=
\sum_{\substack{\x\in\Z^3_{\mathrm{prim}}\\ d|\gcd(Q_0(\x),Q_1(\x))}}W_3(X^{-1}\x)
W_2\left(Y^{-1}d^{-1}(-Q_1(\x),Q_0(\x))\right).
\eeq

  The primitive vectors $\x$ for which $d\mid\gcd(Q_0(\x),Q_1(\x))$ can be
split up as follows. For $a \in \P^2(\Z/d\Z)$ consider the lattice
\[\sL(a,d) = \{ \x \in \Z^3 : \exists \lambda \in \Z \mbox{ such that
  }(x_0,x_1,x_2) \equiv \lambda a \bmod d\}.\]
Then the $\x$ we are counting can be sorted into $f_M(d)$ disjoint sets,
each lying in one of the lattices 
$\sL(a,d)\subset\Z^3$.  We thus have 
 \begin{equation}\label{Sm}
 S(d)=\sum_{a \in \mathcal{M}(\Z/d\Z)}S(\sL(a,d);d),
   \end{equation}
 where
 \[S(\sL;d)=\sum_{\x\in\Z^3_{\mathrm{prim}}\cap\,\sL}
W_3(X^{-1}\x)W_2\left(Y^{-1}d^{-1}(-Q_1(\x),Q_0(\x))\right).\]

We next record the basic properties of these lattices.  In general we
will denote the successive minima of
a lattice $\sL$ (with respect to $||\cdot||_{\infty}$), by
$\lambda_1\le\lambda_2\le\ldots$.

\begin{lemma} \label{lem:lattices}
	$\sL(a,d)$ is a sublattice of $\Z^3$ of determinant $d^2$, and
	with $\lambda_3 \leq d$.
\end{lemma}
\begin{proof}
	We have the inclusions
	$$(d\Z)^3 \subseteq \sL(a,d) \subseteq \Z^3,$$
	which immediately imply that $\lambda_3\le d$.
	The quotient lattice $\sL(a,d)/(d\Z^3)$ takes the shape
\[ \{ \bx \in (\Z/d\Z)^3 : \exists \lambda \in \Z/d\Z
\mbox{ such that } (x_0,x_1,x_2) \equiv \lambda a \Mod{d}\},\]
which is the set of points in $\A^3_{\Z/d\Z}$ on the line through
$a$. Hence
	$$[\sL(a,d) : (d\Z^3)] = d.$$
	The inclusions above therefore show that
	\[\det \sL(a,d) = [\Z^3:\sL(a,d)] = 
	\frac{[\Z^3: (d\Z)^3]} {[\sL(a,d) : (d\Z^3)]}
	= d^2. \qedhere\]
\end{proof}

For future use we record the following simple bound for $S(\sL;d)$.
 \begin{lemma}\label{rl0}
   For any 3-dimensional lattice $\sL$ we have
   \[S(\sL;d)\ll 1+X^2\lambda_3/\det(\sL)+X^3/\det(\sL).\]
 \end{lemma}
 \begin{proof}
We begin by observing that $W_3(X^{-1}\x)$ is supported
 in a box $||\x||_{\infty}\le \tfrac52 X$.
 If $\lambda_2\gg X$ the only vectors $\x$ to be
 counted will be integer scalar multiples of the shortest vector in
 $\sL$, and since $\x$ is primitive there can be at most 2 such
 vectors. On the other hand, when $\lambda_2\ll X$ we have
 \begin{eqnarray*}
   S(\sL;d)&\ll& (1+X/\lambda_1)(1+X/\lambda_2)(1+X/\lambda_3)\\
   &\ll& X^2/\lambda_1\lambda_2+X^3/\lambda_1\lambda_2\lambda_3\\
   &\ll& X^2\lambda_3/\det(\sL)+X^3/\det(\sL). 
   \end{eqnarray*}
 The lemma then follows.
 \end{proof}

We can now drop the primitivity condition for $\x$ and consider
\beql{S0def}
S_0(\sL;d)=\sum_{\x\in\sL}
W_3(X^{-1}\x)W_2\left(Y^{-1}d^{-1}(-Q_1(\x),Q_0(\x))\right),
\eeq
which satisfies the obvious inequality
\[S(\sL;d)\le S_0(\sL;d).\]
We estimate $S_0(\sL;d)$ with the help of  Lemma \ref{QB}.

\begin{lemma}\label{rl1}
  For any 3-dimensional lattice $\sL$ the sum $S_0(\sL;d)$ vanishes unless
  $d\ll X^2Y^{-1}$, in which case
  \begin{equation}\label{re2}
  S_0(\sL;d)=\frac{I(d)}{\det(\sL)}+
  O_{k}\left(\frac{(Yd)^{3/2}}{\det(\sL)}\left(
\frac{X\lambda_3}{Yd}\right)^k\right),
\end{equation}
for any integer $k\ge 4$,  where
\begin{equation}\label{Adef}
I(d)=\int_{\R^3}W_3(X^{-1}\u)W_2\left(Y^{-1}d^{-1}(-Q_1(\u),Q_0(\u))\right)
\d u_0\d u_1\d u_2.
\end{equation}
Moreover
\beql{pdc}
I(d)\ll (Yd)^{3/2},
\eeq
and it follows that
\[  S_0(\sL;d)\ll \frac{(Yd)^{3/2}}{\det(\sL)},\]
provided that $XY^{-1}\lambda_3\ll d\ll X^2Y^{-1}$. 
\end{lemma}
The first statement of the lemma allows us to restrict attention to 
values of $d$ satisfying $d\ll X^2Y^{-1}$. This condition will occur repeatedly
in our analysis, and we interpret it to mean that $d\le cX^2Y^{-1}$,
with a constant $c=c(Q_0,Q_1)$ such that $S_0(\sL;d')$ always vanishes
for $d'> cX^2Y^{-1}$.
\begin{proof}
  Since $W_3(X^{-1}\x)$ is supported on the set
  $||\x||_\infty\le \tfrac52 X$ and $W_2(Y^{-1}\y)$ is supported on
  $\tfrac12 Y\le||\y||_\infty\le \tfrac52 Y$ one sees that
  $S_0(\sL;d)$ must vanish unless $d\ll X^2Y^{-1}$. For the main
  assertion of the lemma we start by applying the Poisson
  summation formula, which yields
  \begin{eqnarray*}
    \lefteqn{S_0(\sL;d)=}\\
    &&\det(\sL)^{-1}\sum_{\x\in\sL^*}\int_{\R^3}e\big(-\x^T\b{u}\big)
W_3(X^{-1}\u)
W_2\left(\left(-\frac{Q_1(\u)}{Yd},\frac{Q_0(\u)}{Yd}\right)\right)
\d u_0\d u_1\d u_2,
\end{eqnarray*}
where $\sL^*$ is the dual of the lattice $\sL$.
The vector $\x=\b{0}$ produces the main term of Lemma \ref{rl1}, so that
it remains to consider the contribution from non-zero $\x$. For this
we choose an index $j$ for which $|x_j|=||\x||_\infty$, and integrate by
parts repeatedly with respect to $u_j$. The partial derivative
\[\frac{\partial^k}{\partial u_j^k}
W_3(X^{-1}\u)W_2\left(Y^{-1}d^{-1}(-Q_1(\u),Q_0(\u))\right)\]
is supported on the set
$\{\u:||(-Q_1(\u),Q_0(\u)||_\infty\le \tfrac52 Yd\}$,
which has measure $O((Yd)^{3/2})$, by Lemma \ref{QB}. This
also gives us the claimed bound for $I(d)$.
The size of the above partial derivative is
$O(X^{-k})+O((XY^{-1}d^{-1})^k)$, with implied constants that depend
on the choice of our basic weight functions $W_{3}$ and $W_{2}$. 
Since we are assuming that 
$d\ll X^2Y^{-1}$ the partial derivative is
$O((XY^{-1}d^{-1})^k)$. Repeated integration by parts then shows that
\begin{eqnarray*}
  \lefteqn{\int_{\R^3}e\big(-\x^T\b{u}\big)
    W_3(X^{-1}\u)W_2\left(Y^{-1}d^{-1}(-Q_1(\u),Q_0(\u))\right)\d u_0\d u_1\d u_2}
  \hspace{5cm}\\
  &\ll_k & ||\x||_\infty^{-k}(Yd)^{3/2}(XY^{-1}d^{-1})^k.
  \end{eqnarray*}
If we write $\mu_1\le\mu_2\le\mu_3$ for the successive
minima of $\sL^*$ there will be
\[\ll \prod_{i=1}^3(1+T/\mu_i)\]
vectors $\x\in\sL^*$ having $||\x||_\infty\le T$, and the shortest non-zero
$\x$ has modulus $\mu_1$. It follows that
\[\sum_{\substack{\x\in\sL^*\\ \x\not=\mathbf{0}}}||\x||_{\infty}^{-k}
\ll_k{\mu_1}^{-k}\]
as soon as $k>3$. Since $\mu_1$ is of order
$\lambda_3^{-1}$ by Mahler's inequality (see Cassels \cite[page
  219]{Cass}, for example), the estimate (\ref{re2}) follows. Finally,
one may obtain the stated upper bound for $S_0(\sL;d)$ by taking $k=4$.
\end{proof}

\subsection{Proof of Proposition \ref{P1}}

Lemma \ref{lem:lattices} shows that the lattice $\sL(a,d)$ has
$\lambda_3\le d$,  
and has determinant $d^2$. We may therefore apply
Lemma \ref{rl1} whenever $X\le Y$ and $d\ll X^2/Y$, giving us
\begin{equation}\label{use}
  S(\sL(a,d);d)\le S_0(\sL(a,d);d)\ll Y^{3/2}d^{-1/2}.
  \end{equation}
Then (\ref{sxy*}) and  (\ref{Sm}) show that
\[S(X,Y)\ll Y^{3/2}\sum_{d\ll X^2/Y}f_M(d)d^{-1/2},\]
whence the bound (\ref{esy}) follows from
Lemma \ref{lem:f2} via partial summation. \qed

\section{Upper bounds via conics}\label{UBVC}

In this section we count by using the
conic fibration. 
Thus for each $\y\in\Z^2_{\mathrm{prim}}$ we will bound the number of
$\x\in\Z^3_{\mathrm{prim}}$ on the conic 
\[Q_{\y}(\x)=y_0Q_0(\x)+y_1Q_1(\x)=0.\]
This in itself is not hard, given appropriate results from the
literature, but we  
will have an awkward task summing up these bounds as $\y$ varies.
The eventual outcome will be the following.
\begin{proposition}\label{BPCP}
When $Y\gg 1$ we have
\[\sum_{\substack{\y\in\Z^2_{\mathrm{prim}}\\ Y<||\y||_{\infty}\le 2Y}}
\card\{\x\in\Z_{\mathrm{prim}}^3:\,Q_{\y}(\x)=0,\,
||\x||_{\infty}\le X\}\ll \{Y^2+XY\}\LL^{\rho-2}.\]
\end{proposition}
The condition $Y\gg 1$ is merely used to eliminate terms with $C(\y)=0$.
\bigskip

\subsection{Proof of Theorem \ref{tub}}

Before beginning the proof of Proposition \ref{BPCP} we pause to show how it
combines with the upper bound (\ref{esy}) in Proposition \ref{P1} to establish 
Theorem \ref{tub}. For this we define
\begin{eqnarray*}
S^*(X,Y)&=&\sum_{\substack{\x\in\Z^3_{\mathrm{prim}}\\ X<||\x||_\infty\le 2 X}}
\card\{\y\in\Z^2_{\mathrm{prim}}: \, Y<||\y||_\infty\le 2 Y,\,  Q_\y(\x)=0\}\\
&=&\sum_{\substack{\y\in\Z^2_{\mathrm{prim}}\\ Y<||\y||_\infty\le 2 Y}}
\card\{\x\in\Z^3_{\mathrm{prim}}: \, X<||\x||_\infty\le 2 X,\,  Q_\y(\x)=0\}.
\end{eqnarray*}
Then Proposition \ref{BPCP}
shows that 
\beql{sxy}
S^*(X,Y)\ll XY\LL^{\rho-2}
\eeq
when $X\ge Y$ and $Y\gg 1$. For the
alternative range we apply Proposition \ref{P1}, taking
$W_3(\u)$ as a majorant for the
characteristic function of the region
$1<||\u||_\infty\le 2 $ in $\R^3$, supported on the set 
$\tfrac12\le||\u||_\infty\le \tfrac52$, and similarly for $W_2$. The bound
(\ref{esy}) then shows that (\ref{sxy}) holds when $X\le Y$ and $X\gg 1$, and we conclude that
it then holds for all large enough $X$ and $Y$. 
  
We can now consider points for which $||\x||_\infty ||\y||_\infty\le B$. Points with $||\x||_\infty\ll 1$
and $(Q_0(\x),Q_1(\x))\not=(0,0)$ contribute $O(1)$ to $N(U,B)$, since each such $\x$ 
corresponds to just one pair $\pm\y$. Similarly points with $||\y||_\infty\ll 1$
and $C(\y)\not=0$ contribute $O(B)$ to $N(U,B)$, using the familiar fact that a fixed nonsingular
conic has $O(B)$ rational points of height at most $B$. (Indeed this is a special case of
Lemma \ref{MLQB} below.)

We subdivide the remaining values of $||\x||_\infty$ into $O(\LL)$
dyadic intervals $X<||\x||_\infty\le 2 X$. For each
such range the
vector $\y$ must satisfy $||\y||_\infty\le B/X$, and we subdivide this
into dyadic intervals $Y<||\y||_\infty\le 2 Y$ with
$Y=2^{-n}B/X$. For
each such $Y$ we have
\[S^*(X,Y)\ll XY\LL^{\rho-2}\ll 2^{-n}B\LL^{\rho-2}.\]
If we sum over $n$ we find that each of the ranges
$X<||\x||_\infty\le 2 X$ produces $O(B\LL^{\rho-2})$ pairs
$\x,\y$ in total. The upper bound part of Theorem \ref{tub} then
follows, since there are
$O(\LL)$ ranges $X<||\x||_\infty\le 2 X$. The lower bound is
already known from the work of Frei, Loughran and Sofos \cite{FLS16}.

\subsection{Counting points on conics --- Upper bounds}

The literature contains various
estimates for the number of points of bounded height on a conic, and
we will use the following result, which follows from injecting Lemma
2.6 of Browning and Heath-Brown \cite{BHBQ} into the argument of
Heath-Brown \cite[\S 2]{HB4/3}. When $v$ is a valuation of $\Q$
we use the notation $\chi_v(Q)$, defined to be $+1$ when the quadratic
form $Q$ is isotropic over $\Q_v$ and $-1$ if not.  When $Q$ is an
integral form and $p$ is an odd prime
not dividing the determinant of $Q$, one has $\chi_p(Q)=+1$ if and only if 
$Q$ is isotropic over $\F_p$.

\begin{lemma}\label{MLQB}
Let $Q$ be a nonsingular ternary quadratic form with integral matrix
$A$ and determinant of modulus $\Delta$, and let $D$ be the
highest common factor of the $2\times 2$ minors of $A$. Define a
multiplicative function $f_Q(n)$ by taking $f_Q(p^e)=1$ for all $e$
when $p\nmid 2\Delta$, and by setting
\[f_Q(p^e) =
\left\{\begin{array}{cc}
e+1, & \text{if } p=2 \mbox{ or } e\ge 2, \\
1+\chi_p(Q), & \mbox{ if } p\ge 3 \mbox{ and } e=1,
\end{array}\right.\]
 when $p| 2\Delta$.  Then
\[\card\{\x\in\Z_{\mathrm{prim}}^3:\,Q(\x)=0,\,||\x||_{\infty}\le X\}\ll
\kappa(Q)\left\{1+\frac{XD^{1/2}}{\Delta^{1/3}}\right\},\]
where
\[\kappa(Q) =\prod_{p^e||\Delta}f_Q(p^e).\]
\end{lemma}

We note that we always have $0\le\kappa(Q)\le\tau(\Delta)$. Indeed it may 
happen that $\kappa(Q)=0$, in which case $Q(\x)=0$
has no solutions in $\mathbb{Z}_{\mathrm{prim}}^3$. 

We will apply the lemma to $Q(\x)=Q_{\y}(\x)$,
which has determinant $C(\y)$, and we write $D_{\y}$ and $\kappa(\y)$ for the
corresponding $D$ and $\kappa$. When $p||C(\y)$, the form $Q_\y$ will have 
rank 2 over $\mathbb{F}_p$, and it follows that in this situation one has 
$\chi_p(Q_\y)=\chi(p;\y)$, in the notation of Lemma \ref{LC}. 

Thus, if we define a function 
$f(n;\y)$ to be multiplicative with respect to $n$  and to satisfy
\beql{;def}
f(p^e;\y) =
\begin{cases}
e+1, & \text{if } p=2 \mbox{ or } e\ge 2, \\
1+\chi(p;\y), & \mbox{ if } p\ge 3 \mbox{ and } e=1,
\end{cases}
\eeq
we will have
\[\kappa(\y) =\prod_{p^e|| C(\y)}f(p^e;\y).\]
Lemma \ref{MLQB} now shows that
\begin{eqnarray}\label{FF}
\lefteqn{\sum_{\substack{\y\in\Z^2_{\mathrm{prim}}\\ Y<||\y||_{\infty}\le 2Y
    \\ C(\y)\not=0}}\card\{\x\in\Z_{\mathrm{prim}}^3:\,Q_{\y}(\x)=0,\,
||\x||_{\infty}\le X\}}\hspace{3cm}\nonumber\\
&\ll&\sum_{\substack{\y\in\Z^2_{\mathrm{prim}}\\
    Y<||\y||_{\infty}\le 2Y \\ C(\y)\not=0}}\kappa(\y)
\left\{1+\frac{XD_{\y}^{1/2}}{|C(\y)|^{1/3}}\right\},
\end{eqnarray}
and we begin by examining
$D_{\y}$.
\begin{lemma}\label{DyL}
  We have
  \[D_{\y}\ll_{Q_0,Q_1}1\]
  for every $\y\in\Z_{\mathrm{prim}}^2$. 
\end{lemma}
\begin{proof}
Let $p$ be an odd prime not dividing $\mathrm{Disc}(C)$. It follows in 
particular that $p$ is odd. Then applying Lemma \ref{lem:C} to 
  $k=\bar{\F}_p$ implies that $Q_{\y}$ has rank at least 2 over 
$\mathbb{F}_p$, for every $\y\in\Z_{\mathrm{prim}}^2$, whence 
$p\nmid D_{\y}$ for such primes. 
Now suppose $p$ does divide
$\mathrm{Disc}(C)$. We need to show that if $p^e|| D_{\y}$ then
$e$ is bounded in terms of $Q_0$ and $Q_1$, uniformly for 
$\y\in\Z_{\mathrm{prim}}^2$.  Suppose
to the contrary that there is a sequence of vectors
$\y_n\in\Z_{\mathrm{prim}}^2$ for which the corresponding exponents
$e$ tend to infinity. Embedding $\Z_{\mathrm{prim}}^2$ in $\Z_p^2$
we may use compactness to find a subsequence of the $\y_n$ converging
to $\y^*\in\Z_p^2-\{(0,0)\}$, say. It then follows by continuity
that $D_{\y^*}=0$, whence $Q_{\y^*}$ has rank at most 1. We now apply
Lemma \ref{lem:C} with $k=\overline{\Q}_p$, to give us 
the desired contradiction, since $\mathrm{Disc}(C)\neq 0$ over $k$.
\end{proof}

Although we will not need it until much later it seems appropriate at
this point to record the real-variable analogue of Lemma \ref{DyL}.
\begin{lemma}\label{DyLR}
Suppose that the matrix for $Q_\y$ has eigenvalues
$\mu_1,\mu_2,\mu_3$, ordered with $|\mu_1|\le|\mu_2|\le|\mu_3|$.  Then
\[||\y||_\infty\ll|\mu_2|\le|\mu_3|\ll||\y||_\infty,\]
and
\[\frac{|C(\y)|}{||\y||^2_\infty}\ll|\mu_1|\ll\frac{|C(\y)|}{||\y||^2_\infty}.\]
\end{lemma}

\begin{proof}
  Since the matrix for $Q_\y$ has entries of size $O(||\y||_\infty)$
  it is clear that $|\mu_3|\ll||\y||_\infty$. Moreover
  $\mu_1\mu_2\mu_3$ is of exact order $|C(\y)|$, so it suffices to show
  that $|\mu_2|\gg ||\y||_\infty$. Suppose to the contrary we have a
  sequence of vectors $\y_n$ for which the corresponding ratios
  $|\mu_2|/||\y_n||_\infty$ tend to zero. Without loss of generality
  these can be rescaled so that $||\y_n||_\infty =1$. Then, by
  compactness we can pick out a convergent subsequence, tending to
  $\y^*$, say. Since the $\mu_j$ vary continuously with $\y$ it
  follows that the matrix $Q_{\y^*}$ has $\mu_1=\mu_2=0$. However
  this is impossible, since $Q_\y$ always has rank at least 2. This
  completes the proof.
  \end{proof}

In order to use the estimate (\ref{FF}) we now need to understand
how $\kappa(\y)$ behaves on average, but in order to control the
factor $|C(\y)|^{1/3}$ in the denominator we will need to examine
averages over small squares $(Y_0,Y_0+H]\times(Y_1,Y_1+H]$, in which
we take $H=Y^{1/3}$. The result we will prove is the following.
\begin{proposition}\label{kAL}
  Let $|Y_0|,|Y_1|\le 2Y$ and $Y^{1/3}\le H\le Y$, and write
  \[\cl{U}=\Z_{\mathrm{prim}}^2\cap(Y_0,Y_0+H]\times(Y_1,Y_1+H].\]
Then
\[  \sum_{\substack{\y\in\cl{U}\\ C(\y)\not=0\\ d\mid C(\y)}}\kappa(\y)\ll_{\alpha}
H^2d^{-\alpha}\LL^{\rho-2}\]
for any constant $\alpha<1$ and any positive integer $d\le Y^{1/5760}$.
\end{proposition}

For our current purposes we shall just use $d=1$,  but other values will 
be required later. The next section will be devoted to our treatment of 
Proposition \ref{kAL},
but for now we content ourselves by showing how one may use it to
deduce Proposition \ref{BPCP}. 

In fact we will prove a statement slightly more general than Proposition \ref{BPCP}.
\begin{lemma}\label{BPCP+}
For any $\delta>0$ we have
\begin{eqnarray*}
\lefteqn{\sum_{\substack{\y\in\Z^2_{\mathrm{prim}}\\ Y<||\y||_{\infty}\le 2Y
    \\ 0<|C(\y)|\le\delta Y^3}}\card\{\x\in\Z_{\mathrm{prim}}^3:\,Q_{\y}(\x)=0,\,
||\x||_{\infty}\le X\}}\hspace{4cm}\\
&\ll& \{Y^2+(\LL^{-1}+\delta^{2/3})XY\}\LL^{\rho-2}.
\end{eqnarray*}
\end{lemma}
This is clearly sufficient for Proposition \ref{BPCP}, taking $\delta$ to be a suitably 
large constant. The general case will be needed elsewhere later.

\begin{proof}
The region $Y<||\y||_{\infty}\le 2Y$ can be covered using $O(Y^2H^{-2})$
squares of type $\cl{U}$, so that Proposition \ref{kAL} yields
\[\sum_{\substack{\y\in\Z^2_{\mathrm{prim}}\\ Y<||\y||_{\infty}\le 2Y
    \\ C(\y)\not=0}}\kappa(\y)\ll Y^2\LL^{\rho-2}.\]
In view of Lemmas \ref{MLQB} and \ref{DyL} it therefore remains to show that
\beql{C1}
\sum_{\substack{\y\in\Z^2_{\mathrm{prim}}\\ Y<||\y||_{\infty}\le 2Y \\ 0<|C(\y)|\le\delta Y^3}}
\frac{\kappa(\y)}{|C(\y)|^{1/3}}\ll \{\LL^{-1}+\delta^{2/3}\}Y\LL^{\rho-2}.
\eeq

We will apply Proposition \ref{kAL} with $H=Y^{1/3}$.
We begin by asking how many squares $\cl{U}$ are needed to cover
a region given by the constraints $Y<||\y||_{\infty}\le 2Y$ and
$|C(\y)|\le C_0$. Let us choose once and for all a factorization of
$C$ as $L_1L_2L_3$ over $\mathbb{C}$, and write $|L_i(\y)|=\ell_i$.
Suppose that $\ell_i$ is smallest for
$i=1$, say, so that $\ell_1\le C_0^{1/3}$. Let $L_1(\y)=ay_1+by_2$ and
$L_2(\y)=cy_1+dy_2$. The determinant $ad-bc$ will be non-zero since
the form $C$ has no repeated factors. We then find that
$y_1=\{dL_1(\y)-bL_2(\y)\}/\{ad-bc\}$, whence $|y_1|\ll\ell_1+\ell_2$,
since $a,d,b,c$ are fixed, given $Q_1$ and $Q_2$. Recalling that
we took $\ell_1\le\ell_2$, we
deduce that $y_1\ll\ell_2$, and similarly $y_2\ll\ell_2$. However we
assumed that $Y<||\y||_{\infty}\le 2Y$, whence $\ell_2\gg Y$.
In the same way we have
$\ell_3\gg Y$, and the assumption that $|C(\y)|\le C_0$ then yields
$\ell_1\ll C_0Y^{-2}$. The points $\y$ under consideration therefore
lie in the intersection of the square $||\y||_{\infty}\le 2Y$ with the
narrow strip $|L_1(\y)|\ll C_0Y^{-2}$, which may not be aligned with
the axes. Since the linear form $L_1$ is fixed as soon as $Q_0$ and
$Q_1$ are specified, we now see that we can cover the
region in question with 
\beql{ndd}
\ll YH^{-1}(1+C_0Y^{-2}H^{-1})
\eeq
squares $\cl{U}$, bearing in mind that $H=Y^{1/3}\le Y$.

We now subdivide the available range for $C(\y)$ into dyadic
subintervals of the type $C_0/2<|C(\y)|\le C_0$, on each of which
\[\frac{\kappa(\y)}{|C(\y)|^{1/3}}\ll \kappa(\y)C_0^{-1/3}.\]
We require $O(YH^{-1}(1+C_0Y^{-2}H^{-1}))$ squares $\cl{U}$ for
each dyadic subinterval, and each such square $\cl{U}$
contributes
\[\ll C_0^{-1/3}H^2\LL^{\rho-2}\]
to the sum (\ref{C1}), by Proposition \ref{kAL}. We therefore deduce that
\begin{eqnarray*}
 \sum_{\substack{\y\in\Z^2_{\mathrm{prim}}\\ Y<||\y||_{\infty}\le
       2Y \\ C_0/2<|C(\y)|<C_0}}
   \frac{\kappa(\y)}{|C(\y)|^{1/3}}
 &\ll& YH^{-1}(1+C_0Y^{-2}H^{-1})C_0^{-1/3}H^2\LL^{\rho-2}\\
&=&(YHC_0^{-1/3}+Y^{-1}C_0^{2/3})\LL^{\rho-2}
\end{eqnarray*}
for each $C_0\ll Y^3$. We will use this to cover the range
\[Y\LL^3\le |C(\y)|\le\delta Y^3,\]
giving a satisfactory contribution
$O(\{\LL^{-1}+\delta^{2/3}\}Y\LL^{\rho-2})$ to (\ref{C1}). 

To handle small values of $C(\y)$, we bound the number of solutions $\y$
of the equation $C(\y)=n$, when $n\not=0$. It was shown by Lewis
and Mahler \cite{LM} 
that there are $O_{\ep}(|n|^{\varepsilon})$ solutions for any fixed $\ep>0$, with
an implied constant depending on $\ep$ and on the cubic form $C$.
We next observe that 
\beql{ke}
\kappa(\y)\le\tau(|C(\y)|)\ll Y^{\ep},
\eeq
whence  the contribution to (\ref{C1}) arising from terms with
$0<|C(\y)|\le Y\LL^3$ is
\[\ll\sum_{n\le Y\LL^3}\frac{Y^{\ep}}{n^{1/3}}\card\{\y:C(\y)=n\}
\ll\sum_{n\le Y\LL^3}\frac{Y^{2\ep}}{n^{1/3}}\ll Y^{2/3+3\ep}.\]
We then choose $\ep=1/10$, say, to obtain a satisfactory bound.
\end{proof}

\subsection{Proof of Proposition \ref{kAL}}

In proving Proposition \ref{kAL} it naturally suffices to consider the case 
$H=Y^{1/3}$, and we shall assume that we are in this situation for the 
rest of this section. We will draw on ideas from the work of Shiu
\cite{shiu} and its predecessor Erd\H{o}s \cite{Erd}, as exposed in
the proof of Theorem 3.2 of Browning and Heath-Brown \cite{BHBQ}. However
$\kappa(\y)$ is not the result of evaluating a multiplicative function
at values of a binary form, so that previous variants of the method do
not apply directly. The first stage in the argument is summarized in
the following lemma, in which we use the standard functions $P^+(n)$
and $P^-(n)$ to denote the largest and smallest prime factors of $n$
respectively, with the convention that $P^+(1)=1$ and $P_-(1)=\infty$.
\begin{lemma}\label{shiu1}
  Let
  \[U(a;\tau)=\sum_{\substack{\y\in\mathcal{U}\\ a\mid C(\y)\\
    P^-(C(\y)/a)\ge\tau}}f(a;\y)\]
and
\[U_0(a)=\card\{\y\in\mathcal{U}:\,  a\mid C(\y)\},\]
and set $z=H^{1/240}=Y^{1/720}$, so that $d\le z^{1/8}$.  Then
 \beql{TTT}
  \sum_{\substack{\y\in\mathcal{U}\\ C(\y)\not=0\\ d|C(\y)}}\kappa(\y)\ll T_1+T_2+T_3,
  \eeq
with
\[T_1=\sum_{a\le z}U([a,d];z),\]
\[T_2=\sum_{\substack{\log z<q\le z^2\\ q\;\mathrm{prime}}}z^{1500/\log q}
\sum_{\substack{z<a\le z^4\\ P^+(a)=q}}U([a,d];q),\]
and
\[T_3=z^{1/4}\sum_{\substack{z<a\le z^2\\ P^+(a)\le\log z}}U_0([a,d]).\]
\end{lemma}
Here $[a,d]$ is the lowest common multiple of $a$ and $d$, as usual.

\begin{proof}
  We begin by recalling that $\kappa(\y)=f(C(\y);\y)$. We then
observe that the definition of $f(p^e;\y)$ shows that
\[f(p^{e+f};\y)\le f(p^e;\y)\tau(p^f)\]
except possibly when $e=1$ and $f\ge 1$, whence in general
\beql{fin}
f(mn;\y)\le f(m;\y)\tau(n),
\eeq
except possibly when there is a prime $p||m$ such that $p|n$.

Now let $|C(\y)|=p_1\cdots p_r$ with 
\[p_1\le\dots\le p_r,\]
and choose $j$ maximally with $a=p_1\cdots p_j\le z^2$, taking $j=0$ if $p_1>z^2$. We then
put $b=a^{-1}C|(\y)|$. We will consider three cases. If $a\le z$,
then either $j=r$ or $ap_{j+1}>z^2$, since $j$ was chosen
maximally. For this second alternative we have
$p_{j+1}>z^2/a\ge z\ge a\ge p_j$. Either way we are able to
apply (\ref{fin}) to show that
$\kappa(\y)\le f(a;\y)\tau(b)$. Since $p_{j+1}>z$ and
$C(\y)\ll ||\y||_{\infty}^3\ll Y^3$ we see that $\Omega(b)\ll 1$, whence
$\kappa(\y)\ll f(a;\y)$. We should also note
that in this case we have $P^-(b)>z$. We then see that this first
case leads to the term $T_1$ in the lemma.

The second case to consider is that in which $z<a\le z^2$, with
$p_j>\log z$. If $p_j\nmid b$ we set $a'=a$ and $b'=b$, while if 
$p_j|b$ we take $a'=ap_j$ and $b'=b/p_j$. Then (\ref{fin}) will 
apply to $a'$ and $b'$ so that $\kappa(\y)\le f(a';\y)\tau(b')$, with 
$a'\le ap_j\le a^2\le z^4$.   Moreover 
\[\Omega(b')\le\frac{3\log Y+O(1)}{\log p_j},\]
so that $\tau(b')\le2^{\Omega(b')}\ll Y^{(\log 8)/(\log p_j)}$ and
\[\kappa(\y)\ll f(a';\y)z^{1500/\log p_j}.\]
Since $P_-(b')\ge P^+(a')$ we see that this second case leads to 
the term $T_2$ in the lemma.

Finally, in the third case where $z<a\le z^2$ but $p_j\le\log z$ we
will merely use the 
estimate (\ref{ke}) to show that
\[\kappa(\y)\ll z^{1/4}.\]
The lemma then follows on combining these three cases together.
\end{proof}

It turns out that the analysis of $T_2$ is similar to that of $T_1$,
but slightly more complicated. In contrast $T_3$ is rather
easy to bound, since there are rather few integers with $P^+(a)\le\log z$.

We begin to estimate $U(b;\tau)$ and $U_0(b)$ by splitting the sums
into congruence
classes for $\y$. To do this we check from the definition that
$f(p^e;\y)$ only depends
on $\y$ modulo $p^e$, whence $f(b;\y)$ only depends
on $\y$ modulo $b$. Thus if we drop the condition that $\y$ should be
primitive, and set
\begin{eqnarray*}
\lefteqn{U(b;\tau,\y_0)=}\\
&&\card\{\y\in\Z^2\cap(Y_0,Y_0+H]\times(Y_1,Y_1+H]:
\,\y\equiv\y_0(\mbox{mod }b),\,
P^-(C(\y)/b)\ge\tau\}
\end{eqnarray*}
and
\[U_0(b;\y_0)=\card\{\y\in\Z^2\cap(Y_0,Y_0+H]\times(Y_1,Y_1+H]:
\,\y\equiv\y_0(\mathrm{mod}\;b)\},\]
then
\beql{Uat}
U(b;\tau)\le \sum_{\substack{\y_0(\mathrm{mod}\;b)\\ C(\y_0) \equiv 0 \;(\mathrm{mod}\;b)}}
   \hspace{-25pt}{}^*\;\;\;
   f(b;\y_0)U(b;\tau,\y_0)
   \eeq
   and
\beql{Va}
U_0(b)=\sum_{\substack{\y_0(\mathrm{mod}\; b)\\ C(\y_0) \equiv 0 \;(\mathrm{mod}\;b)}}
\hspace{-25pt}{}^*\;\;\; U_0(b;\y_0),
\eeq
where $\Sigma^*$ indicates that $\gcd(y_0,y_1,b)=1$.
We immediately note that
\[U_0(b;\y_0)\ll H^2b^{-2}+1,\]
whence (\ref{Va}) yields
\beql{T3e}
U_0(b)\ll \phi(b)f_C(b)\left\{H^2b^{-2}+1\right\}.
\eeq
To handle $U(b;\tau,\y_0)$ we will use the following lemma, which we
will prove later in this section.

\begin{lemma}\label{LS}
Define an arithmetic function by setting
\[g(k)=k^{-2}\card\{\z(\mathrm{mod}\; k):\,k\mid C(\z)\}.\]  
Then for any $\tau\le H^{1/120}$ we have
\[U(b;\tau,\y_0)\ll
H^2b^{-2}\prod_{\substack{p<\tau\\ p\nmid b\mathrm{Disc}(C)}}(1-g(p))
  +H^{7/4}.\]
\end{lemma}

In the notation of (\ref{af1}) we have 
\[g(p)=\frac{1+(p-1)f_C(p)}{p^2},\]
and since Lemma \ref{fests} shows that $f_C(p)\le 3$  for 
$p\nmid \mathrm{Disc}(C)$ we deduce that
\[\prod_{\substack{p<\tau\\ p\nmid b\mathrm{Disc}(C)}}(1-g(p))\ll
\prod_{\substack{p<\tau\\ p\nmid b\mathrm{Disc}(C)}}(1-f_C(p)/p).\]
Thus, in view of (\ref{Uat}) and Lemma \ref{LS}, for any $\tau\le H^{1/120}$
we will have
\beql{KK}
U(b;\tau)\ll \phi(b)\theta(b)\left\{H^2b^{-2}
\prod_{\substack{p<\tau\\ p\nmid b\mathrm{Disc}(C)}}(1-f_C(p)/p)+H^{7/4}\right\},
\eeq
where 
\beql{thetadef}
\theta(b)=\frac{1}{\phi(b)}
\sum_{\substack{\y_0(\mathrm{mod}\;b)\\ C(\y_0) \equiv 0 \;(\mathrm{mod}\;b)}}
\hspace{-25pt}{}^*\;\;\; f(b;\y_0).
\eeq
Since $f(d;\y)$ only depends on the value of $\y$ modulo $d$,
standard arguments show that $\theta$ is multiplicative.   We
trivially have $0\le f(p^e;\y_0)\le e+1$ and it follows that
\beql{thest}
0\le \theta(b)\le f_C(b)\tau(b)\ll \tau_3(b)\tau(b )\ll\tau(b)^3
\eeq
for every integer $b$, by Lemma \ref{fests}.

We are now ready to estimate the sums $T_1, T_2$ and $T_3$ occurring in
(\ref{TTT}).  We have
\begin{eqnarray*}
T_1&=&\sum_{a\le z}U([a,d];z)\\
&\ll&\sum_{a\le z}\phi([a,d])\theta([a,d])\left\{H^2[a,d]^{-2}
\prod_{\substack{p<z\\ p\,\nmid\, [a,d]\mathrm{Disc}(C)}}(1-f_C(p)/p)+H^{7/4}\right\}.
\end{eqnarray*}
Since $\theta([a,d])\ll ad$ by (\ref{thest}), the term $H^{7/4}$ produces a total
$O(z^3d^2H^{7/4})$. This makes a satisfactory
contribution, $O(H^2d^{-1})$ say, to (\ref{TTT}) since $z= H^{1/240}$ and 
$d\le z^{1/8}$.
Moreover, since $f_C(p)\le 3$ whenever $p\nmid \mathrm{Disc}(C)$, by 
Lemma \ref{fests}, we may
replace the condition $p\nmid [a,d]\mathrm{Disc}(C)$ by
$p\nmid \mathrm{Disc}(C)$ at the cost of a factor $O([a,d]^3\phi([a,d])^{-3})$.
We now have
\beql{T1b}
T_1\ll H^2d^{-1}+
H^2\prod_{\substack{p<z\\ p\nmid \mathrm{Disc}(C)}}(1-p^{-1}f_C(p))
\sum_{a\le z}\frac{[a,d]\theta([a,d])}{\phi([a,d])^2}.
\eeq
Writing $a=uv$, with $u,d1$ coprime and $v \mid d^\infty$ we find that
\beql{uv}
\sum_{a\le z}\frac{[a,d]\theta([a,d])}{\phi([a,d])^2}=
\sum_{\substack{u\le z\\ \gcd(u,d)=1}}\frac{u\theta(u)}{\phi(u)^2}
\sum_{\substack{v\le z/u\\ v| d^\infty}}\frac{[v,d]\theta([v,d])}{\phi([v,d])^2}.
\eeq
For the inner sum we drop the condition $v\le z/u$. Then if $d=\prod p^e$ we see that
\[\sum_{\substack{v\le z/u\\ v| d^\infty}}\frac{[v,d]\theta([v,d])}{\phi([v,d])^2}\le
\prod_{p|d}\left\{\sum_{f=0}^\infty\frac{[p^f,p^e]\theta([p^f,p^e])}{\phi([p^f,p^e])^2}\right\}.\]
When $f\le e$ we have
\[\frac{[p^f,p^e]\theta([p^f,p^e])}{\phi([p^f,p^e])^2}\ll (e+1) p^{-e}\]
by (\ref{thest}), while for $f>e$ we have similarly
\[\frac{[p^f,p^e]\theta([p^f,p^e])}{\phi([p^f,p^e])^2}\ll (f+1)p^{-f}.\]
It follows that
\[\sum_{f=0}^\infty\frac{[p^f,p^e]\theta([p^f,p^e])}{\phi([p^f,p^e])^2}\ll (e+1)^2 p^{-e},\]
and hence that
\[\sum_{\substack{v\le z/u\\ v| d^\infty}}\frac{[v,d]\theta([v,d])}{\phi([v,d])^2}\ll_{\alpha}
d^{-\alpha}\]
for any fixed $\alpha<1$. Inserting this into (\ref{uv}) we see that (\ref{T1b}) becomes
\[T_1\ll H^2d^{-1}+
H^2d^{-\alpha}\prod_{\substack{p<z\\ p\nmid \mathrm{Disc}(C)}}(1-p^{-1}f_C(p))
\sum_{\substack{u\le z\\ \gcd(u,d)=1}}\frac{u\theta(u)}{\phi(u)^2}.\]
For the sum on the right we may drop the condition $\gcd(u,d)=1$, whence
\begin{eqnarray*}
  \sum_{\substack{u\le z\\ (u,d)=1}}\frac{u\theta(u)}{\phi(u)^2}\le \prod_{p\le z}\sigma_p,
\end{eqnarray*}
with
\[\sigma_p=1+\sum_{f=1}^{\infty}p^{-f}\frac{\theta(p^f)}{(1-p^{-1})^2}
=(1+\theta(p)/p)(1+O(p^{-2})).\]
It then follows that
\[T_1\ll H^2d^{-1}+
H^2d^{-\alpha}\prod_{\substack{p<z\\ p\nmid \mathrm{Disc}(C)}}(1-p^{-1}f_C(p))
\prod_{p\le z}(1+p^{-1}\theta(p)).\]
Since $f_C(p)$ and $\theta(p)$ are both $O(1)$ we have
\[\prod_{\substack{p<z\\ p\nmid \mathrm{Disc}(C)}}(1-p^{-1}f_C(p))
\prod_{p\le z}(1+p^{-1}\theta(p))\ll
\exp\left\{\sum_{\substack{p\le z\\ p\nmid \mathrm{Disc}(C)}}
\frac{\theta(p)-f_C(p)}{p}\right\}.\]
However (\ref{thetadef}), (\ref{;def}) and Lemma \ref{LC} show that
\beql{ext}
\theta(p)=f_M(p)+f_C(p)-1,
\eeq
for $p\nmid \mathrm{Disc}(C)$, and
then Lemma \ref{lem:f2} shows that the product is $O(\LL^{\rho-2})$, as
required for Proposition \ref{kAL}.

We turn next to $T_2$, for which (\ref{KK}) yields
\beql{T2T2}
T_2=\sum_{\substack{\log z<q\le z^2\\ q\;\mathrm{prime}}}z^{1500/\log q}\; T_2(q),
\eeq
with
\[T_2(q)=
\sum_{\substack{z<a\le z^4\\ P^+(a)=q}}\phi([a,d])\theta([a,d])\left\{H^2[a,d]^{-2}
\prod_{\substack{p<q\\ p\,\nmid\, [a,d]\mathrm{Disc}(C)}}(1-f_C(p)/p)+H^{7/4}\right\}.\]
The term $H^{7/4}$ contributes a total
$O(z^{12}d^2H^{7/4})$, say, to $T_2(q)$, and since $z=H^{1/240}$ and $d\le z^{1/8}$ 
this is at most $O(H^2d^{-1}z^{-2})$.  Moreover, 
as in our treatment of $T_1$ we may amend the condition 
$p\,\nmid\, [a,d]\mathrm{Disc}(C)$ so as to
become $p\nmid \mathrm{Disc}(C)$, at a cost of a factor
$O([a,d]^3\phi([a,d])^{-3})$. We therefore have
\beql{BB}
T_2(q)\ll H^2d^{-1}z^{-2}+H^2
\prod_{\substack{p<q\\ p\nmid \mathrm{Disc}(C)}}(1-p^{-1}f_C(p))
\sum_{\substack{z<a\le z^4\\ P^+(a)=q}}\frac{[a,d]\theta([a,d])}{\phi([a,d])^2}.
\eeq

The sum over $a$ requires more care than was needed for $T_1$, and we
will apply Rankin's trick.  We set 
\beql{deldef}
\delta=1600/\log q
\eeq
and observe that
\begin{eqnarray}\label{LL}
  \sum_{\substack{z<a\le z^4\\ P^+(a)=q}}\frac{[a,d]\theta([a,d])}{\phi([a,d])^2}&\le&
  \sum_{\substack{z<a\le z^4\\ P^+(a)=q}}\frac{[a,d]\theta([a,d])}{\phi([a,d])^2}
  \left(\frac{a}{z}\right)^{\delta}\nonumber\\
  &\le & z^{-\delta}\sum_{\substack{a=1\\ P^+(a)=q}}^{\infty}
  \frac{a^{\delta}[a,d]\theta([a,d])}{\phi([a,d])^2}\nonumber\\
  &=&z^{-\delta}\prod_{p}\sigma_p.
\end{eqnarray}
for appropriate factors $\sigma_p$. We will have $\sigma_p=1$ unless $p\le q$
or $p|d$. When $p<q$ and $p\nmid d$ our choice (\ref{deldef}) of $\delta$ produces
\begin{eqnarray*}
  \sigma_p&=&1+\sum_{f=1}^{\infty}p^{(\delta-1)f}
  \frac{\theta(p^f)}{(1-p^{-1})^2}\\
  &=& 1+p^{\delta-1}\theta(p)+O(p^{2\delta-2})\\
  &=& 1+p^{-1}\theta(p)+O\left(p^{-1}\frac{\log p}{\log q}\right)+O(p^{-3/2}),
\end{eqnarray*}
 provided that $Y$ (and hence also $H$, $z$ and $q$) is large enough.
Thus
\begin{eqnarray}\label{MM}
  \prod_{\substack{p<q\\ p\nmid d}}\sigma_p&=&
\exp\left\{\sum_{\substack{p<q\\ p\nmid d}}\log \sigma_p\right\}\nonumber\\
&\ll &\exp\left\{\sum_{p<q}\left(p^{-1}\theta(p)+
  O\left(p^{-1}\frac{\log p}{\log q}\right)+O(p^{-3/2})\right)\right\}\nonumber\\
  &\ll& \exp\left\{\sum_{p<q}p^{-1}\theta(p)\right\}.
\end{eqnarray}
We next consider the case in which $p<q$ and $p|d$. If
$p^e||d$ we see that
\[  \sigma_p= (1-p^{-1})^2\left\{\sum_{f=0}^{\infty}
  p^{\delta f-\max(e,f)}\theta(p^{\max(e,f)})\right\}
 \ll (e+1)^4 p^{(\delta-1)e} \]
 by (\ref{thest}). Similarly if $p>q$ and $p|d$ we will have
 \[\sigma_p=\frac{p^e}{\phi(p^e)^2} \theta(p^e)\ll (e+1)^3p^{-e}\]
 when $p^e|| d$. Finally, if $q^e||d$ with $e\ge 0$
  we find that
  \[  \sigma_q= (1-q^{-1})^2\left\{\sum_{f=1}^{\infty}
  q^{\delta f-\max(e,f)}\theta(q^{\max(e,f)})\right\}
 \ll (e+1)^4 q^{(\delta-1)\max(e,1)}. \]
 Combining these various results we conclude that there is a
 constant $C$, depending only on the quadratic forms $Q_0$ 
 and $Q_1$, such that 
 \[\prod_{p|d\,\mathrm{or}\, p=q}\sigma_p\le 
 C^{\omega(d)}\tau(d)^4[d,q]^{\delta-1}\ll_\ep d^{\ep}[d,q]^{\delta-1}\ll_\ep 
 d^{2\ep}[d,q]^{-1}\]
 for any fixed $\ep>0$. Here we use the facts that $d^\delta\ll_\ep d^\ep$ 
 and $q^\delta\ll 1$. 
 
 The estimates (\ref{LL}) and (\ref{MM}) now show that
\[ \sum_{\substack{z<a\le z^4\\ P^+(a)=q}}\frac{[a,d]\theta([a,d])}{\phi([a,d])^2}
\ll_\ep z^{-\delta}d^{2\ep}[d,q]^{-1}\exp\left\{\sum_{p<q}p^{-1}\theta(p)\right\}.\]
In view of  (\ref{BB}) we will need to consider
\[\prod_{\substack{p<q\\ p\nmid \mathrm{Disc}(C)}}(1-p^{-1}f_C(p))
\exp\left\{\sum_{p<q}p^{-1}\theta(p)\right\}.\]
However
\[\prod_{\substack{p<q\\ p\nmid \mathrm{Disc}(C)}}(1-p^{-1}f_C(p))
=\exp\left\{\sum_{\substack{p<q\\ p\nmid \mathrm{Disc}(C)}}
  \log(1-p^{-1}f_C(p))\right\}\]
  and
\begin{eqnarray*}
\lefteqn{\sum_{\substack{p<q\\ p\nmid \mathrm{Disc}(C)}}\log(1-p^{-1}f_C(p))
+\sum_{p<q}p^{-1}\theta(p)}\hspace{2cm}\\
&=&
-\sum_{\substack{p<q\\ p\nmid \mathrm{Disc}(C)}}p^{-1}f_C(p)
+\sum_{p<q}p^{-1}\theta(p)+O(1)\\
&=&\sum_{p<q}\frac{\theta(p)-f_C(p)}{p}+O(1)\\
&=&\sum_{p<q}\frac{f_M(p)-1}{p}+O(1)\\
&=&(\rho-2)\log\log q +O(1),
\end{eqnarray*}
by (\ref{ext}) and Lemma \ref{lem:f2}. It therefore follows from (\ref{BB}) that
\beql{fst}
T_2(q)\ll_{\ep} H^2d^{-1}z^{-2}+H^2(\log q)^{\rho-2}z^{-\delta}d^{2\ep}[d,q]^{-1}.
\eeq
The contribution to (\ref{T2T2}) arising from the first term is
\[\ll H^2 d^{-1}z^{-2}
\sum_{\substack{\log z<q\le z^2\\ q\;\mathrm{prime}}}z^{1500/\log q}\ll H^2 d^{-1}.\]
When $q\mid d$ the second term of (\ref{fst}) is 
$O_\ep(H^2\LL^{\rho-2}z^{-1600/\log q}d^{2\ep-1})$, giving a contribution
\[\ll_\ep H^2\LL^{\rho-2}d^{3\ep-1}\]
to (\ref{T2T2}), since there are $O_\ep(d^\ep)$ primes $q|d$. Finally, when $q\nmid d$ the
second term of (\ref{fst}) is $O_\ep(H^2\LL^{\rho-2}z^{-1600/\log q}d^{2\ep-1}q^{-1})$, 
and the corresponding contribution to (\ref{T2T2}) is
\[\ll_{\ep}H^2 \LL^{\rho-2}d^{2\ep-1}
\sum_{\substack{\log z<q\le z^2\\ q\;\mathrm{prime}}}z^{-100/\log q}q^{-1}.\]
However
\[\sum_{\substack{z^{1/k}<q\le z^{2/k}\\ q\;\mathrm{prime}}}z^{-100/\log q}q^{-1}
\ll e^{-50k}\sum_{\substack{z^{1/k}<q\le z^{2/k}\\ q\;\mathrm{prime}}}q^{-1}
\ll e^{-50k},\]
and summing over $k$ we see that primes $q\nmid d$ contribute 
$O_\ep(H^2 \LL^{\rho-2}d^{2\ep-1})$. Thus
\[T_2\ll_\ep H^2d^{3\ep-1}\LL^{\rho-2},\]
which suffices for Proposition \ref{kAL}, on choosing $\ep=(1-\alpha)/3$.

Finally we examine $T_3$. The estimate (\ref{T3e}) shows that
\[T_3\ll z^{1/4}\sum_{\substack{z<a\le z^2\\ P^+(a)\le\log z}}
\phi([a,d])f_C([a,d])\left\{H^2[a,d]^{-2}+1\right\}.\]
Since $z=H^{1/240}$ and $d\le z^{1/8}$ we have $1\ll H^2[a,d]^{-2}$. Moreover 
Lemma \ref{fests} shows that $f_C([a,d])\ll z^{1/4}$, say, while
\[\phi([a,d])[a,d]^{-2}\le [a,d]^{-1}\le a^{-1}\le a^{-1/4}z^{-3/4}.\]
Thus
\[T_3\ll z^{1/2}H^2\sum_{\substack{z<a\le z^2\\ P^+(a)\le\log z}}a^{-1/4}z^{-3/4}
\ll z^{-1/4}H^2\sum_{\substack{a=1\\ P^+(a)\le\log z}}^{\infty}a^{-1/4}.\]
The infinite sum is
\begin{eqnarray*}
\prod_{p\le\log z}\left\{\sum_{f=0}^\infty p^{-f/4}\right\}&=&
\exp\left\{\sum_{p\le\log z} O(p^{-1/4})\right\}\\
&\ll&\exp\{O((\log z)^{3/4})\}\\
&\ll& z^{1/8},
\end{eqnarray*}
say, whence $T_3\ll H^2 z^{-1/8}$.  This is satisfactory for Proposition 
\ref{kAL}, since we have $d\le z^{1/8}$ and $\rho\ge 2$. \qed
\bigskip

It remains to prove Lemma \ref{LS}, which we restate here for 
convenience. 

\noindent {\bf Lemma 4.8.(again).} 
{\em Define an arithmetic function by setting} 
\[g(k)=k^{-2}\card\{\z(\mathrm{mod}\; k):\,k\mid C(\z)\}.\] 
{\em Then for any} $\tau\le H^{1/120}$ {\em we have} 
\[U(b;\tau,\y_0)\ll 
H^2b^{-2}\prod_{\substack{p<\tau\\ p\nmid b\mathrm{Disc}(C)}}(1-g(p)) 
  +H^{7/4}.\] 

\begin{proof} For the proof 
we apply a sieve upper bound, sifting out primes
$p<\tau$ which do not divide $b\mathrm{Disc}(C)$.
There are many possible
bounds that we could apply, and we choose Corollary 6.2 of Iwaniec
and Kowalski \cite{IK}. We 
first note that $g(d)$ is a multiplicative function,
with $0\le g(d)\le 1$, and that $g(p)\le 3/p<1$ when
$p\nmid b\mathrm{Disc}(C)$.  (Here we recall that we had artificially arranged 
at the outset that $\mathrm{Disc}(C)$ is divisible by 6.)  Thus
\[\prod_{w\le p<z}(1-g(p))^{-1}\le \prod_{w\le p<z}(1-3/p)^{-1}\ll
\left(\frac{\log z}{\log w}\right)^3\]
when $w\ge 4$, whence condition (6.34) of Iwaniec and Kowalski
\cite{IK} holds, with $\kappa=3$ and with $K$ depending only on $Q_0$
and $Q_1$.  We will apply the sieve bound to $C(\y_0+b\z)$, where
$\z$ runs over a square $S$ of side-length $H/b$.  When $d$ is coprime
to $b\mathrm{Disc}(C)$ we have
\[g(d)=d^{-2}\card\{\z(\mathrm{mod}\; d):\,d\mid C(\y_0+b\z)\},\]  
so that
\[\card\{\z\in\Z^2\cap S: d\mid C(\y_0+b\z)\}=
g(d)d^2\left(\frac{H}{bd}+O(1)\right)^2=
g(d)\frac{H^2}{b^2}+r_d,\]
say, with $r_d\ll H^{3/2}$ when $d\le H^{1/2}$.  Then, taking $s=30$
in \cite[Corollary 6.2]{IK} we find that
\[U(b;\tau,\y_0)\ll H^2b^{-2}\prod_{\substack{p<\tau\\ p\nmid
    b\mathrm{Disc}(C)}}(1-g(p))+\sum_{d<\tau^{30}}|r_d|.\] 
If $\tau\le H^{1/120}$ the remainder sum is $O(H^{7/4})$, and Lemma \ref{LS}
follows.
\end{proof}

\section{Asymptotics via lattices}

In this section we will refine the methods of \S \ref{CVL}, so
as to produce an asymptotic formula for $S(X,Y)$ from \eqref{def:S(X,Y)}.

\begin{proposition}\label{P1*}
There is a constant $A=A(Q_0,Q_1)>0$ such that
  \[S(X,Y)=\frac{2\mathfrak{S}_S}{(\rho-2)!}
  \int_1^{\infty}(\log  t)^{\rho-2}t^{-2}I(t)\d t+O(XY\LL^{\rho-3})\]
  for $X\LL^A\le Y\le X^2\LL^{-A}$.
Here $\mathfrak{S}_S$ and $I(t)$ are given by Lemma \ref{lem:f2} and
(\ref{Adef}) respectively. 
\end{proposition}

For the proof we will first estimate various error terms, and then go
on to show that 
the main term can be put in the shape given above. 

\subsection{Proposition \ref{P1*} --- Error terms}
As in \S \ref{S6}, when $X\gg 1$ we have $S(X,Y)=2S_1(X,Y)$, with
$S_1(X,Y)$ given by (\ref{S1def}). We sort our points $\x$ according
to the value $d=\gcd(Q_0(\x),Q_1(\x))$, so that
\[S_1(X,Y)=\sum_{d=1}^\infty S_1(X,Y;d),\]
with
\[S_1(X,Y;d)=\sum_{\substack{\x\in\Z^3_{\mathrm{prim}}\\ \gcd(Q_0(\x),Q_1(\x))=d}}
W_3(X^{-1}\x)W_2\left(Y^{-1}d^{-1}(-Q_1(\x),Q_0(\x))\right).\]
The sum vanishes unless $d\ll X^2Y^{-1}$, and we proceed to show
that only values of $d$ close to $X^2Y^{-1}$ can make a
substantial contribution.
\begin{lemma}\label{L1}
  We have
\[  \sum_{d\le X^2Y^{-1}\LL^{-2}}S_1(X,Y;d)\ll XY\LL^{\rho-3}\]
for any $X\le Y$.
\end{lemma}
\begin{proof}
We have $S_1(X,Y;d)\le S(d)$, with $S(d)$ given by (\ref{23a}). Then, using (\ref{Sm})
and (\ref{use}) we have
  \begin{eqnarray}\label{L1s}
\sum_{d\le X^2Y^{-1}\LL^{-2}}S(d)&=&\sum_{d\le X^2Y^{-1}\LL^{-2}}\;\;
\sum_{a \in \mathcal{M}(\Z/d\Z)}S(\sL(a,d);d)\nonumber\\
&\ll &\sum_{d\le X^2Y^{-1}\LL^{-2}}f_M(d)Y^{3/2}d^{-1/2}.
  \end{eqnarray}
  However Lemma \ref{lem:f2} shows that
  \[\sum_{d\le x}f_M(d)d^{-1/2}\ll x^{1/2}(\log x)^{\rho-2},\]
so that the sum (\ref{L1s}) is $O(XY\LL^{\rho-3})$,  
  as required.
\end{proof}

When $d$ is in the remaining range
$X^2Y^{-1}\LL^{-2}< d\ll X^2Y^{-1}$ we write
\begin{eqnarray*}
  \lefteqn{S_1(X,Y;d)}\\
  &=&
  \sum_{\substack{\x\in\Z^3_{\mathrm{prim}}\\ \gcd(Q_0(\x),Q_1(\x))=d}}
  W_3(X^{-1}\x)W_2\left(Y^{-1}d^{-1}(-Q_1(\x),Q_0(\x))\right)\\
  &=& \sum_{h=1}^{\infty}\mu(h)S(d,dh),
  \end{eqnarray*}
where
\[S(d,dh)=\sum_{\substack{\x\in\Z^3_{\mathrm{prim}}\\ dh|\gcd(Q_0(\x),Q_1(\x)}}
W_3(X^{-1}\x)W_2\left(Y^{-1}d^{-1}(-Q_1(\x),Q_0(\x))\right),\]
and we plan to show that only relatively small integers $h$ can
make a substantial contribution. 
\begin{lemma}\label{L2}
  We have
  \[\sum_{X^2Y^{-1}\LL^{-2}< d\ll X^2Y^{-1}}\;\;\sum_{h\ge \LL^{10}}
  S(d,dh)\ll XY\LL^{\rho-3}\]
  for $1\ll X\le Y\LL^{-9}$.
\end{lemma}
\begin{proof}
  Since $X\gg 1$ the quadratic forms $Q_0(\x)$ and $Q_1(\x)$ cannot
  both vanish for any relevant $\x$.  We may therefore suppose
  that $dh\ll X^2$ whenever $dh|\gcd(Q_0(\x),Q_1(\x))$. By the same argument 
  that produced (\ref{Sm}) we have
  \beql{25a}
  S(d,dh)=\sum_{a \in \mathcal{M}(\Z/dh\Z)}S(\sL(a,dh);d).
  \eeq
  Lemmas \ref{lem:lattices} and \ref{rl0} show that
  \[S(\sL(a,dh);d)\ll 1+X^2/(dh)+X^3/(dh)^2, \]
  whence
 \begin{eqnarray*}
 \lefteqn{\sum_{X^2Y^{-1}\LL^{-2}< d\ll X^2Y^{-1}}\;\;\sum_{h\ge \LL^{10}}
  S(d,dh)}\hspace{2cm}\\
  &\ll& \sum_{X^2Y^{-1}\LL^{-2}< d\ll X^2Y^{-1}}\;\;
  \sum_{\substack{h\ge \LL^{10}\\ dh\ll X^2}} f_M(dh)\left\{1+
  \frac{X^2}{dh}+\frac{X^3}{(dh)^2}\right\}\\
  &\ll& \sum_{X^2Y^{-1}\LL^{-2}< d\ll X^2Y^{-1}}\;\;
  \sum_{\substack{h\ge \LL^{10}\\ dh\ll X^2}} \tau_4(dh)\left\{
  \frac{X^2}{dh}+\frac{X^3}{(dh)^2}\right\},
  \end{eqnarray*}
  on using Lemma \ref{fests}.
The term $X^2/dh$ contributes
\[\ll X^2\sum_{n\ll X^2}\tau(n)\tau_4(n)/n\ll X^2\LL^8\ll XY\LL^{-1}\ll XY\LL^{\rho-3}\]
by Lemma \ref{lem:tau_k}, on using the inequality
$\tau(n)\tau_4(n) \leq \tau_8(n)$,
while the term $X^3/(dh)^2$ contributes
\begin{eqnarray*}
&\ll& X^3\sum_{n> X^2Y^{-1}\LL^8}\tau(n)\tau_4(n)/n^2\\
&\ll& X^3\{X^2Y^{-1}\LL^8\}^{-1}\LL^7\\
&\ll& XY\LL^{-1},
\end{eqnarray*}
via a similar application of Lemma \ref{lem:tau_k}.  This is enough, since
$\rho\ge 2$.
\end{proof}

Lemmas \ref{L1} and \ref{L2} now show that
\[S_1(X,Y)=\sum_{\substack{X^2Y^{-1}\LL^{-2}< d\ll X^2Y^{-1}\\
    h\le\LL^{10}}}\mu(h)S(d,dh)+O(XY\LL^{\rho-3})\]
provided that $X\le Y\LL^{-9}$. In view of (\ref{25a}) we
now seek an asymptotic formula for
$S(\sL(a,m);d)$. We begin by writing
\[S(\sL(a,m);d)=\sum_{r=1}^{\infty}\mu(r)
\sum_{\substack{\x\in\sL(a,m)\\ r\mid\x}}
W_3(X^{-1}\x)W_2\left(Y^{-1}d^{-1}(-Q_1(\x),Q_0(\x))\right).\]
Terms with $\x=\mathbf{0}$ or $||\x||_\infty>\tfrac52 X$ make no contribution,
so we can restrict the $r$-sum to the range $r\ll X$.
The conditions $\x\in\sL(a,m)$ and $r|\x$ define a 3-dimensional
lattice which we denote by $\sL_r(a,m)$. The inner sum is 
$S_0(\sL_r(a,m);d)$ in the notation (\ref{S0def}), so that
\begin{eqnarray*}
S_1(X,Y;d)&=&\sum_{\substack{X^2Y^{-1}\LL^{-2}< d\ll X^2Y^{-1} \\
    h\le\LL^{10}}}\mu(h)\;\;\sum_{a \in \mathcal{M}(\Z/dh\Z)}
\sum_{r\ll X}\mu(r)S_0(\sL_r(a,dh);d)\\
&&\hspace{3cm}{}+O(XY\LL^{\rho-3})
\end{eqnarray*}
for the same range of $X$ as before. We proceed to show that large values
of $r$ make a small overall contribution.

\begin{lemma}\label{L3}
We have
  \[\sum_{\substack{X^2Y^{-1}\LL^{-2}< d\ll X^2Y^{-1} \\ h\le\LL^{10}}}\;\;
  \sum_{a \in \mathcal{M}(\Z/dh\Z)}\;\sum_{\LL^{20}\le r\ll X}S_0(\sL_r(a,dh);d)\ll
  XY\LL^{\rho-3}\]
    when $X\le Y\LL^{-10}$.
\end{lemma}
\begin{proof}
We begin the proof by noting that if $\lambda_j$ are the successive minima of
$\sL_r(a,dh)$ then
\begin{eqnarray*}
  S_0(\sL_r(a,dh);d)&\ll& \prod_{j=1}^3(1+X/\lambda_j)\\
  &\ll&
1+X/\lambda_1+X^2/\lambda_1\lambda_2+X^3/\lambda_1\lambda_2\lambda_3\\
&\ll &1+X/\lambda_1+X^2\lambda_3/\det(\sL_r(a,dh))+X^3/\det(\sL_r(a,dh)).
\end{eqnarray*}
We now require bounds for $\lambda_1$ and $\lambda_3$.
Since $\sL(a,dh)$ arises from the condition
$\x\equiv\lambda\mathbf{a}$(mod $dh)$ with $\gcd(\mathbf{a},dh)=1$, it
is clear that the successive minima of $\sL_r(a,dh)$ will have
$\lambda_1\ge r$ and $\lambda_3\le\lcm(dh,r)$. Moreover we will have
$\det(\sL_r(a,dh))=r\lcm(dh,r)^2$. Since $\LL^{20}\le r\ll X$ it follows that
\begin{eqnarray*}
S_0(\sL_r(a,dh))&\ll& \frac{X}{r}+\frac{X^2}{r\lcm(dh,r)}+\frac{X^3}{r\lcm(dh,r)^2}\\
&\ll& \frac{X}{r}+\frac{X^2}{rdh}+\frac{X^3}{\LL^{20}dh\lcm(dh,r)}.
\end{eqnarray*}
On summing over the range $\LL^{20}\le r\ll X$ the first and second 
terms produce $O(X\LL)+O(X^2 \LL/(dh))$.
Moreover
\[\sum_{r\ll X}\frac{1}{\lcm(dh,r)}\le\sum_{k\mid dh}
\sum_{\substack{r\ll X\\k\mid r}}
\frac{k}{rdh}=\sum_{k\mid dh}\sum_{s\ll X/k}\frac{1}{sdh}
\ll (dh)^{-1}\tau(dh)\LL.\]
Using Lemma \ref{lem:tau_k} and the bound $f_M(dh)\le\tau_4(dh)$ from
Lemma \ref{fests} we deduce that
\begin{eqnarray*}
 \lefteqn{\sum_{\substack{X^2Y^{-1}\LL^{-2}< d\ll X^2Y^{-1} \\ h\le\LL^{10}}}
\sum_{a \in \mathcal{M}(\Z/dh\Z)}\;\;\sum_{r\ge \LL^{20}}S_0(\sL_r(a,dh);d)}\\
  &\ll&
  \sum_{X^2Y^{-1}\LL^{-2}< m\ll X^2Y^{-1}\LL^{10}}\tau(m)\tau_4(m)
 \left \{X+\frac{X^2}{m}+\frac{X^3\tau(m)}{\LL^{20} m^2}\right\}\LL\\
 &  \ll& X^3Y^{-1}\LL^{18}+X^2\LL^9+XY\LL^{-1}.
\end{eqnarray*}
The lemma then follows.
\end{proof}

Lemma \ref{L3} allows us to conclude that
\begin{eqnarray}\label{xx}
S_1(X,Y)&=&
\sum_{\substack{X^2Y^{-1}\LL^{-2}< d\ll X^2Y^{-1} \\ h\le\LL^{10}}}
\mu(h)\sum_{a \in \mathcal{M}(\Z/dh\Z)}\;\sum_{r\le\LL^{20}}\mu(r)S_0(\sL_r(a,dh);d)
\nonumber\\
&&\hspace{3cm}{}+O(XY\LL^{\rho-3})
\end{eqnarray}
for $X\le Y\LL^{-10}$. We next use Lemma \ref{rl1} to derive the
following lemma.

\begin{lemma}\label{rl2}
  We have
  \[S_1(X,Y)=
  \sum_{\substack{X^2Y^{-1}\LL^{-2}< d\le X^2Y^{-1} \\ h\le\LL^{10}, \, r\le\LL^{20}}}
\mu(h)\mu(r)\frac{I(d)f_M(dh)}{r\lcm(dh,r)^2}+O(XY\LL^{\rho-3})\]
provided that $X\le Y\LL^{-31}$.
\end{lemma}
\begin{proof}
Lemma \ref{rl1} produces 
\[S_0(\sL_r(a,dh);d)=\frac{I(d)}{r\lcm(dh,r)^2}
+O_k\left(\frac{(Yd)^{3/2}}{r\lcm(dh,r)^2}\left(
\frac{X\lcm(dh,r)}{Yd}\right)^k\right),\]
and we immediately obtain the main term in Lemma \ref{rl2}.
Since $k\ge 4$ and $\lcm(dh,r)\le dhr$ the error term above is 
\[\ll_k Y^{3/2}d^{-1/2}h^{-2}r^{-3}\left(XY^{-1}hr\right)^k.\]
However our assumptions on $X,Y,h$ and $r$ yield
$XY^{-1}hr\le\LL^{-1}$, whence
\[S_0(\sL_r(a,dh);d)=\frac{I(d)}{r\lcm(dh,r)^2}
+O_k\left(Y^{3/2}d^{-1/2}h^{-2}r^{-3}\LL^{-k}\right).\]
Since $f_M(dh)\ll\tau_4(dh)\le\tau_4(d)\tau_4(h)$ by Lemma \ref{fests},
we see from Lemma~\ref{lem:tau_k} that the 
contribution of the error term to (\ref{xx}) is
\[\ll_k \sum_{X^2Y^{-1}\LL^{-2}< d\ll X^2Y^{-1}}
  Y^{3/2}d^{-1/2}\LL^{-k}\tau_4(d)  \ll_k XY\LL^{3-k}.\]
This is $O(XY\LL^{-1})$ if we choose $k=4$, and the lemma follows.
\end{proof}
\bigskip

We must now estimate the error incurred on extending the
summations over $r,h$ and $d$ to run over all positive integers.  We 
begin this process by dealing with $h$ and $r$.
\begin{lemma}\label{hr}
  We have
\[S_1(X,Y)=\sum_{X^2Y^{-1}\LL^{-2}< d\ll X^2Y^{-1} }\;\;\sum_{h,r=1}^\infty
\mu(h)\mu(r)\frac{I(d)f_M(dh)}{r\lcm(dh,r)^2}+O(XY\LL^{\rho-3})\]
provided that $X\le Y\LL^{-31}$.
\end{lemma}
\begin{proof}
According to Lemma \ref{rl1} we have
$I(d)\ll (Yd)^{3/2}$. Moreover we will have
$f_M(dh)\ll\tau_4(dh)\le\tau_4(d)\tau_4(h)$ as before. Thus extending the
$h$-summation gives an error
\[\ll Y^{3/2}\sum_{\substack{X^2Y^{-1}\LL^{-2}\le d\ll X^2Y^{-1} \\ r\le\LL^{20}}}
\frac{\tau_4(d)d^{3/2}}{r}\sum_{h>\LL^{10}}\frac{\tau_4(h)}{\lcm(dh,r)^2}.\]
If we use the lower bound $\lcm(dh,r)\ge dh$ we find that
\[\sum_{h>\LL^{10}}\frac{\tau_4(h)}{\lcm(dh,r)^2}\ll d^{-2}\LL^{-9},\]
by Lemma \ref{lem:tau_k}.  We have
\[\sum_{r\le\LL^{20}}r^{-1}\ll\LL,\]
and 
\[\sum_{d\ll X^2Y^{-1}}\tau_4(d)d^{-1/2}\ll (X^2Y^{-1})^{1/2}\LL^3,\]
so that the overall error is $O(XY\LL^{-5})$. Similarly, if we now
extend the $r$-summation, the error will be
\[\ll Y^{3/2}\sum_{X^2Y^{-1}\LL^{-2}< d\ll X^2Y^{-1}}
\tau_4(d)d^{3/2}\sum_{h=1}^{\infty}\tau_4(h)
\sum_{r>R}\frac{1}{r\lcm(dh,r)^2},\]
where, for typographic convenience we have temporarily set $R=\LL^{20}$.  If we write
$\gcd(dh,r)=k$ we see that 
\begin{eqnarray*}
\sum_{r>R}\frac{1}{r\lcm(dh,r)^2}&\le&\sum_{k\mid dh}k^2d^{-2}h^{-2}
\sum_{\substack{r>R\\ k\mid r}}^{\infty}r^{-3}\\
&\ll&\sum_{k\mid dh}k^2d^{-2}h^{-2}\min(k^{-3},k^{-1}R^{-2})\\
&\ll&\sum_{k\mid dh}k^2d^{-2}h^{-2}k^{-2}R^{-1}\\
&=&d^{-2}h^{-2}\tau(dh)R^{-1}.
\end{eqnarray*}
The overall error resulting from
extending the $r$-summation is therefore
\[\ll Y^{3/2}\sum_{X^2Y^{-1}\LL^{-2}< d\ll X^2Y^{-1}}
\tau_4(d)d^{-1/2}\sum_{h=1}^{\infty}\tau_4(h)h^{-2}\tau(dh)R^{-1}.\]
Since $\tau(dh)\le\tau(d)\tau(h)$ this is
\begin{eqnarray*}
&\ll& Y^{3/2}R^{-1}\sum_{X^2Y^{-1}\LL^{-2}\le d\ll X^2Y^{-1}}
\tau_4(d)\tau(d)d^{-1/2}\\
&\ll& Y^{3/2}R^{-1}(X^2Y^{-1})^{1/2}\LL^7
\end{eqnarray*}
by a further application of Lemma \ref{lem:tau_k}.
On recalling that $R=\LL^{20}$ we now see that this is
$O(XY\LL^{-13})$, which is satisfactory for the lemma.  
\end{proof}

\subsection{Proposition \ref{P1*} --- Evaluating the main term}

In this section we complete the proof of Proposition \ref{P1*}. 
Continuing from Lemma \ref{hr}, the triple sum on the right takes the form
\begin{equation}\label{f0A}
  \sum_{X^2Y^{-1}\LL^{-2}< d\ll X^2Y^{-1}}f_0(d)I(d),
  \end{equation}
with
\[f_0(d)=\sum_{h,r=1}^\infty\mu(h)\mu(r)f_M(dh)r^{-1}\lcm(dh,r)^{-2}.\]
It will be convenient to write the range of summation above as 
$D_0< d\le D_1$ with $D_0=X^2Y^{-1}\LL^{-2}$ and
$D_1$ of order $X^2Y^{-1}$, chosen so that
$I(d)$ vanishes for $d> D_1$. 

We proceed to split $h$ and $r$ as $h=h_1h_2$ and $r=r_1r_2$ with
$\gcd(h_1r_1,d)=1$ and $h_2r_2\mid d^{\infty}$. Since only square-free
values of $h$ and $r$ can make non-zero contributions this
latter condition can be written as the requirement that $h_2\mid d$
and $r_2\mid d$.  The sum for $f_0(d)$ now factors as $f_1(d)f_2(d)$, with
\[f_1(d)=\sum_{\substack{h_1,r_1=1\\ \gcd(h_1r_1,d)=1}}^\infty\mu(h_1)\mu(r_1)
  f_M(h_1)r_1^{-1}\lcm(h_1,r_1)^{-2}\]
and
  \[f_2(d)=\sum_{h_2\mid d}\sum_{r_2\mid d}\mu(h_2)\mu(r_2)f_M(dh_2)r_2^{-1}
\lcm(dh_2,r_2)^{-2} .\]
The sum $f_1$ is a product of local factors $f_3(p)$, say, for primes
$p\nmid d$, with
\[f_3(p)=1-f_M(p)p^{-2}-p^{-3}+f_M(p)p^{-3}=
(1-p^{-1})(1+p^{-1}+p^{-2}-f_M(p)p^{-2}).\]
Looking back to the definition of $f_M$ we note that $f_3(p)$ is
strictly positive for all primes $p$ outside a certain finite set. We
will write $P$ for the product of the primes in this bad set, so that
$f_3(p)>0$ for $p\nmid P$.
We also see that $f_2$ is a multiplicative function given by
\begin{eqnarray*}
  f_2(p^e)
&=&f_M(p^e)p^{-2e}-f_M(p^{e+1})p^{-2e-2}-f_M(p^e)p^{-2e-1}+f_M(p^{e+1})p^{-2e-3}\\
  &=&p^{-2e}(1-p^{-1})(f_M(p^e)-f_M(p^{e+1})p^{-2})
  \end{eqnarray*}
for $e\ge 1$. We therefore conclude that $f_0(d)$ vanishes unless $P|d$, 
in which case $f_0(d)=\kappa_0 f_4(d)d^{-2}$,
where
\[\kappa_0=\prod_{p\nmid P}f_3(p)\]
and $f_4$ is a multiplicative function defined by taking
\[f_4(p^e)=\frac{p^{2e}f_2(p^e)}{f_3(p)}=
\frac{f_M(p^e)-f_M(p^{e+1})p^{-2}}{1+p^{-1}+p^{-2}-f_M(p)p^{-2}},\quad
p\nmid P,\, e\ge 1\]
and
\[f_4(p^e)=p^{2e}f_2(p^e)=(1-p^{-1})(f_M(p^e)-f_M(p^{e+1})p^{-2}),\quad p\mid P,\, e\ge 1.\]
We proceed to investigate the asymptotic behaviour of the sum
\[\Sigma(D)=\sum_{\substack{d\le D\\ P|d}}f_4(d)\]
as $D\to\infty$, using the Perron formula. We begin by examining the function
\[F_4(s)=\sum_{\substack{d=1\\ P|d}}^{\infty}f_4(d)d^{-s}=
\prod_{p|P}\left\{\sum_{e=1}^\infty f_4(p^e)p^{-s}\right\}
\prod_{p\nmid P}\left\{\sum_{e=0}^\infty f_4(p^e)p^{-s}\right\}.\]
The Euler factors of this take the approximate shape $1+f_M(p)p^{-s}+\ldots$ for large 
primes $p$. In the light of Lemma \ref{lem:f1}, we therefore compare $F_4(s)$ with
$\prod_{i=1}^{\rho-1}\zeta_{K_i}(s)$. We have
\[F_4(s)=\prod_{i=1}^{\rho-1}\zeta_{K_i}(s)\prod_p G_p(s),\]
where
\[G_p(s)=1+O(p^{-\sigma-1})+O(p^{-2\sigma}),\;\;\;(p\nmid P,\;\;\sigma>0).\]
When $p\mid P$ the corresponding factor $G_p(s)$ will be holomorphic and 
bounded for $\sigma>0$. It follows that
we may write $F_4(s)=G_1(s)\prod_{i=1}^{\rho-1}\zeta_{K_i}(s)$, with a function $G_1(s)$ 
which is holomorphic and bounded for $\sigma\ge\tfrac34$. 
The standard analysis based on Perron's formula now shows that
\begin{equation}\label{asform}
  \Sigma(D)=\mathrm{Res}\left(\frac{F_4(s)D^s}{s};\, s=1\right)
  +O(D^{1-\delta})
  \end{equation}
with a constant $\delta>0$ depending only on $Q_0$ and $Q_1$.
The residue will take the shape $DP(\log D)$ for some polynomial $P$
of degree at most $\rho-2$, and the coefficient of
$(\log D)^{\rho-2}$ will be
\[\lim_{s\to 1}\frac{(s-1)^{\rho-1}F_4(s)}{(\rho-2)!}=
\lim_{s\to 1}\frac{\zeta(s)^{1-\rho}F_4(s)}{(\rho-2)!}=\frac{\kappa_1}{(\rho-2)!},\]
say.  Thus (\ref{asform}) has leading term 
$\kappa_1 D(\log D)^{\rho-2}/(\rho-2)!$.
By a similar argument to that used in the proof of Lemma \ref{earlier}
we find that
\[\kappa_1=\prod_{p|P}\left(1-\frac{1}{p}\right)^{\rho-1}
\left\{\sum_{e=1}^{\infty}f_4(p^e)p^{-e}\right\}
\prod_{p\nmid P}\left(1-\frac{1}{p}\right)^{\rho-1}
\left\{1+\sum_{e=1}^{\infty}f_4(p^e)p^{-e}\right\},\]
where the final product is only conditionally convergent. 

We are now ready to estimate the sum (\ref{f0A}), using partial summation.
We have
\begin{eqnarray*}
\sum_{D_0<d\le D_1}f_0(d)I(d)
&=&\kappa_0\sum_{\substack{D_0< d\le D_1\\ P|d}}f_4(d)d^{-2}I(d)\\
&=&\kappa_0
\left(\left[\Sigma(t)t^{-2}I(t)\right]_{D_0}^{D_1}-
\int_{D_0}^{D_1} \Sigma(t)\frac{\d}{\d t}\{t^{-2}I(t)\}\d t\right).
\end{eqnarray*}
The asymptotic formula (\ref{asform}) produces a main term
\[T=\kappa_0
\int_{D_0}^{D_1} \frac{\d}{\d t}\{tP(\log t)\}t^{-2}I(t)\d t\]
and an error term
\[E\ll  D_1^{-1-\delta}I(D_1)+D_0^{-1-\delta}I(D_0)+
\int_{D_0}^{D_1}t^{1-\delta}\left|\frac{\d}{\d t}\{t^{-2}I(t)\}\right|\d t.\]

To estimate the error term $E$ we first note that $I(D_1)=0$ and that
\[I(D_0)D_0^{-1-\delta}\ll (YD_0)^{3/2}D_0^{-1-\delta}=XY\LL^{-1}D_0^{-\delta},\]
which is satisfactory for Proposition \ref{P1*}. We also observe that
\[\frac{\d}{\d t}\{t^{-2}I(t)\}\ll \left(t^{-3}+X^2Y^{-1}t^{-4}\right)(Yt)^{3/2},\]
by the same argument that produced (\ref{pdc}).  Hence
\[\int_{D_0}^{D_1}t^{1-\delta}\left|\frac{\d}{\d t}\{t^{-2}I(t)\}\right|\d t
\ll XY\LL D_0^{-\delta}.\]
This too is satisfactory for Proposition \ref{P1*}, provided that
$Y\le X^2\LL^{-A}$ with a suitably large constant $A$.
For the main term $T$ we observe that
\[\int_{D_0}^{D_1}(\log t)^kt^{-2}I(t)\d t\ll XY\LL^k,\]
for any fixed $k\ge 0$, on using the bound (\ref{pdc}).
We may use this to estimate the contribution
from powers of $\log t$ of degree strictly less than $\rho-2$. Moreover, 
\[\int_1^{D_0}(\log t)^{\rho-2}t^{-2}I(t)\d t\ll 
Y^{3/2}D_0^{1/2}\LL^{\rho-2}\ll XY\LL^{\rho-3},\]
whence
\[T=\frac{\kappa_0\kappa_1}{(\rho-2)!}\int_1^{\infty}(\log t)^{\rho-2}t^{-2}I(t)\d t
+O(XY\LL^{\rho-3}).\]
To complete the proof of Proposition \ref{P1*} it now suffices to observe that
\[  \kappa_0\kappa_1=\prod_p\varpi_p\left(1-\frac{1}{p}\right)^{\rho}=\prod_p\tau_p
=\mathfrak{S}_S\]
with
\begin{eqnarray*}
\varpi_p&=&1+p^{-1}+p^{-2}-f_M(p)p^{-2}+\sum_{e=1}^{\infty}
\{f_M(p^e)-f_M(p^{e+1})p^{-2}\}p^{-e}\\\
&=&1+p^{-1}+p^{-2}+(1-p^{-1})\sum_{e=1}^{\infty}f_M(p^e)p^{-e},
\end{eqnarray*}
as in Lemma \ref{tl}.
\qed

\subsection{Proof of Theorem \ref{extra}}\label{pt13}

In this section we will complete the proof of Theorem
\ref{extra}. Before doing this we pause to take stock of what has
been achieved so far, and what remains to be done. 
We have the universal upper bound
\[S(X,Y)\ll XY\LL^{\rho-2}\]
for sufficiently large $X,Y$, by (\ref{sxy}). The implied constant
here depends on the weights $W_2$ and $W_3$, and on the quadratic
forms $Q_0$ and $Q_1$. Moreover, Proposition \ref{P1*}
tells us that
\beql{P1a}
S(X,Y)=2\frac{\mathfrak{S}_S}{(\rho-2)!}
\int_1^{\infty}(\log t)^{\rho-2}t^{-2}I(t)\d t+O(XY\LL^{\rho-3})
\eeq
for $X\LL^A\le Y\le X^2\LL^{-A}$, where
\[I(t)=\int_{\R^3}W_3\left(X^{-1}\u\right)
W_2\left(Y^{-1}t^{-1}\big(- Q_1(\u),Q_0(\u)\big)\right)\d u_0\d u_1\d u_2.\]
In addition we should recall that
$S(X,Y)$ vanishes when $Y\gg X^2$, as noted in Proposition \ref{P1}.

It is now apparent that further investigation is required to make the real
density comparable to $\tau_\infty$, given by
the expressions (\ref{tinf}). We will also need to
handle the ranges $X^2\LL^{-A}\le Y\ll X^2$ and $X\ll Y\le
X\LL^A$ which are not covered by our asymptotic formula. Our upper
bound estimate will suffice for this. Moreover it will be 
necessary to count points satisfying
$||\x||_\infty ||\y||_\infty\le B$ using the weighted sums
$S(X,Y)$. To handle the weights $W_2$ and $W_3$ 
we will ultimately use an ``$(\ep,\delta)$-argument'', 
of a kind familiar from elementary real analysis. 

Our first step is to remove the weight $W_2$ from $S(X,Y)$ in \eqref{P1a}.

\begin{lemma}\label{rw2}
Let $\eta\in(0,\tfrac12)$ be given, and write
\[S^{(0)}(X,Y)=\sum_{\substack{\x\in\Z^3_{\mathrm{prim}}\\ (Q_0(\x),Q_1(\x))\not=(0,0)}}
\;\sum_{\substack{\y\in\Z^2_{\mathrm{prim}}\\  y_0Q_0(\x)+y_1Q_1(\x)=0\\
Y< ||\y||_\infty\le (1+\eta)Y}}
W_3(X^{-1}\x).\]
Then if $X\LL^A\le Y\le X^2\LL^{-A}$ we have
\[S^{(0)}(X,Y)=2\frac{\mathfrak{S}_S}{(\rho-2)!}I^{(0)}(X,Y)
+O_{\eta}(XY\LL^{\rho-3})+O(\eta^2 XY\LL^{\rho-2}),\]
where
\[I^{(0)}(X,Y)=\int_1^{\infty}\frac{(\log t)^{\rho-2}}{t^2}
\int_{Yt\le h_Q(\u)\le (1+\eta)Yt}W_3(X^{-1}\u)\d u_0\d u_1\d u_2\d t.\]
\end{lemma}
We should remark here that the implied constant in the first error
term above may depend on the weight $W_3$, in addition to $\eta$, and
similarly that the second error term has an implied constant that may
depend on $W_3$.  We remind the reader that $h_Q(\u)$ is shorthand 
for $||(Q_0(\u),Q_1(\u))||_\infty$, as in \S \ref{trd}.
\begin{proof}
Given our small positive $\eta$ we fix an infinitely
differentiable even weight
function $W_2^{(-)}(\u)$ supported on the set
$1\le||\u||_\infty\le 1+\eta$,
and taking values in the range $[0,1]$, with the further property that
$W_2^{(-)}(\u)=1$ whenever
\[1+\eta^2\le||\u||_\infty\le 1+\eta-\eta^2.\]
We note in particular that $W_2^{(-)}$ will be supported on the set
$\tfrac12\le||\u||_\infty\le \tfrac52$, as required.
The weight $W_2^{(-)}(\u)$ can be thought of as  a
good approximation to the
characteristic function for the set
$1\le||\u||_\infty\le 1+\eta$; and indeed it will be a minorant
for this characteristic function. The
choice of function will depend on $\eta$, so that the error term in
(\ref{P1a}), which previously had an implied constant depending on
$W_2$ and $W_3$, now has an error term depending on $\eta$ and $W_3$.
In an analogous way we can construct a corresponding majorant
$W_2^{(+)}$ for the characteristic function of the set
$1\le||\u||_\infty\le 1+\eta$, supported this time on
the set
\[1-\eta^2\le||\u||_\infty\le 1+\eta+\eta^2.\]

The formula (\ref{P1a}) now becomes
\[S^{(\pm)}(X,Y)=2\frac{\mathfrak{S}_S}{(\rho-2)!}V^{(\pm)}
+O_{\eta}(XY\LL^{\rho-3}),\]
where
\[S^{(\pm)}(X,Y)=\sum_{\substack{\x\in\Z^3_{\mathrm{prim}}\\ (Q_0(\x),Q_1(\x))\not=(0,0)}}
\;\;\sum_{\substack{\y\in\Z^2_{\mathrm{prim}}\\  y_0Q_0(\x)+y_1Q_1(\x)=0}}
W_3(X^{-1}\x)W_2^{(\pm)}(Y^{-1}\y)\]
and
\[V^{(\pm)}=\int_1^{\infty}(\log t)^{\rho-2}t^{-2}I^{(\pm)}(t)\d t,\]
with
\[I^{(\pm)}(t)=
\int_{\R^3}W_3(X^{-1}\u)
W_2^{(\pm)}\left(Y^{-1}t^{-1}(-Q_1(\u),Q_0(\u))\right)\d u_0\d u_1\d u_2.\]
In view of our choices for $W_2^{(\pm)}$ we have
\[S^{(-)}(X,Y)\le\sum_{\substack{\x\in\Z^3_{\mathrm{prim}}\\ (Q_0(\x),Q_1(\x))\not=(0,0)}}
\;\;\sum_{\substack{\y\in\Z^2_{\mathrm{prim}}\\  y_0Q_0(\x)+y_1Q_1(\x)=0\\
Y< ||\y||_\infty\le (1+\eta)Y}}
W_3(X^{-1}\x)\le S^{(+)}(X,Y).\]
Similarly the integral
\[I^{(0)}(X,Y)=\int_1^{\infty}(\log t)^{\rho-2}t^{-2}
\int_{Yt\le h_Q(\u)\le (1+\eta)Yt}W_3(X^{-1}\u)
\d u_0\d u_1\d u_2\d t\]
lies between $V^{(-)}$ and $V^{(+)}$, whence 
\[V^{(+)}=I^{(0)}(X,Y)+O(V^{(+)}-V^{(-)}),\]
and similarly for $V^{(-)}$.  However
\[I^{(+)}-I^{(-)}\ll
\int_1^{\infty}(\log t)^{\rho-2}t^{-2}\left(m_1(t)+m_2(t)\right)\d t,\]
where 
\[m_1(t)=\int_{(1-\eta^2)Yt\le h_Q(\u)\le(1+\eta^2)Yt}
W_3(X^{-1}\u) \d u_0\d u_1\d u_2\]
and
\[m_2(t)=
\int_{(1+\eta-\eta^2)Yt\le h_Q(\u)\le(1+\eta+\eta^2)Yt}
W_3(X^{-1}\u) \d u_0\d u_1\d u_2.\]
Thus for example we have
\[\int_1^{\infty}(\log t)^{\rho-2}t^{-2}m_1(t)\d t=
\int_{\R^3}W_3(X^{-1}\u)I^{(1)}(\u)\d u_0\d u_1\d u_2\]
where $I^{(1)}(\u)$ is the integral of $(\log t)^{\rho-2}t^{-2}$,
subject to $t\ge 1$ and
\[(1-\eta^2)Yt\le h_Q(\u)\le(1+\eta^2)Yt.\]
We deduce that the relevant values of
$t$ have order of magnitude  
$Y^{-1}h_Q(\u)$, and are restricted to an interval of length 
$O(\eta^2 Y^{-1}h_Q(\u))$. It follows that
\[I^{(1)}(\u)\ll \eta^2 Y\LL^{\rho-2}h_Q(\u)^{-1}.\]
There is a similar estimate for the integral $I^{(2)}(\u)$ arising
from $m_2(t)$ and we deduce that
\begin{eqnarray*}
I^{(+)}-I^{(-)}
&\ll&\eta^2 Y\LL^{\rho-2}\int_{\R^3}
\frac{W_3(X^{-1}\u)}{h_Q(\u)}\d u_0\d u_1\d u_2\\
&=&\eta^2 XY\LL^{\rho-2}\int_{\R^3}
\frac{W_3(\u)}{h_Q(\u)}\d u_0\d u_1\d u_2.
\end{eqnarray*}
The integral above is bounded in terms of $Q_0$ and $Q_1$, by
Lemma \ref{QB}, and we deduce that $I^{(+)}-I^{(-)}\ll \eta^2 XY\LL^{\rho-2}$.
This produces the second error term in the lemma.
\end{proof}

Our next move is to sum over appropriate values of $Y$, so as to cover
the full range for $||\y||_\infty$. 
\begin{lemma}\label{8.2}
Set
\[S^{(1)}(B;X)=\sum_{\substack{\x\in\Z^3_{\mathrm{prim}}\\ (Q_0(\x),Q_1(\x))\not=(0,0)}}
\;\; \sum_{\substack{\y\in\Z^2_{\mathrm{prim}}\\  y_0Q_0(\x)+y_1Q_1(\x)=0\\
||\x||_\infty\le ||\y||_\infty\le B/X}}W_3(X^{-1}\x).\]
Then if $B^{1/3}(\log B)^A\le X\le B^{1/2}(\log B)^{-A}$ with a suitably large
constant $A$, we have 
\begin{eqnarray}\label{8.2Est}
S^{(1)}(B;X)&=&2\frac{\mathfrak{S}_S}{(\rho-2)!}B(\log X^3B^{-1})^{\rho-2}
\int_{\R^3}\frac{W_3(\v)}{h_Q(\v)}\d v_0\d v_1\d v_2\nonumber\\
&&\hspace{1cm}\mbox{}+O_{\eta}(B(\log B)^{\rho-3})+O(\eta B(\log B)^{\rho-2})
\end{eqnarray}
for any $\eta>0$.
\end{lemma}
As with Lemma \ref{rw2}, the error terms have implied constants that
may depend on $W_3$.  The constant $A$ here is not necessarily the same as in
Proposition \ref{P1*} or Lemma \ref{rw2}.
\begin{proof}
We begin by noting the general upper bound
\[S^{(0)}(X,Y)\ll XY\LL^{\rho-2},\;\;\;(X,Y\gg 1),\]
which follows from (\ref{sxy}).
We proceed to sum $S^{(0)}(X,Y)$ for values 
\[Y=BX^{-1}(1+\eta)^{-n}\;\;\;(n\ge 1). \]
Let $Y_1$ be the largest such value for which
$Y_1\le BX^{-1}(\log B)^{-1}$. Then for $Y<Y_1$ our trivial bound 
for $S^{(0)}(X,Y)$ produces a total $O(\eta^{-1}B(\log B)^{\rho-3})$.
The reader should observe here that $||\x||_\infty <Y_1$, so that the
awkward cut-off around $||\y||_\infty=||\x||_\infty$ is handled by our 
general upper bound.
For the range $Y_1\le Y\le BX^{-1}(1+\eta)^{-1}$
we may use Lemma \ref{rw2}. In this case we will have 
$X\LL^A\le Y\le X^2\LL^{-A}$
provided that we adjust the value of $A$ appropriately. 
 The error terms from 
Lemma \ref{rw2} contribute $O_\eta(B(\log B)^{\rho-3})+O(\eta B(\log B)^{\rho-2})$
and we deduce that
\begin{eqnarray}\label{e1}
S^{(1)}(B;X)&=&\sum_Y S^{(0)}(X,Y)+O(\eta^{-1}B(\log B)^{\rho-3})\nonumber\\
&=&2\frac{\mathfrak{S}_S}{(m-1)!}\sum_Y I^{(0)}(X,Y)\nonumber\\
&&\hspace{1cm}+
O_\eta(B(\log B)^{\rho-3})+O(\eta B(\log B)^{\rho-2}),
\end{eqnarray}
where the summations are for
\[Y_1\le Y=BX^{-1}(1+\eta)^{-n}\le BX^{-1}(1+\eta)^{-1}.\]

For the main term we see that
\begin{eqnarray*}
\sum_Y I^{(0)}(X,Y)&=&\int_1^{\infty}\frac{(\log t)^{\rho-2}}{t^2}
\int_{Y_1t\le h_Q(\u)\le BX^{-1}t}W_3(X^{-1}\u)
\d u_0\d u_1\d u_2\,\d t\\
&=&\int_{\R^3}W_3(X^{-1}\u)
\int_{\max\{XB^{-1}h_Q(\u),1\}}^{Y_1^{-1} h_Q(\u)}
\frac{(\log t)^{\rho-2}}{t^2}\d t\,\d u_0\d u_1\d u_2.
\end{eqnarray*}
We can extend the range of the inner integral to run up to infinity, 
at an overall cost  
\begin{eqnarray*}
  &\ll& (\log B)^{\rho-2}Y_1\int_{\R^3}\frac{W_3(X^{-1}\u)}{h_Q(\u)}
  \d u_0\d u_1\d u_2\\
&=&(\log B)^{\rho-2}Y_1X\int_{\R^3}\frac{W_3(\u)}{h_Q(\u)}\d u_0\d u_1\d u_2.
\end{eqnarray*}
The second integral is bounded, by Lemma \ref{QB}, whence
\begin{eqnarray*}
\sum_Y I^{(0)}(X,Y)&=&\int_{\R^3}W_3(X^{-1}\u)
\int_{\max\{XB^{-1}h_Q(\u),1\}}^{\infty}
\frac{(\log t)^{\rho-2}}{t^2}\d t\d u_0\d u_1\d u_2\\
&&\hspace{2cm}+O(B(\log B)^{\rho-3}).
\end{eqnarray*}
We now substitute $\u=X\v$, and then $t=X^3B^{-1}h_Q(\v)s$ to see that
\[\sum_Y
I^{(0)}(X,Y)=B\int_{\R^3}\frac{W_3(\v)}{h_Q(\v)}J
\left(X^3B^{-1}h_Q(\v)\right)\d v_0\d v_1\d v_2+O(B(\log B)^{\rho-3}),\] 
where we have temporarily written
\[J(x)=\int_1^\infty\frac{(\lp xs)^{\rho-2}}{s^2}\d s\]
with
\[\lp t:=\max(\log t,1).\]

We plan to replace 
\[J\left(X^3B^{-1}h_Q(\v)\right)=
\int_1^\infty\frac{(\lp X^3B^{-1}h_Q(\v)s)^{\rho-2}}{s^2}\d s\]
with
\[\int_1^\infty\frac{(\log X^3B^{-1})^{\rho-2}}{s^2}\d s=
\left(\log X^3B^{-1}\right)^{\rho-2}.\]
When $X^3B^{-1}h_Q(\v)s\ge 1$ we find that
  \begin{eqnarray*}
  \left(\lp(X^3B^{-1}h_Q(\v)s)\right)^{\rho-2}&=&(\log X^3B^{-1})^{\rho-2}\\
  && \hspace{-5mm}+O\!\left(\{1+\log s+\left|\log h_Q(\v)\right|\}^{\rho-2}
  (\log B)^{\rho-3}\right).
\end{eqnarray*} 
(The reader may wish to consider separately the cases $\rho=2$ and $\rho\ge 3$.)
On the other hand, when $X^3B^{-1}h_Q(\v)s\le 1$ we see that
\[\left(\lp(X^3B^{-1}h_Q(\v)s)\right)^{\rho-2}-(\log X^3B^{-1})^{\rho-2}
=-(\log X^3B^{-1})^{\rho-2}.\]
Thus if $X^3B^{-1}h_Q(\v)\ge 1$ it follows 
that
\[J\left(X^3B^{-1}h_Q(\v)\right)-\left(\log X^3B^{-1}\right)^{\rho-2}
\ll(\log B)^{\rho-3}\left\{1+\left|\log h_Q(\v)\right|\right\}^{\rho-2},\]
while if $X^3B^{-1}h_Q(\v)\le 1$ we see that
\begin{eqnarray*}
\lefteqn{J\left(X^3B^{-1}h_Q(\v)\right)-\left(\log X^3B^{-1}\right)^{\rho-2}}\hspace{3cm}\\
&\ll & (\log B)^{\rho-3}\left|\log h_Q(\v)\right|^{\rho-2}+(\log B)^{\rho-2}.
\end{eqnarray*}
We then conclude that
\begin{eqnarray*}
\lefteqn{\sum_Y I^{(0)}(X,Y)-
B(\log X^3B^{-1})^{\rho-2}\int_{\R^3}\frac{W_3(\v)}{h_Q(\v)}\d v_0\d v_1\d v_2}\\
&\ll& B(\log B)^{\rho-3}+B(\log B)^{\rho-3}\int_{\R^3}
\frac{W_3(\v)}{h_Q(\v)}
\left\{1+\left|\log h_Q(\v)\right|\right\}^{\rho-2}\d v_0\d v_1\d v_2\\
&&\hspace{1cm}\mbox{}+B(\log B)^{\rho-2}\int_{h_Q(\v)\le BX^{-3}}
\frac{W_3(\v)}{h_Q(\v)}\d v_0\d v_1\d v_2.
\end{eqnarray*}
The first integral is $O(1)$, by Lemma \ref{QB}. For the second, we
consider dyadic ranges $\tfrac12\lambda<h_Q(\v)\le\lambda$, with
$\lambda\le BX^{-3}\le\LL^{-3A}$. Each such range contributes
$O(\lambda^{1/2})$, by a second application of Lemma \ref{QB}, so that
the second integral is  
\[\ll\sum_{\lambda=2^n\le BX^{-3}}\lambda^{1/2}\ll \LL^{-3A/2}.\]
We therefore deduce that
\[\sum_Y I^{(0)}(X,Y)=B(\log X^3B^{-1})^{\rho-2}\int_{\R^3}\frac{W_3(\v)}{h_Q(\v)}
\d v_0\d v_1\d v_2+O(B(\log B)^{\rho-3})\]
when $A$ is large enough. The required estimate (\ref{8.2Est}) then
follows from (\ref{e1}). 
\end{proof}

We now remove the weight $W_3$, using the same ideas as for Lemma
\ref{rw2}. This will require the introduction of a second small
parameter, $\xi$. 
\begin{lemma}\label{rw3}
Let $\eta,\xi\in(0,\tfrac12)$ and set
\[S^{(2)}(B;X)=
\sum_{\substack{\x\in\Z^3_{\mathrm{prim}}\\  X<||\x||_\infty\le(1+\xi)X\\
    (Q_0(\x),Q_1(\x))\not=(0,0)}}\;\;\; \sum_{\substack{\y\in\Z^2_{\mathrm{prim}}\\ 
 y_0Q_0(\x)+y_1Q_1(\x)=0\\ ||\x||_\infty\le||\y||_\infty\le B/X}}1.\]
Then if $B^{1/3}(\log B)^A\le X\le B^{1/2}(\log B)^{-A}$ with a
suitably large constant $A$, we have 
\begin{eqnarray*}
  S^{(2)}(B;X)&=&4\xi\frac{\tau_\infty\mathfrak{S}_S}{(\rho-2)!}
  B(\log X^3B^{-1})^{\rho-2}+O(\xi^2 B(\log B)^{\rho-2})\\
&&\hspace{1cm}\mbox{}+O_{\eta,\xi}(B(\log B)^{\rho-3})
+O_\xi(\eta B(\log B)^{\rho-2}).
\end{eqnarray*}
\end{lemma}
\begin{proof}
Following the same strategy as for Lemma \ref{rw2}, we apply Lemma
\ref{8.2} with two weights $W_3^{(+)}(\v)$ and $W_3^{(-)}(\v)$, which
are majorants and minorants for the characteristic function of the set
$1\le||\v||_\infty\le 1+\xi$, supported respectively on  
$1-\xi^2\le||\v||_\infty\le 1+\xi+\xi^2$ and
$1+\xi^2\le||\v||_\infty\le 1+\xi-\xi^2$. We fix these weights once
and for all, so that when we apply Lemma \ref{8.2} the error terms are 
$O_{\eta,\xi}(B(\log B)^{\rho-3})$ and $O_\xi(\eta B(\log B)^{\rho-2})$.
Thus to complete the proof of Lemma \ref{rw3} it will
be enough to show that 
\beql{suff}
\int_{\R^3}\frac{W_3^{(+)}(\v)-W_3^{(-)}(\v)}{h_Q(\v)}\d v_0\d v_1\d v_2\ll \xi^2
\eeq
and that
\beql{suff1}
\int_{1\le||\v||_\infty\le 1+\xi}\frac{\d v_0\d v_1\d v_2}{h_Q(\v)}=2\xi\tau_\infty.
\eeq
To accomplish this we consider the integral
\[I(Q_0,Q_1;\lambda)=\int_{||\v||_\infty\le\lambda}\frac{\d v_0\d v_1\d v_2}{h_Q(\v)},\]
which is finite by Lemma \ref{QB}.
We note that $I(Q_0,Q_1;1)=2\tau_\infty$ by Lemma \ref{lem:rd}.
The substitution $\v=\lambda\u$ shows that $I(Q_0,Q_1;\lambda)=\lambda
I(Q_0,Q_1;1)$. 
Thus
\begin{eqnarray*}
\int_{1-\xi^2\le||\v||_\infty\le 1+\xi^2}\frac{\d v_0\d v_1\d v_2}{h_Q(\v)}&=&
I(Q_0,Q_1;1+\xi^2)-I(Q_0,Q_1;1-\xi^2)\\
&=&2\xi^2 I(Q_0,Q_1;1)\\
&\ll& \xi^2.
\end{eqnarray*}
Similarly we have
\[\int_{1+\xi-\xi^2\le||\v||_\infty\le
  1+\xi+\xi^2}\frac{\d v_0\d v_1\d v_2}{h_Q(\v)}
\ll\xi^2,\]
and (\ref{suff}) follows. Moreover we find that
\[\int_{1\le||\v||_\infty\le 1+\xi}\frac{\d v_0\d v_1\d v_2}{h_Q(\v)}=
I(Q_0,Q_1;1+\xi)-I(Q_0,Q_1;1),\]
which is enough to yield the relation (\ref{suff1}).
This completes the proof of the lemma.
\end{proof}

We next consider the range $X<||\x||_\infty\le 2X$.
\begin{lemma}\label{s31}
  Suppose that $B^{1/3}(\log B)^A\le X\le \tfrac12B^{1/2}(\log B)^{-A}$ with a
  suitably large constant $A$. Then if we write 
\[S^{(3)}(B;X)=\sum_{\substack{\x\in\Z^3_{\mathrm{prim}}\\  X<||\x||_\infty\le
    2X\\ (Q_0(\x),Q_1(\x))\not=(0,0)}}\;\;\; 
\sum_{\substack{\y\in\Z^2_{\mathrm{prim}}\\ C(\y)\not=0\\ y_0Q_0(\x)+y_1Q_1(\x)=0\\
 ||\x||_\infty\le   ||\y||_\infty\le B/||\x||_\infty}}1\]
we will have
\begin{eqnarray}\label{S3}
S^{(3)}(B;X)&=&4\log 2\frac{\tau_\infty\mathfrak{S}_S}{(\rho-2)!}B
(\log X^3B^{-1})^{\rho-2}+O(K^{-1}B(\log B)^{\rho-2})\nonumber\\
&&\hspace{1cm}\mbox{}+O_{\eta,K}(B(\log B)^{\rho-3})
+O_K(\eta B(\log B)^{\rho-2}).
\end{eqnarray}
for any positive integer $K$ and any $\eta>0$.
\end{lemma}
\begin{proof}
If we write
\[S^{(4)}(B;X)=\sum_{\substack{\x\in\Z^3_{\mathrm{prim}}\\  X<||\x||_\infty\le
    (1+\xi)X\\ (Q_0(\x),Q_1(\x))\not=(0,0)}}\;\;\; 
\sum_{\substack{\y\in\Z^2_{\mathrm{prim}}\\ C(\y)\not=0\\ y_0Q_0(\x)+y_1Q_1(\x)=0\\
 ||\x||_\infty\le   ||\y||_\infty\le B/||\x||_\infty}}1\]
we will have $S^{(2)}(B(1+\xi)^{-1};X)\le S^{(4)}(B;X)\le S^{(2)}(B;X)$. Note that the 
condition $C(\y)\not=0$ holds automatically, since $||\y||_\infty\ge||\x||_\infty\gg 1$.
Thus Lemma~\ref{rw3} shows that
\begin{eqnarray*}
  S^{(4)}(B;X)&=&4\xi\frac{\tau_\infty\mathfrak{S}_S}{(\rho-2)!}
  B(\log X^3B^{-1})^{\rho-2}+O(\xi^2 B(\log B)^{\rho-2})\\
&&\hspace{1cm}\mbox{}+O_{\eta,\xi}(B(\log B)^{\rho-3})
+O_\xi(\eta B(\log B)^{\rho-2}).
\end{eqnarray*}

We now set
$\xi=\exp\{K^{-1}\log 2\} -1$, so that $\xi\ll K^{-1}$. Then 
\[S^{(3)}(B;X)=\sum_{1\le n\le K}S^{(4)}(B; X(1+\xi)^{n-1}),\]
and we deduce that 
\begin{eqnarray*}
S^{(3)}(B;X)&=&4\xi\frac{\tau_\infty\mathfrak{S}_S}{(\rho-2)!}B
\sum_{1\le n\le K}(\log X^3(1+\xi)^{3n-3}B^{-1})^{\rho-2}\\
&&\hspace{-1.5cm} \mbox{}+O(K^{-1} B(\log B)^{\rho-2})
+O_{\eta,K}(B(\log B)^{\rho-3})+O_K(\eta B(\log B)^{\rho-2}).
\end{eqnarray*}
However
\[(\log X^3(1+\xi)^{3n-3}B^{-1})^{\rho-2}=
(\log X^3B^{-1})^{\rho-2}+O((\log B)^{\rho-3})\]
for $n\le K$.  (This holds trivially when $\rho=2$.) Thus 
(\ref{S3}) follows, since $K\xi=\log 2+O(K^{-1})$.
\end{proof}

We are now ready to complete the proof of Theorem \ref{extra}.
\begin{proof}[Proof of Theorem \ref{extra}]
When $||\x||_\infty\ll X\ll B^{1/3}(\log B)^{-1}$ the corresponding $\y$ has 
\[||\y||_\infty\ll ||\x||_\infty^2\ll B^{2/3}(\log B)^{-2},\]
by the argument from the introduction to \S \ref{S6}.
It follows that 
\[||\x||_\infty||\y||_\infty\ll B(\log B)^{-3},\]
so that the contribution
from all such points is of order $N(U,cB(\log B)^{-3})$ for a suitable constant
$c$. Thus points with $||\x||_\infty\ll X\ll B^{1/3}(\log B)^{-1}$ make a 
negligible contribution, by Theorem \ref{tub}.

For the remaining
ranges not covered by Lemma \ref{s31}, namely
\[B^{1/3}(\log X)^{-1}\ll X\ll B^{1/3}(\log X)^A\;\;\;\mbox{and}\;\;\;
B^{1/2}(\log B)^{-A}\ll X\ll B^{1/2},\]
we apply the universal bound
(\ref{sxy}). Summing over dyadic values of $Y$ with
of $X\ll Y\ll B/X$ this will suffice to show that  
\beql{S3gb}
S^{(3)}(B;X)\ll B(\log B)^{\rho-2}.
\eeq
whenever $X\gg 1$.

We proceed to sum $S^{(3)}(B;X)$ over dyadic values
\[X=B^{\1/2}2^{-n}\gg B^{1/3}(\log B)^{-1}\]
using (\ref{S3}) where 
applicable, and otherwise the bound (\ref{S3gb}). This
produces 
\begin{eqnarray*}
\sum_X S^{(3)}(B;X)&=&4\log 2\frac{\tau_\infty\mathfrak{S}_S}{(\rho-2)!}B
\sum_{B^{1/3}\le X=B^{1/2}2^{-n}< B^{1/2}} (\log X^3B^{-1})^{\rho-2}\\
&&\hspace{1cm} \mbox{}+O(B(\log B)^{\rho-2}\log\log B)
+O(K^{-1} B(\log B)^{\rho-1})\\
&&\hspace{1cm} \mbox{}+O_{\eta,K}(B(\log B)^{\rho-2})
+O_K(\eta B(\log B)^{\rho-1}).
\end{eqnarray*}
For the main term, we have
\begin{eqnarray*}
\lefteqn{\sum_{B^{1/3}\le X=B^{1/2}2^{-n}< B^{1/2}} (\log X^3B^{-1})^{\rho-2}}
\hspace{2cm}\\
&=&
\sum_{1\le n<(\log B)/(6\log 2)}(\tfrac12\log B-3n\log 2)^{\rho-2}\\
&=&\int_0^{(\log B)/(6\log 2)}(\tfrac12\log B-3x\log 2)^{\rho-2}\d x 
+O((\log B)^{\rho-2})\\
&=&\frac{(\tfrac12 \log B)^{\rho-1}}{3(\rho-1)\log 2}+O((\log B)^{\rho-2}).
\end{eqnarray*}
Since each point $x\in\P^3(\Q)$ corresponds to two values of $\x\in\Zprim^3$,
and similarly for $y\in\P^2(\Q)$, we deduce that
\begin{eqnarray*}
N_1(U,B)&=&\card\{(x,y)\in U(\Q): \, H(x,y)\le B,\, H(x)\le H(y)\}\\
&=&\tfrac14\sum_{X=B^{1/2}2^{-n} }S^{(3)}(B;X)+O(B(\log B)^{\rho-2})\\
&=&\frac{2^{1-\rho}\tau_\infty\mathfrak{S}_S}{3(\rho-1)!}B(\log B)^{\rho-1}
+O(B(\log B)^{\rho-2}\log\log B)\\
&&\hspace{1cm}+O(K^{-1} B(\log B)^{\rho-1})+O_{\eta,K}(B(\log B)^{\rho-2})\\
&&\hspace{2cm}+O_K(\eta B(\log B)^{\rho-1}).
\end{eqnarray*}

We can now complete the proof of the theorem. Given a small
$\ep>0$ we need to show that there is a corresponding $B(\ep)$ such
that the four error terms above have a total smaller than $\ep
B(\log B)^{\rho-1}$ whenever $B\ge B(\ep)$. Suppose that the order
constants 
for the four error terms are $c_1, c_2, c_3(\eta,K)$ and $c_4(K)$
respectively. We begin by fixing a value $K\in\N$ with
$c_2K^{-1}<\ep/3$. We then fix $\eta>0$ with  
$c_4(K)\eta<\ep/3$. It then suffices to choose $B(\ep)$ with 
\[\frac{c_1\log\log B(\ep)}{\log B(\ep)}+\frac{c_3(\eta,K)}{\log B(\ep)}<\ep/3. \qedhere\]
\end{proof}

\section{Asymptotics via conics}

In this section, just as in \S \ref{UBVC}, we aim to estimate $S(X,Y)$ by
counting points $\x$ on each individual conic $Q_\y(\x)=0$, and
summing the results with respect to $\y$.  However we now aim to prove
an asymptotic formula, which will only be possible if $X$ is suitably
large compared to $Y$. Thus we shall assume that $X\ge Y$ in the
current section.  

When $q(\x)\in\ZZ[\x]$ is a non-singular integral ternary quadratic
form and $W$ is a smooth compactly supported non-negative
weight function we define 
\[N(X,q;W)=\sum_{\x\in\Z^3_{\mathrm{prim}},\,\,\,q(\x)=0}W(X^{-1}\x),\]
whence
\[S(X,Y)=\sum_{\y\in\Z^2_{\mathrm{prim}}}
W_2(Y^{-1}\y)N(X,Q_\y;W_3),\]
provided that $X$ and $Y$ are large enough to ensure that
$(Q_0(\x),Q_1(\x))\not=(0,0)$ and $C(\y)\not=0$. 
 The asymptotic formula for $N(X,q;W)$ takes the shape
  \beql{NME}
  N(X,q;W)=\Main(X,q;W)+\Err(X,q;W)
  \eeq
  where the main term $\Main(X,q;W)$ is as predicted by the
  Hardy--Littlewood circle method, see 
  Heath-Brown \cite[Theorem 8 and Corollary 2]{HBcirc}.
  It will be convenient for us to write $N(X,q)$ in place of 
 $N(X,q;W_3)$, and similarly $\Main(X,q)$ 
  and $\Err(X,q)$ for $\Main(X,q;W_3)$ and $\Err(X,q;W_3)$; 
  but the reader should note
that we will consider another weight function as well as $W_3$.   
  
It turns out that regions where $C(\y)$ is small make a negligible
contribution.  
For the remaining $\y$  we will first show that
  $\Err(X,Q_\y)$ is suitably small on average. We then have the task of
  summing the main terms $\Main(X,Q_\y)$, which turns out to be
  surprisingly difficult. Our approach for these two steps
  will require our del Pezzo surface to be suitably generic.
We shall discuss the appropriate conditions further at the
points where their necessity arises. 

Our ultimate goal in the first phase of the argument is the following,
in which we write 
\beql{cEd}
\mathcal{E}=\exp\left\{\frac{\LL}{\sqrt{\log\LL}}\right\}
\eeq
for convenience.
\begin{proposition}\label{ultgoal}
Suppose that neither $\cl{M}$ nor $\cl{C}$ have rational points, so
that $\rho=2$, by Lemma \ref{r2}. Then
when $Y\le X\cl{E}^{-8}$ we have
  \begin{eqnarray*}
 S(X,Y)&=& XY\left\{\int_{\R^2} W_2(\y)\sigma_\infty(Q_\y;W_3)\d y_0\d y_1\right\}
\mathfrak{S}_S\\
&&\hspace{1cm}\mbox{}+O_\eta\left(XY(\log\log 3Y)^{-1/5}\right)+O(\eta^{2/3}XY) 
\end{eqnarray*}
for any $\eta>0$, where
\[\sigma_\infty(Q_\y;W_3)=\int_{-\infty}^\infty\int_{\mathbb{R}^3}
W_3(\x)e(-\theta Q_\y(\x))\d x_0\d x_1\d x_2 \d \theta.\]
\end{proposition}
We emphasize that the first $O$-term, with a subscript $\eta$,
has an implied constant which depends on $\eta$, (as well as $Q_0$ and
$Q_1$), while the implied
constant for the second $O$-term is independent of $\eta$. We fix
$\eta$ throughout this section and it will appear frequently in our
results. 

During the course of the argument we will only introduce the
conditions on $\cl{M}$ and $\cl{C}$ as they arise, so that we will not
initially assume that $\rho=2$. The reader should recall that $\cl{M}$
is the variety 
$Q_0(\x)=Q_1(\x)=0$, while $\cl{C}$ is given by $C(\y)=0$.

\subsection{Summing the error terms}

Our goal in this section is to show that the error terms $\Err(X,Q_\y)$
are suitably small on average, in regions where $C(\y)$ is reasonably large. 
To describe this latter condition we introduce a small real parameter 
$\eta\in(0,1)$, and restrict to vectors $\y$ in the square $||\y||_\infty\le Y$ 
that satisfy $|C(\y)|\ge\eta Y^3$. To state our result we will use the
parameter $\mathcal{E}$. We note that
this grows faster than any power of
$\LL$, and more slowly than any positive power of $XY$. Moreover one
has $\tau(n)^A\ll_A \mathcal{E}$ for $n\ll (XY)^A$, for any positive
constant $A$.

We then have the following.
\begin{proposition}\label{cdiff}
  Let $X\ge Y\mathcal{E}^8$ and assume that the intersection
  $\cl{M}$, given by $Q_0=Q_1=0$, has no rational points.  Then
\[\sum_{\substack{\y\in\Zprim^2,\; |C(\y)|\ge \eta Y^3\\ ||\y||_{\infty}\le Y}}
|\Err(X,Q_{\y})|\ll_\eta XY\LL^{-1}.\]
\end{proposition}
We will comment later, in the proof of Lemma \ref{zsmall}, on the
condition that the intersection $Q_0=Q_1=0$ should have no rational
points.

We begin the argument by describing $\Main(X,q;W)$, which is the main term
 in (\ref{NME}), in generality. This will involve a product of $p$-adic
densities given by 
\[\mathfrak{S}(q)=\prod_p\sigma(p;q),\]
where
\begin{eqnarray*}
  \sigma(p;q)&=&(1-p^{-1})\lim_{k\to\infty}p^{-2k}\card\{\x\Mod{p^k}:\,
  q(\x)\equiv 0\Mod{p^k}\}\\
  &=&\lim_{k\to\infty}p^{-2k}\card\{\x\Mod{p^k}:\,p\nmid\x,
  q(\x)\equiv 0\Mod{p^k}\},
\end{eqnarray*}
(see Heath-Brown \cite[Theorem 8 and Corollary 2]{HBcirc}).  The reader
should note that this is the $p$-adic density for $q=0$ seen as a
projective variety, while papers on the circle method often work with
the $p$-adic density for the affine cone.

We also need to define the singular integral 
\beql{quant}
\sigma_{\infty}(q;W)=
\lim_{T\to\infty}\int_{\R^3}W(\x)K_T(q(\x))\d x_0\d x_1\d x_2,
\eeq
where
\[  K_T(t)=T\max\{1-T|t|\,,\,0\}.\]
Equivalently we have
\[\sigma_{\infty}(q;W)=\int_{-\infty}^\infty J(\theta,q;W)\d \theta\]
with
\beql{Jdef}
J(\theta,q;W)=\int_{\R^3}W(\x)e(-\theta q(\x))\d x_0\d x_1\d x_2.
\eeq
In the special case $W=W_3$ we just set
\[\sigma_{\infty}(q)=\sigma_{\infty}(q;W_3).\]
 
With this notation we have
\[N(X,q;W)\sim \Main(X,q;W)\;\;\; (X\to\infty),\]
with
\[\Main(X,q;W)=\tfrac12\sigma_{\infty}(q;W)\mathfrak{S}(q)X,\]
see Heath-Brown \cite[Theorem 8]{HBcirc}.
\bigskip

Before proceeding further we present some estimates for
$\sigma(p;Q_\y)$, $\mathfrak{S}(Q_\y)$ and $\sigma_\infty(Q_\y;W)$.
We can use Lemma \ref{MLQB} to estimate $N(X,Q_\y)$. In view of 
Lemma \ref{DyL} we find, on letting $X\to\infty$, that
\[\sigma_{\infty}(Q_\y)\mathfrak{S}(Q_\y)\ll \frac{\kappa(\y)}{|C(\y)|^{1/3}}.\]
A second application of Lemma \ref{MLQB}, along with the definition 
(\ref{NME}) now shows that
\beql{fSe1}
\Err(X,Q_\y)\ll \kappa(\y)\left\{1+\frac{X}{|C(\y)|^{1/3}}\right\}.
\eeq
This will allow us to estimate averages of $\Err(X,Q_\y)$ when 
$C(\y)$ is small, by summing over small squares, using Proposition \ref{kAL}.

We will need information about the local densities $\sigma(p;Q_\y)$.
These are discussed elsewhere,
see Frei, Loughran and Sofos \cite[Propositions 3.4, 3.5 \&
  3.6]{FLS16}, for example. However we establish enough for our
purposes here. We assume throughout that $\y\in\Z^2$ is primitive,
with $C(\y)\not=0$.  It will be convenient to define
\beql{ssd}
\sigma^*(m;\y)=
m^{-2}\card\{\x\Mod{m}:\,\gcd(\x,m)=1,\,Q_{\y}(\x)\equiv 0\Mod{m}\},
\eeq
so that $\sigma(p,Q_\y) = \lim_{k \to \infty} \sigma^*(p^k,\y).$
We begin by examining $\sigma(p;Q_\y)$ in two easy cases.
\begin{lemma}\label{2ec}
  If $p\nmid C(\y)$ then
  \[\sigma(p;Q_\y)=1-p^{-2}=\sigma^*(p^k;\y), \quad \text{ for all }k\ge 1.\]
  Moreover if $p|| C(\y)$ then
  \[\sigma(p;Q_\y)=(1-p^{-1})(1+\chi(p;\y))=\sigma^*(p^k;\y),
  \quad \text{ for all }k\ge 2.\]
Lastly, when $p|| C(\y)$ we have
\[\chi(p;\y)=\frac{\sigma^*(p;\y)}{1-p^{-1}}-1-p^{-1}.\]
\end{lemma}
We refer the reader to the preamble to Lemma~\ref{LC}, where
$\chi(p;\y)$ was defined. We should also point out
that we arranged in \S \ref{sec:conventions}
that $C$ should vanish identically modulo $6^3$. Thus the primes of
relevance here all have $p\ge 5$. 
\begin{proof}
  When $p\nmid C(\y)$ one has
\[\card\{\x\Mod{p^k}:\,p\nmid\x,Q_{\y}(\x)\equiv 0\Mod{p^k}\}\hspace{2cm}\]
\[\hspace{2cm}=
p^{2k-2}\card\{\x\Mod{p}:\,p\nmid\x,Q_{\y}(\x)\equiv 0\Mod{p}\}\]
for $k\ge 1$, and one then finds that $\sigma(p;Q_\y)=1-p^{-2}$. If
$p|| C(\y)$ then $Q_\y$ can be diagonalised
over $\Z_p$ as $ax_0^2-bx_1^2+pcx_2^2$ for certain integers $a,b,c$
coprime to $p$. If $ab$ is a quadratic non-residue of $p$ one then
sees that  
\[\card\{\x\Mod{p^k}:\,p\nmid\x,Q_{\y}(\x)\equiv 0\Mod{p^k}\}=0\]
for $k\ge 2$, so that $\sigma(p;Q_\y)=0$. On the other hand if
$ab$ is a quadratic residue of $p$ then $Q_\y$ is equivalent over
$\Z_p$ to $x_0x_1+pdx_2^2$, with $p\nmid d$. If $p^k\mid x_0x_1+pdx_2^2$
with $k\ge 2$ and $p\nmid\x$ then $p$ divides exactly one of $x_0$ and
$x_1$, and one finds that
\[\card\{\x\Mod{p^k}:\,p\nmid\x,Q_{\y}(\x)\equiv 0\Mod{p^k}\}=
2p^k(p^k-p^{k-1}),\]
whence $\sigma(p;Q_\y)=2(1-p^{-1})$ in this case. Moreover one sees
that $\sigma^*(p;\y)=p^{-2}(p-1)$ when $ab$ is a non-residue of $p$ and
$\sigma^*(p;\y)=p^{-2}(2p^2-p-1)$ when $ab$ is a quadratic
residue. Thus
\[\sigma^*(p;\y)=(1-p^{-1})\{1+p^{-1}+\chi(p;\y)\}\]
for $p||C(\y)$, and the final assertion of the lemma follows.
\end{proof}

For the remaining
primes the expression for $\sigma(p;Q_\y)$ is potentially more
complicated, but we do have the following result.
\begin{lemma}\label{ymod1}
  Let $p$ be prime, and suppose that 
  $p^d|| C(\y)$ with $d\ge 0$.  Then for odd primes $p$ and $k\ge 2d+1$ we have
\[\sigma(p;Q_\y)=\sigma^*(p^k;\y),\]
while for $p=2$ and $k\ge 2d+3$ we have
\[\sigma(2;Q_\y)=\sigma^*(2^k;\y),\]
\end{lemma}
\begin{proof}
We begin by observing that
\[\card\{\x\Mod{p^j}:\,p\nmid\x,Q_{\y}(\x)\equiv 0\Mod{p^j}\}=
p^{-j}\sum_{u\Mod{p^j}}S(uQ_\y;p^j),\]
where
\[S(q;p^j)=\sum_{\substack{\x\Mod{p^j}\\ p\nmid\x}}e_{p^j}(q(\x)).\]
If $j\ge 2$ we have $S(uQ_\y;p^j)=p^3S(p^{-1}uQ_\y;p^{j-1})$ when $p\mid u$, 
and we deduce that
\begin{eqnarray*}
  \lefteqn{\card\{\x\Mod{p^j}:\,p\nmid\x,Q_{\y}(\x)\equiv
    0\Mod{p^j}\}}\hspace{2cm}\\
  &=&p^2\card\{\x\Mod{p^{j-1}}:\,p\nmid\x,Q_{\y}(\x)\equiv
  0\Mod{p^{j-1}}\}\\
  &&\hspace{2cm}\mbox{}+p^{-j}\sum_{\substack{u\Mod{p^j}\\ p\nmid u}}S(uQ_\y;p^j)
\end{eqnarray*}
whenever $j\ge 2$. We now claim that if $p\nmid u$ then $S(uQ_\y;p^j)$
vanishes when 
$j\ge 2d+2$ and $p\ge 3$, and for $j\ge 2d+4$ when $p=2$. Hence 
$\sigma^*(p^j;\y)=\sigma^*(p^{j-1};\y)$ for $j\ge 2d+2$ or $j\ge 2d+4$
as appropriate, which will suffice for the lemma. 

To establish the claim we choose $h=j-d-1$ for $p$ odd, and $h=j-d-2$
for $p=2$, and we write 
\[S(uQ_\y;p^j)=\sum_{\substack{\x_1\Mod{p^h}\\ p\nmid\x_1}}\;\; \sum_{\x_2\Mod{p^{j-h}}}
e_{p^j}(uQ_\y(\x_1+p^h\x_2)).\]
Since we will necessarily have $2h\ge j$ we may deduce that
\[Q_\y(\x_1+p^h\x_2)\equiv Q_y(\x_1)+p^h\x_2.\nabla
Q_\y(\x_1)\Mod{p^j}, \]
whence the sum over $\x_2$ will vanish unless
$p^j| p^hu\nabla Q_\y(\x_1)$. However both $u$ and $\x_1$ are
coprime to $p$ so that this condition would imply that the matrix for
$2Q_\y$ has rows that are linearly dependent modulo $p^{j-h}$. When $p$ is 
odd it would follow that $p^{j-h}\mid C(\y)$, while for $p=2$ the corresponding 
conclusion would be that $2^{j-h-1}\mid C(\y)$. This yields a contradiction,  
establishing the claim.
\end{proof}

 We will also need some general bounds for $\sigma^*(p^k;\y)$,
 applicable even when $p^2|C(\y)$. 
 \begin{lemma}\label{gb2}
 Let $D_{\y}$ be the highest common factors of the $2\times 2$ minors
 of the matrix for $Q_{\y}$. Then if $p^d|| C(\y)$ and $p\nmid D_{\y}$
 we have $\sigma^*(p^k;Q_{\y})\le d+1$ whenever $k>d\ge 0$. 
 \end{lemma}
 Since $D_{\y}\ll 1$, by Lemma \ref{DyL}, the lemma applies to all $p\gg 1$.
 
 \begin{proof}
 Since $p\nmid D_\y$ the prime $p$ must be odd. We can therefore
 diagonalize $Q_\y$ modulo $p^k$, as $A_0x_0^2+A_1x_1^2+A_2x_2^2$. The
 assumption $p\nmid D_\y$ shows that at most one of the coefficients
 $A_i$ is a multiple of $p$, and we will assume without loss of
 generality that $p^d|| A_2$ and $p\nmid A_0 A_1$. Under these
 conditions we write 
 \[N(p;k,d)=\card\{\x\Mod{p^k}: p\nmid \x,\, p^k\mid A_0x_0^2+A_1x_1^2+A_2x_2^2\},\]
Our goal is then to show that 
\beql{ine}
N(p;k,d)\le(d+1)p^{2k}
\eeq
for $k>d\ge 0$.

We argue by induction on $d$. The cases $d=0$ and $d=1$ of (\ref{ine}) follow 
directly from the methods used to establish Lemma \ref{2ec}, so we assume that 
$d\ge 2$. In general, we will have two cases. When $p$ divides both $x_0$ and 
$x_1$ the condition $p\nmid\x$ implies that $p\nmid x_2$. Writing 
$x_0=py_0,x_1=py_1$ and $x_2=y_2$ we then see that the corresponding contribution to $N(p;k,d)$ is 
\begin{eqnarray*}
\lefteqn{\card\{y_0,y_1\Mod{p^{k-1}},\, y_2\Mod{p^k}: p\nmid y_2,\, p^{k-2}\mid A_0y_0^2+A_1y_1^2+p^{-2}A_2y_2^2\}}\hspace{1cm}\\
&=&
p^4\card\{\y\Mod{p^{k-2}}: p\nmid y_2,\, p^{k-2}\mid A_0y_0^2+A_1y_1^2+p^{-2}A_2y_2^2\}\\
&\le&p^4 N(p;k-2,d-2).
\end{eqnarray*}
At the last step we merely use the fact that $p\nmid y_2$ implies $p\nmid\y$. By our 
induction assumption the contribution to $N(p;k,d)$ from solutions with $p\mid \x$ 
is therefore at most $(d-1)p^{2k}$.

Alternatively, if $p\nmid x_0$ say, then we also have $p\nmid x_1$, and vice-versa. For any 
given value of $x_2$ there are at most $2p$ solutions $z_0,z_1$ of \[A_0z_0^2+A_1z_1^2\equiv -A_2x_2^2\Mod{p}\]
with $p\nmid z_0 z_1$, and each of these lifts to $p^{k-1}$ solutions modulo $p^k$.
 It follows that
 \[\card\{\x\Mod{p^{k}}: p\nmid x_0 x_1,\, p^{k}\mid A_0x_0^2+A_1x_1^2+p^{-2}A_2x_2^2\}
 \le 2p^{2k}.\]
 Combining this with our result for the case in which $p$ divides both $x_0$ and $x_1$ we find that
  \[N(p;k,d) \le (d-1)p^{2k}+2p^{2k}=(d+1)p^{2k}\]
 which completes the induction step.
 \end{proof}

For primes dividing $D_\y$ we will use the following crude bound.
\begin{lemma}\label{gb}
  When $\gcd(\y,p)=1$ we have
  \[\sigma^*(p^k;\y)\ll k\]
  for $k\ge 1$.
\end{lemma}
\begin{proof}
  If $p$ is odd we can diagonalize $Q_\y$ modulo $p^k$ with a
  unimodular matrix, while for
  $p=2$ we can diagonalize $Q_\y$ modulo $2^k$ with a matrix
  of determinant $4$. (By this we mean that $Q_\y(M\x)$ is diagonal modulo 
  $2^k$, where $M$ is an integer matrix of determinant 4.) It thus suffices 
  to consider zeros of a diagonal
  form $A_0x_0^2+A_1x_1^2+A_2x_2^2$. It follows from Lemma \ref{DyL}
  that
  \[\gcd(A_0A_1,A_0A_2,A_1A_2,p^k)\ll 1,\]
  even when $p=2$. We can assume that
  $\gcd(A_0A_1,p^k)\ll 1$ by re-labelling the variables if necessary,
  and then it suffices to show that 
  \[\card\{x_0,x_1\,(\mbox{mod }p^k):\, p^k\mid A_0x_0^2+A_1x_1^2-n\}
  \ll kp^k\]
  uniformly for every integer $n$.

  To establish this we classify pairs $x_0,x_1$ according to the
  value of $\gcd(x_0,x_1,p^k)$, which we take to be $p^h$. When 
  $h\ge k/2$ the number of such
  pairs is $O(p^{2k-2h})$, which is satisfactory. For each remaining
  value of $h$ we claim that
 \[ \card\{z_0,z_1\,(\mbox{mod }p^{k-2h}):\,\gcd(z_0,z_1,p)=1,\,
  p^{k-2h}\mid A_0z_0^2+A_1z_1^2-n'\} \ll p^{k-2h},\]
  for any $n'$.  This will be sufficient, since each pair $z_0,z_1$
  corresponds to $p^{2h}$ pairs $x_0,x_1$.   However the congruence
  $A_0z_0^2\equiv n''($mod $p^e)$ has $O(1)$ solutions $z_0$ coprime
  to $p$ when $\gcd(A_0,p^e)\ll 1$, and similarly for the congruence
  $A_1z_1^2\equiv n''($mod $p^e)$ when $\gcd(A_1,p^e)\ll 1$.
The lemma now follows.
  \end{proof}

Combining Lemmas \ref{2ec}, \ref{ymod1}, \ref{gb2}, and \ref{gb} we
deduce the following bound for $\mathfrak{S}(Q_\y)$. 
\begin{lemma}\label{mfSb}
We have
  \[\mathfrak{S}(Q_\y)\ll_{Q_0,Q_1}\kappa(\y)\le\tau(C(\y))\]
  for any primitive $\y$ with $C(\y)\not=0$.
\end{lemma}
The reader should recall that $\kappa(\y)=\kappa(Q_\y)$ was defined in
Lemma \ref{MLQB}.
\bigskip

Next we examine the singular integral $\sigma_\infty(Q_\y;W)$. We 
remind the reader of
the definition of $J(\theta,Q_\y;W)$ in \eqref{Jdef}.

\begin{lemma}\label{pdsi}
 We have 
  \beql{JB}
  J(\theta,Q_\y;W)\ll_W
  \min\left(1\,,\,\frac{1}{|\theta|\,||\y||_\infty}\right)
  \min\left(1\,,\,\frac{||\y||_\infty}{\sqrt{|\theta C(\y)|}}\right).
  \eeq
	Hence
  \[  \sigma_\infty(Q_\y;W)\ll_W
  \frac{1}{||\y||_\infty}\log\left(2+\frac{||\y||_\infty^3}{|C(\y)|}\right).\]
  Moreover
  \[\sigma_\infty(Q_\y)-\sigma_\infty(Q_{\y'})\ll_W
  ||\y-\y'||_\infty^{1/5}|C(\y)|^{-2/5}\]
 for any $\y,\y'\in\R^2$ with $0<|C(\y)|\le|C(\y')|$.
\end{lemma}

\begin{proof}
  By the argument in Browning and Heath-Brown \cite[\S
    4.4]{BHBQB} (which relates to forms in 4 variables, rather than
  ternary forms) we have
  \[J(\theta,Q_\y;W)\ll_W
  \prod_{j=1}^3\min\left(1\,,\,\frac{1}{\sqrt{|\theta
      \mu_j|}}\right),\]
  where $\mu_1,\mu_2,\mu_3$ are the eigenvalues of the matrix for
  $Q_\y$. Lemma \ref{DyLR} then shows \eqref{pdsi}.
The first claim of the lemma then follows on integrating over $\theta$.

For the second claim we note that
\[e(-\theta Q_\y(\x))-e(-\theta Q_{\y'}(\x))\ll_W |\theta| \cdot ||\y-\y'||_\infty\]
 on the support of $W$, whence
 \[J(\theta,Q_\y)-J(\theta,Q_{\y'})\ll_W |\theta| \cdot ||\y-\y'||_\infty.\]
 The bound (\ref{JB}) shows that
 $J(\theta,Q_\y)\ll_W |\theta|^{-3/2}|C(\y)|^{-1/2}$ for all $\theta$, whence
\begin{eqnarray*}
\sigma_\infty(Q_\y)-\sigma_\infty(Q_{\y'})&\ll_W&
\int_{-A}^A |\theta| \cdot ||\y-\y'||_\infty \d \theta+\int_{|\theta|\ge A}
|\theta|^{-3/2}|C(\y)|^{-1/2} \d \theta\\
&\ll_W& A^2||\y-\y'||+A^{-1/2}|C(\y)|^{-1/2}
\end{eqnarray*}
for any $A>0$. The lemma then follows on taking
$A=||\y-\y'||_{\infty}^{-2/5}|C(\y)|^{-1/5}$.
\end{proof}

Having completed these preliminaries we can begin the
proof of Proposition~\ref{cdiff}.  This will
require the error term $\Err(X,q;W)$ to be estimated with a suitable
dependence on the coefficients of the quadratic form $q$.  In
order to have such an asymptotic we will require the coefficients of
$q$ to be suitably small compared to $X$.
A satisfactory estimate is given by Heath-Brown
\cite[Theorem 6]{HBconic}. To describe this latter result we introduce a
little more notation. Since $q$ is an integral form $2q$ will be represented by
a symmetric integral matrix, $Q$ say, and we set
$\Delta(q)=\tfrac12\det(Q)$.  The reader should note that with this 
convention we will have $\Delta(Q_\y)=4C(\y)$.  We write $||q||$ for 
the $L^2$-norm of the entries of $Q$, and define
\[\xi(q)=\frac{||q||^3}{|\Delta(q)|}.\]
Finally we set
\[W_0(\x)=\left\{\begin{array}{cc}
\exp\left\{-\frac{1}{1-||\x||_2^{2}}\right\}, &
||\x||_2<1,\\ 0, & \mbox{otherwise.}\end{array}\right.\]
We are now ready to state the following consequence of Theorem 6 of
Heath-Brown \cite{HBconic}. (The latter works with $||\z||_2$, but we
can clearly replace this with $||\z||_{\infty}$.)
\begin{proposition}\label{c2}
  Suppose that the form $q$ is primitive and isotropic over $\Q$, and
 let $\z_q$ be  
  a non-trivial integer zero with $||\z||_{\infty}$ minimal.  Then
\[\Err(X,q;W)\ll_W \tau(|\Delta(q)|)+\sigma_{\infty}(q;W_0)\mathfrak{S}(q)
4^{\omega(|\Delta(q)|)}\xi(q)
\left(\frac{||q||^2}{||\b{z}_q||_{\infty}}\right)^{1/4}X^{3/4}\LL.\] 
\end{proposition}
When $q$ is not isotropic over $\Q$ one readily verifies that either
$\sigma_{\infty}(q;W_0)$ or $\mathfrak{S}(q)$ vanishes, so that the
result remains true. If $q$ is not primitive, but $q=kq_0$ for some
primitive form $q_0$, then $\mathfrak{S}(q)=|k|\mathfrak{S}(q_0)$ and  
$\sigma_\infty(q;W)=|k|^{-1}\sigma_\infty(q_0;W)$. One may then readily use 
Proposition \ref{c2} for $q_0$ to obtain a version for $q$. This will
take exactly the same shape, except that the implied constant will
depend on $k$.

We next state some basic estimates involving
$||Q_{\y}||$ and $\xi(Q_{\y})$.
\begin{lemma}\label{c1}
Suppose that $\y$ is primitive, with $||\y||_{\infty}\ll Y$ and $|C(\y)|\ge\eta Y^3$. 
Let $c(Q_\y)$ be the content of $Q_\y$, (namely 
    the highest common factor of the coefficients).  Then:
  \begin{enumerate}
  \item[(i)]
   \[c(Q_\y)\ll 1,\] 
      \item[(ii)] 
    \[  ||\y||_{\infty}\ll ||Q_{\y}||\ll ||\y||_{\infty},\]
  \item[(iii)]
    \[1 \ll_\eta \xi(Q_{\y})\ll_\eta 1,\]
  and
  \item[(iv)]
  \[\sigma_{\infty}(Q_\y;W_0)\mathfrak{S}(Q_\y)
4^{\omega(|C(\y)|)}\xi(Q_\y)\LL  \ll_\eta \frac{\mathcal{E}}{||\y||_{\infty}}\]
	where $\mathcal{E}$ is as in \eqref{cEd}.
   \end{enumerate}
\end{lemma}
\begin{proof}
  Part (i) is an immediate consequence of Lemma \ref{DyL}.
The upper bound in part (ii) is trivial. To prove
  the lower bound we observe that $||Q_{\y}||$ is a
  continuous function of $\y$, and hence attains its lower bound on
  the compact set $||\y||_{\infty}=1$. This lower bound cannot be zero, since the
  two forms $Q_0$ and $Q_1$ are not proportional. Thus
  $||Q_{\y}||\gg 1$ on the set $||\y||_{\infty}=1$, and the general result then
  follows by homogeneity.
  Part (iii) of the lemma is also trivial, while part (iv) follows from Lemmas \ref{mfSb} and \ref{pdsi}.
\end{proof}

Lemma \ref{c1} enables us to deduce from Proposition \ref{c2} that
\beql{cme}
\Err(X,Q_{\y})\ll_\eta  \mathcal{E}
  \left\{1+\frac{X^{3/4}}{||\y||_{\infty}}
  \left(\frac{||\y||_{\infty}^2}{||\z(\y)||_{\infty}}\right)^{1/4} \right\},
  \eeq
 where we have put $\z(\y)=\z_{Q_{\y}}$.  This will give us a good
 bound unless $\z(\y)$ is too
 small. Specifically, we now have the following obvious conclusion.
\begin{lemma}\label{RZ}
   We have
\[\sum_{\substack{\y\in\Zprim^2\\ Y<||\y||_{\infty}\le 2Y,\; |C(\y)|\ge\eta Y^3
         \\  ||\z(\y)||_{\infty}\ge Z}}
  |\Err(X,Q_{\y})|\ll_\eta
   \mathcal{E}\left\{Y^2+XY(Y^2/XZ)^{1/4}\right\}.\]
Moreover, when $Y\le X^{1/6}$ we have
   \[\sum_{\substack{\y\in\Zprim^2\\ Y<||\y||_{\infty}\le 2Y,\; |C(\y)|\ge\eta Y^3}}
          |\Err(X,Q_{\y})| \ll_\eta (XY)^{7/8}.\]
 \end{lemma}
 The second statement follows from (\ref{cme}) in view of the trivial
 bound $||\z(\y)||_{\infty}\ge 1$.

When $||\z(\y)||_{\infty}\le Z$ we will estimate $\Err(X,Q_\y)$ trivially 
and prove the following result.
  \begin{lemma}\label{zsmall}
Suppose that the intersection $Q_0=Q_1=0$ has no rational points. Then we have
\[\sum_{\substack{\y\in\Zprim^2\\ Y<||\y||_{\infty}\le 2Y,\, |C(\y)|\ge \eta Y^3
         \\  ||\z(\y)||_{\infty}\le Z}} |\Err(X,Q_{\y})| \ll_\eta  \mathcal{E}^2XZ\]
   whenever $Z\le Y\le X$.
 \end{lemma}
  We shall see in the course of the proof where the condition
  on the intersection $Q_0=Q_1=0$ arises.
 
 \begin{proof}
 The estimate (\ref{fSe1}) shows that
\[  \Err(X,Q_{\y})\ll_\eta \mathcal{E}XY^{-1}\]
  for the vectors $\y$ under consideration, so we need to estimate the
  number $N_0$ of choices for $\y$ with $||\z(\y)||_{\infty}\le
  Z$. This is at most the number of pairs  
 $(\y,\z)\in\Zprim^2\times\Zprim^3$ with $||\y||_{\infty}\le 2Y$,
 $ ||\z||_{\infty}\le Z$, $C(\y)\not=0$, and $ Q_{\y}(\z)=0$. Thus
$(\z,\y)$ produces a point on our del Pezzo surface surface $S$, with
 height at most $2YZ$. Moreover the point will be in the open subset
 $U$, unless 
$Q_0(\z)=Q_1(\z)=0$. Hence, if we assume that there are no such points $\z$
 we may deduce from Theorem \ref{tub} that $N_0\ll N(U,2YZ)\ll\mathcal{E}YZ$, 
  and the result follows.
  \end{proof}
  
  The reader will see that if, on the other hand, there is a primitive
  vector $\z=\z_0$ with  
  $Q_0(\z)=Q_1(\z)=0$ then $\z(\y)$ will always be at least as small
  as $\z_0$, and in this case there is no way to use Proposition
  \ref{c2} effectively. 
\bigskip

We can now combine Lemmas \ref{RZ} and \ref{zsmall} to
conclude that 
\[\sum_{\substack{\y\in\Zprim^2\\ Y<||\y||_{\infty}\le 2Y,\; |C(\y)|\geq\eta Y^3}}
          |\Err(X,Q_{\y})| \ll_\eta (XY)^{7/8}\]
when $Y\le X^{1/6}$, and otherwise that
\[\sum_{\substack{\y\in\Zprim^2\\ Y<||\y||_{\infty}\le 2Y,\; |C(\y)|\geq\eta Y^3}}
  |\Err(X,Q_{\y})|\ll_\eta \mathcal{E}\left(Y^2+XY(Y^2/XZ)^{1/4}\right)+\mathcal{E}^2 XZ\]
for $1\le Z\le Y$. If $X^{1/6}\le Y\le X\mathcal{E}^{-8}$ we may choose 
$Z=Y\mathcal{E}^{-3}$, giving a bound $O_\eta(\mathcal{E}^{-1/4}XY)$.
Since we already have a satisfactory 
bound when $Y\le X^{1/6}$ we therefore see that
\[\sum_{\substack{\y\in\Zprim^2\\ Y<||\y||_{\infty}\le 2Y,\; |C(\y)|\ge\eta Y^3}}
|\Err(X,Q_{\y})|\ll \mathcal{E}^{-1/4}XY\]
throughout the range $Y\le X\mathcal{E}^{-8}$, so that Proposition
\ref{cdiff} follows, on summing over dyadic ranges for $Y$.

\subsection{Counting via conics -- Manipulating the main term}
We now turn our attention to the sum of the main terms
\[\sum_{\substack{\y\in\Zprim^2,\; C(\y)\not=0\\ ||\y||_{\infty}\le B/X}}
\Main(X,Q_{\y}).\]
In view of Proposition \ref{cdiff} we will assume for the rest of the
argument that $\cl{M}$ has no rational points.

In order to control the size of $C(\y)$ we will restrict $\y$ to lie in a 
square $\cl{V}=(Y_0,Y_0+H]\times(Y_1,Y_1+H]$ with 
$|Y_0|,|Y_1|\le Y$ and $Y^{1/3}\le H\le Y$, and we write
$\cl{U}=\Zprim^2\cap\cl{V}$ as in Proposition \ref{kAL}. In this and the next
section we shall assume that $|C(\y)|\ge\eta Y^3$ throughout
$\cl{V}$. Our goal is the following result. 
\begin{proposition}\label{PMB}
Suppose that $|C(\y)|\ge\eta Y^3$ throughout $\cl{V}$ and that $C(\y)$
is irreducible over $\Q$. Then if $Y(\log\log 3Y)^{-1}\le H\le Y$ we
have  
\begin{eqnarray*}
 \sum_{\y\in\cl{U}}W_2(Y^{-1}\y)\Main(X,Q_\y)
&=& X
\left\{\int_{\cl{V}}W_2(Y^{-1}\y)\sigma_\infty(Q_\y)\d y_0 \d y_1\right\}\mathfrak{S}_S\\
&&\mbox{}
+O_\eta\left(XY^{-1}H^2(H/Y)^{1/5}\right).
\end{eqnarray*}
\end{proposition}
We were already assuming that $\cl{M}(\Q)=\emptyset$, and the
hypotheses of the proposition include the assumption that
$\cl{C}(\Q)=\emptyset$. We are therefore in the situation of Lemma
\ref{r2}, so that we have $\rho=2$ in Proposition \ref{PMB}.
\bigskip

We begin our argument by considering the singular integral.
It could happen that $Q_\y$ is positive definite for all
$\y\in\cl{V}$, in which case 
$\Main(X,Q_\y)$ will vanish for every $\y$. We now consider the
alternative situation. 
\begin{lemma}\label{siconst}
  Let $\cl{V}$ be as above, and suppose that $Q_\y$ is indefinite for
  at least one $\y\in\cl{V}$. Then it must be indefinite at all
  $\y\in\cl{V}$. 
\end{lemma}
\begin{proof} 
Suppose for a contradiction that $Q_\w$, say, is not indefinite. 
Then $\pm Q_\w$ must be a positive definite form for some choice of
sign, since $C(\w)\not=0$. It follows that $Q_\w$ and 
  $Q_\y$ can be simultaneously diagonalized using a matrix in
  GL$_3(\mathbb{R})$, as Diag$(\lambda_1,\lambda_2,\lambda_3)$ and
  Diag$(\mu_1,\mu_2,\mu_3)$ respectively, say. Here the three
  $\lambda_i$ will all have the same sign, while the three $\mu_i$
  will not. In particular there is an index $i$ for which $\lambda_i$
  and $\mu_i$ have opposite signs. It follows that
  $t\lambda_i+(1-t)\mu_i=0$ for some $t\in(0,1)$, so that
$t$Diag$(\lambda_1,\lambda_2,\lambda_3)+(1-t)$Diag$(\mu_1,\mu_2,\mu_3)$
is singular. We would then conclude that $Q_{t\w+(1-t)\y}$ would be
singular. This gives the required contradiction since the discriminant
cubic $C$ cannot vanish on $\cl{V}$.
\end{proof}

Our next step is to remove the effect of the weight $W_2$ and of the singular 
integral from the main term, as shown in the following result.    
\begin{lemma}\label{sib}
We have
\begin{eqnarray*}
\lefteqn{\sum_{\y\in\cl{U}}W_2(Y^{-1}\y)\sigma_\infty(Q_\y)\mathfrak{S}(Q_\y)}\\
&=&\left\{H^{-2}\int_{\cl{V}}W_2(Y^{-1}\y)\sigma_\infty(Q_\y)\d y_0\d y_1
+O_\eta(H^{1/5} Y^{-6/5})\right\} 
\sum_{\y\in\cl{U}}\mathfrak{S}(Q_\y).
\end{eqnarray*}
\end{lemma}
\begin{proof}
The construction of the weight $W_2$ at the beginning of \S \ref{S6} shows that
$W_2(Y^{-1}\y)-W_2(Y^{-1}\y')\ll HY^{-1}$ for any two points $\y',\y''\in\cl{V}$.
Moreover Lemma \ref{pdsi} shows that $\sigma_\infty(Q_\y)\ll_\eta Y^{-1}$ throughout
$\cl{V}$, and that 
\[\sigma_\infty(Q_{\y'})-\sigma_\infty(Q_{\y''})\ll_\eta H^{1/5} Y^{-6/5}\]
for any two points $\y',\y''\in\cl{V}$. These bounds suffice to complete the proof.
\end{proof}

In view of Lemma \ref{sib} we proceed to consider
\[\sum_{\y\in\cl{U}}\mathfrak{S}(Q_{\y}).\]
Our task for the remainder of this section is to put $\mathfrak{S}(Q_\y)$ 
into a form where we can sum conveniently over $\y\in\cl{U}$.

Our first concern is to separate out values of $\y$ for which $C(\y)$
has a  large prime power factor. We define
\beql{Ndef}
N=\left[\frac{\log Y}{\sqrt{\log\log 3Y}}\right]
\eeq
and set
\[P=\prod_{p\le N}p^{[(\log N)/(\log p)]}.\]
We then put
\[\bad_1(\y)=\prod_{p^e||C(\y),\; p\le N}p^e,\;\;\;
\bad_2(\y)=\prod_{\substack{p^e||C(\y),\; e\ge 2\\ p> N}}p^e\]
and write
\[\bad(\y)=\bad_1(\y)\bad_2(\y)\]
for the ``bad'' part of $C(\y)$, which includes all prime
divisors up to $N$. It follows from this definition that $\bad(\y)$ is
coprime to the ``good'' part $C(\y)/\bad(\y)$.
The bad part $\bad(\y)$ will usually not be too large, so that we may
expect that $\bad(\y)\mid P$ for ``most'' choices of $\y$.
This leads us to define
\begin{equation} \label{def:SVB}
S(\cl{V},\beta)=\sum_{\substack{\y\in\cl{U}\\ \bad(\y)=\beta}}
\mathfrak{S}(Q_{\y}),
\end{equation}
whence
\beql{dded1}
\sum_{\y\in\cl{U}}\mathfrak{S}(Q_{\y})=\sum_{\beta\mid P}S(\cl{V},\beta)+
\sum_{\substack{\y\in\cl{U}\\ \bad(\y)\, \nmid\, P}}
\mathfrak{S}(Q_{\y}).
\eeq
We plan to estimate $S(\cl{V},\beta)$ asymptotically when $\beta\mid P$,
and to show that the second sum is negligible.

Our next result begins our analysis of $S(\cl{V},\beta)$. We extend the 
definition of $\chi(p;\y)$ (see the preamble to Lemma \ref{LC})
by writing
\[\chi(d;\y)=\prod_{p|d}\chi(p;\y),\]
and we define $\chi(\infty;\y)$ as an analogue of $\chi(p;\y)$ to be 0 when 
$C(\y)=0$, and otherwise to be $+1$ if $Q_\y$ is indefinite and $-1$
when $Q_\y$ is definite. We then see from Lemma \ref{siconst} that
$\chi(\infty;\y)$ is constant  
on $\cl{V}$. We now have the following result.
\begin{lemma}\label{de}
  Suppose that $\chi(\infty;\y)=+1$ and that
$\bad(\y)=\beta$ with $\beta| P$.
Then
\[\mathfrak{S}(Q_{\y})=
\left\{1+O\left((\log\log 3Y)^{-1/2}\right)\right\}\sigma^*(P^5;\y)
\{\Sigma_1(\y)+\Sigma_2(\y)\},\]
with
\[\Sigma_1(\y)=\sum_{\substack{d\mid C(\y),\, \gcd(d,P)=1\\ d\le
    Y^{3/2}}}\mu^2(d)\chi(d;\y)\] 
and
\[\Sigma_2(\y)=
\sum_{\substack{d\mid C(\y),\, \gcd(d,P)=1\\ d<|C(\y)|/(Y^{3/2}\beta)}}\mu^2(d)\chi(d;\y).\]
\end{lemma}

The proof naturally produces a single sum, taken over all $d$ dividing 
$C(\y)/\beta$. Thus $d$ has order potentially as big as $Y^3$. This is
too large for us to perform the corresponding summation over $\y$, and
so we switch to the complementary divisor for values $d>Y^{3/2}$,
producing the two sums $\Sigma_1(\y)$ and $\Sigma_2(\y)$. This
will of course be explained in detail during the proof.  

\begin{proof}[Proof of Lemma \ref{de}]
  We begin by writing
$\mathfrak{S}(Q_{\y})$ as a product of terms $\sigma(p;Q_\y)$
separated into two cases, namely those with $p>N$, and those with
$p\le N$. If $p>N$ then $p$ cannot divide $\beta=\bad(\y)$. It follows that
$p\not=2$, and that $p^2\nmid C(\y)$. Thus Lemma \ref{2ec} shows that
either $\sigma(p;Q_\y)=1-p^{-2}$, (when $p\nmid C(\y)$) or
$\sigma(p;Q_\y)=(1-p^{-1})(1+\chi(p;\y))$, (when $p\mid C(\y)$). We
therefore see that
\[\prod_{p>N}\sigma(p;Q_\y)=\prod_{p>N}\{1+O(p^{-2})\}
\prod_{\substack{p>N\\ p\mid C(\y)}}\{1+O(p^{-1})\}
\prod_{\substack{p>N\\ p\mid C(\y)}}\{1+\chi(p;\y)\}.\]
The first product is $1+O(N^{-1})$.  For the second, we note that
\[\omega(C(\y))\ll \log Y/\log\log 3Y\ll N,\]
so that the product is
$1+O(\log Y/N\log\log 3Y)$. It follows that
\[\prod_{p>N}\sigma(p;Q_\y)=
\left\{1+O\left((\log\log 3Y)^{-1/2}\right)\right\}
\prod_{\substack{p>N\\ p\mid C(\y)}}\{1+\chi(p;\y)\},\]
and hence
\[\mathfrak{S}(Q_{\y})=
\left\{1+O\left( (\log\log 3Y)^{-1/2}\right)\right\}
\prod_{p\le N}\sigma(p;Q_\y)
\prod_{\substack{p>N\\ p\mid C(\y)}}\{1+\chi(p;\y)\}.\]

We next consider the product for primes $p\le N$.
According to Lemma \ref{ymod1} we will
have $\sigma(p;Q_\y)=\sigma^*(p^\nu;\y)$ provided that $\nu\ge 2h+3$,
where $h$ is the exponent for which $p^h|| C(\y)$. Our definition for
$P$, along with the condition that $\bad(\y)=\beta\mid P$, ensures
that $2h+3\le 5[(\log N)/(\log p)]$ whether $h\ge 1$ or not.  Hence
\beql{add*}
\prod_{p\le N}\sigma(p;Q_\y)=
\prod_{p\le N}\sigma^*(p^{5[(\log N)/(\log p)]};\y)=\sigma^*(P^5;\y),
\eeq
This is sufficient for the lemma, in the case that
\[\prod_{p\le N}\sigma(p;Q_\y)=0,\]
so that for the rest of the proof we may assume that 
$\sigma(p;Q_\y)>0$ for every prime $p\le N$.

In a moment we will require a form of the global product formula for 
Hilbert symbols. In general, if $q$ is any nonsingular
ternary quadratic form over $\Q$, and $\Q_v$ is a completion of $\Q$,
we write $\chi_v(q)=+1$ if $q$ is isotropic over $\Q_v$, and
$\chi_v(q)=-1$ if $q$ is anisotropic over $Q_v$. In the notation of
Serre \cite[Chapter IV, \S 2]{serre} we see that
$\chi_v(q)=\varepsilon_v(q)(-1,-d)_v$, where $d=\det(q)$ and $(-1,-d)_v$
is the Hilbert symbol. It then follows from \cite[Chapter III, Theorem
  3 \& Chapter IV, \S 3.1]{serre} that
\[\prod_v \chi_v(q)=+1,\]
the product being over all valuations of $\Q$. 

We are assuming that $\chi(\infty;\y)=+1$, whence
$\chi_{\infty}(Q_\y)=\chi(\infty;\y)=+1$. We are also assuming that
$\sigma(p;Q_\y)>0$ for every prime $p\le N$, so that $Q_\y$ is isotropic 
over $\Q_p$ for such primes. Hence $\chi_p(Q_\y)=+1$ whenever $p\le N$.
If $p\nmid C(\y)$ then again $Q_\y$ is isotropic over $\Q_p$ and 
$\chi_p(Q_\y)=+1$. If $p>N$ and $p\mid C(\y)$ then $p||C(\y)$, since 
Bad$(\y)=\beta|P$, and it follows that $\chi_p(Q_\y)=\chi(p;\y)$. The conditions
$p>N$ and $p|C(\y)$ are equivalent to $p\mid C(\y)/\beta$ when
Bad$(\y)=\beta|P$, and we deduce that
\beql{hilb}
\prod_{p\mid C(\y)/\beta}\chi(p;\y)=+1.
\eeq
Moreover our definition of $\bad(\y)$
ensures that $C(\y)/\beta$ is square-free. It follows that
\[\prod_{\substack{p>N\\ p\mid C(\y)}}\{1+\chi(p;\y)\}=
\prod_{p\mid C(\y)/\beta}\{1+\chi(p;\y)\}=
\sum_{d\mid C(\y)/\beta}\chi(d;\y).\]

Every divisor of $C(\y)/\beta$ is automatically square-free and
coprime to $P$, and conversely, any square-free factor of $C(\y)$
which is coprime to $P$ must divide $C(\y)/\beta$. Thus
values $d\le Y^{3/2}$ produce the sum $\Sigma_1(\y)$.
For larger values of $d$ we switch to the complementary divisor, guided by
Dirichlet's ``hyperbola method''. Thus we write $e=|C(\y)|/d\beta$,
so that $\chi(d;\y)=\chi(e;\y)$, by (\ref{hilb}).

 Since the
condition $d>Y^{3/2}$ is equivalent to the requirement that  
$e<|C(\y)|/(Y^{3/2}\beta)$, it now follows that
\[\mathfrak{S}(Q_{\y})=
\left\{1+O\left((\log\log 3Y)^{-1/2}\right)\right\}
\prod_{p\le N}\sigma(p;Q_\y)\left\{\Sigma_1(\y)+\Sigma_2(\y)\right\}\]
and the lemma then follows via (\ref{add*}).
\end{proof}

By definition, $\sigma^*(P^5;\y)$ only depends on $\y$
modulo $P^5$. In fact, slightly more is true. When $\y_1$ and $\y_2$ are
vectors in $\Z^2$ coprime to some modulus $m$, we say that they are
``projectively equivalent'' modulo $m$ when there is an integer
$\lambda$ (necessarily coprime to $m$) such that
$\y_1\equiv\lambda\y_2\Mod{m}$. It is then clear that $\sigma^*(P^5;\y)$
is constant on projective congruence classes modulo $P^5$. Similarly we 
see that $\chi(d;\y)$ is constant on projective congruence classes 
modulo $d$. These observations allow us to refer to $\sigma^*(m;a)$
and $\chi(m;a)$, for example, when $a\in\P^1(\Z/m\Z)$ .  

We are now ready to begin our analysis of $S(\cl{V},\beta)$ from \eqref{def:SVB}. For this we
need to recall the definition 
\[\sL(a,m) = \{ \y \in \Z^2 : \exists k \in \Z
\mbox{ such that } \y \equiv k a \bmod m\}\] 
as introduced in \S \ref{seclat}.
\begin{lemma}\label{ba}
Define
\[S_1(a,m,\Theta)=
\card\{\y\in\sL(a,m)\cap\cl{U}:\, |C(\y)|>\Theta\}.\]
There is a constant $c_1>0$ depending only on $Q_0$ and $Q_1$ such that
  \beql{form}
  S(\cl{V},\beta)=\left\{1+O\left((\log\log 3Y)^{-1/2}\right)\right\}
  \left(S^{(1)}(\cl{V},\beta)+S^{(2)}(\cl{V},\beta)\right)
  \eeq
  whenever $\beta\mid P$,  with
  \begin{eqnarray*}
  S^{(1)}(\cl{V},\beta)&=&\sum_{\substack{d\le Y^{3/2}\\ \gcd(d,P)=1,\, \mu^2(d)=1}}
  \sum_{\gcd(f,P)=1}\mu(f)\\
  &&
\sum_{\substack{a \in \P^1(\Z/P^5[d,f^2]\Z)\\ \bad_1(a)=\beta,\,\,\, [d,f^2]\mid C(a)}}
\sigma^*(P^5;a)\chi(d;a)S_1(a,P^5[d,f^2],0).
\end{eqnarray*}
and
  \begin{eqnarray*}
  S^{(2)}(\cl{V},\beta)&=&\sum_{\substack{d\le c_1 Y^{3/2}\\ \gcd(d,P)=1,\, \mu^2(d)=1}}
  \sum_{\gcd(f,P)=1}\mu(f)\\
  &&
\sum_{\substack{a \in \P^1(\Z/P^5[d,f^2]\Z)\\ \bad_1(a)=\beta,\,\,\, [d,f^2]\mid C(a)}}
\sigma^*(P^5;a)\chi(d;a)S_1(a,P^5[d,f^2], d\beta Y^{3/2}).
\end{eqnarray*}
\end{lemma}
Here we use the standard notation $[a,b]$ for the
lowest common multiple of $a$ and $b$.
The reader should note that $\sigma^*(P^5;a)$ and
$\chi(d;a)$ are well defined for  
$a \in \P^1(\Z/P^5[d,f^2]\Z)$, as are the conditions $\bad_1(a)=\beta$ and  $[d,f^2]\mid C(a)$,
when $\beta|P$.

We should also point out that care is needed with
estimates of the type $A=\{1+O(E)\}B$, where $E$ tends to zero.
In general if one has $A_j=\{1+O(E)\}B_j$ one may want to deduce that
\[\sum_{j\le J}A_j=\{1+O(E)\}\sum_{j\le J}B_j.\]
However this does not necessarily follow if the $A_j$ can vary in
sign. In our situation $A_j$ corresponds to $\mathfrak{S}(Q_{\y})$
which is always non-negative, so that the above problem
does not arise.

\begin{proof}[Proof of Lemma \ref{ba}]
Lemma \ref{de} produces a formula of the shape (\ref{form}) with
\[S^{(i)}(\cl{V},\beta)=\sum_{\substack{\y\in\cl{U}\\ \bad(\y)=\beta}}\sigma^*(P^5;\y)
\Sigma_i(\y),\;\;\;\;(i=1,2).\]
We split the sum over $\y$ into lattices $\sL(a,P^5)$, so that 
$\sigma(P^5;\y)=\sigma(P^5;a)$, for all primitive $\y$ in
$\sL(a,P^5)$. We can restrict attention to those
lattices $\sL(a,P^5)$ for which $\bad_1(a)=\beta$. Moreover, if 
we then have $\y\in\sL(a,P^5)$, the condition $\bad(\y)=\beta$
will hold if and only if $\bad_2(\y)$=1, or equivalently if $C(\y)$
has no factor $p^2$ with $p>N$ being prime. Thus
\[S^{(i)}(\cl{V},\beta)=
\sum_{\substack{a\in\P^1(\Z/P^5\Z)\\ \bad_1(a)=\beta}}\sigma^*(P^5;a)
\sum_{\gcd(f,P)=1}\mu(f)\sum_{\substack{\y\in\sL(a,P^5)\cap\,\cl{U}\\ f^2\mid C(\y)}}
\Sigma_i(\y).\]
When $i=1$, for example, we have a contribution 
\[\sum_{\substack{a\in\P^1(\Z/P^5\Z)\\ \bad_1(a)=\beta}}\sigma^*(P^5;a)
\sum_{\gcd(f,P)=1}\mu(f)\sum_{\substack{d\le Y^{3/2}\\ \gcd(d,P)=1,\, \mu^2(d)=1}}\;\;
\sum_{\substack{\y\in\sL(a,P^5)\cap\,\cl{U}\\  [d,f^2]\mid C(\y)}}\chi(d;\y).\]
However $\chi(d;\y)$ depends only on
the projective congruence class of $\y$ modulo $d$. It therefore makes
sense to split the vectors $\y$ into projective equivalence classes
modulo $P^5[d,f^2]$, so that
\[\sum_{\substack{a\in \P^1(\Z/P^5\Z)\\ \bad_1(a)=\beta}}\sigma^*(P^5;a)
\sum_{\substack{\y\in\sL(a,P^5)\cap\,\cl{U}\\ [d,f^2]\mid C(\y)}}
\chi(d;\y)\]
becomes
\[\sum_{\substack{a\in \P^1(\Z/P^5[d,f^2]\Z)\\ \bad_1(a)=\beta,\,\,\, [d,f^2]\mid C(a)}}
\sigma^*(P^5;a)\chi(d;a)S_1(a,P^5[d,f^2],0),\]
giving us the required expression for $S^{(1)}(\cl{V},\beta)$.

In a similar way the terms corresponding to $\Sigma_2(\y)$ produce
\[\sum_{\substack{a\in \P^1(\Z/P^5[d,f^2]\Z)\\ \bad_1(a)=\beta,\,\,\, [d,f^2]\mid C(a)}}
\sigma^*(P^5;a)\chi(d;a)S_1(a,P^5[d,f^2], d\beta Y^{3/2}).\]
If $ d\ge c_1 Y^{3/2}$ with a suitably large constant $c_1$ there are no points 
$\y\in\cl{U}$ with $|C(\y)|\ge d\beta Y^{3/2}$,
and the lemma follows.
\end{proof}

\subsection{Counting via conics -- Summing over $\y$}
In this section we will perform the summation over $\y$ and complete
the proof of Proposition \ref{PMB}. 
In the course of the argument
we will encounter many estimates which involve an exponent $\ep>0$
which can be arbitrarily small. For these we follow the standard
practice of allowing $\ep$ to change from one occurrence to the
next. This allows us to write $Y^\ep\log Y\ll_\ep Y^\ep$, for
example.  With this notation we observe that
\beql{n1}
P=Y^{O((\log\log 3Y)^{-1/2})}\ll_\ep Y^\ep,
\eeq
in view of our choice of $N$ and $P$.  Moreover it is immediate 
from the definition that
\beql{n2}
\sigma^*(P^5;a)\ll P^5\ll_\ep Y^\ep.
\eeq
We will use these estimates repeatedly without further comment.

We now estimate
$S^{(1)}(\cl{V},\beta)$ and $S^{(2)}(\cl{V},\beta)$
using Lemma \ref{easyL2}.
\begin{lemma}\label{apply}
Suppose that the cubic form $C(\y)$ is irreducible over $\Q$.  Then
  \[S^{(1)}(\cl{V},\beta)=
  \frac{6}{\pi^2}\sum_{\substack{d\le Y^{3/2}\\ \gcd(d,P)=1,\, \mu^2(d)=1}}
  \sum_{\substack{\gcd(f,P)=1\\ f\le Y^{1/6}}}\mu(f)
  \frac{V_1A(d,f)}{P^5[d,f^2]}+O(Y^{15/8}),\]
  and
  \[S^{(2)}(\cl{V},\beta)=
  \frac{6}{\pi^2}\sum_{\substack{d\le c_1 Y^{3/2}\\ \gcd(d,P)=1,\, \mu^2(d)=1}}
  \sum_{\substack{\gcd(f,P)=1\\ f\le Y^{1/6}}}\mu(f)
  \frac{V_2(d)A(d,f)}{P^5[d,f^2]}+O(Y^{15/8}),\]
  where $V_1=\meas(\cl{V})=H^2$,
  \[V_2(d)=
  \meas\left(\{\y\in\cl{V}: \, |C(\y)|> d\beta Y^{3/2}\}\right),\]
  and
  \[A(d,f)=    \prod_{p\mid Pdf}\left(\frac{p}{p+1}\right)
\sum_{\substack{a\in \P^1(\Z/P^5[d,f^2]\Z)\\ \bad_1(a)=\beta,\,\,\, [d,f^2]\mid C(a)}}
\sigma^*(P^5;a)\chi(d;a).\]
\end{lemma}
It is in the proof of this lemma that the condition that $C(\y)$ is
irreducible over $\Q$ arises. For the rest of this section we will
assume both that $\cl{M}(\Q)=\emptyset$ and that
$\cl{C}(\Q)=\emptyset$, so that $\rho$ must be 2, by Lemma \ref{r2}.

\begin{proof}
We present the argument for $S^{(1)}(\cl{V},\beta)$ only, the details for 
$S^{(2)}(\cl{V},\beta)$ being very similar.
We begin by showing how to reduce the range of the summations over $f$
in Lemma \ref{ba}.  
The contribution to $S^{(1)}(\cl{V},\beta)$ arising from terms with
$f\ge F$ say, is
\begin{eqnarray*}
  &\ll_\ep & Y^\ep\sum_{d\le Y^{3/2}}\;\sum_{f\ge F}\;
 \sum_{\substack{a\in \P^1(\Z/P^5[d,f^2]\Z) \\ [d,f^2]\mid C(a)}}
S_1(a,P^5[d,f^2],0)\\
&\ll_\ep & Y^\ep\sum_{d\le Y^{3/2}}\;\sum_{f\ge F}
\card\{\y\in \cl{U}:\, [d,f^2]\mid C(\y)\}.
\end{eqnarray*}
When $\y\in\cl{U}$ we have $0<|C(\y)|\ll Y^3$, so that only
values $f\ll Y^{3/2}$ can contribute. When $F\le f\ll Y^{3/2}$ we
use Lemma \ref{fests} to show that the condition $f^2\mid C(\y)$ restricts
$\y$ to lie in $O_\ep (Y^\ep)$ lattices
$\sL(a,f^2)$, each of which has determinant $f^2$. Moreover
\[\card\{\y\in\sL(a,f^2)\cap\cl{U}\}\ll 1+Y^2/f^2\]
by Heath-Brown \cite[Lemma 2]{HBsqf}. We then conclude that each value
of $f\ll Y^{3/2}$ corresponds to $O_\ep(Y^\ep(1+Y^2/f^2))$ vectors
$\y$. Since $d\mid C(\y)$ each such $\y$ produces $O_\ep(Y^{\ep})$
values of $d$, whence 
the contribution to $S^{(1)}(\cl{V},\beta)$ arising from terms with
$F\le f\ll Y^{3/2}$ is
\[\ll_\ep Y^\ep\sum_{F\le f\ll Y^{3/2}}(1+Y^2/f^2)
\ll_\ep Y^\ep(Y^{3/2}+Y^2/F).\]

For the range $f\le F$ we apply Lemma \ref{easyL2} to estimate 
$S_1(a,P^5[d,f^2],0)$. 
The main terms in Lemma \ref{apply} drop out immediately, since we will
eventually take $F=Y^{1/6}$. Hence it suffices to consider the error terms.
We will have $R\ll Y$ and $L\ll Y$, so that the error terms in Lemma
\ref{easyL2} will be  
$O_\ep(Y^{1+\ep}\lambda_1^{-1})$, with $\lambda_1$ corresponding
to $\sL(a,P^5[d,f^2])$.  As in (\ref{n2}) we have
$\sigma^*(P^5;a)\ll_\ep Y^\ep$. Thus the error terms from Lemma
\ref{easyL2} contribute to $S^{(1)}(\cl{V},\beta)$ a total
\[\ll_\ep Y^\ep(Y^{3/2}+Y^2 F^{-1})+
Y^{1+\ep} \sum_{d\le Y^{3/2}}\sum_{f\le F}
\sum_{\substack{a \in \P^1(\Z/P^5[d,f^2]\Z)\\  [d,f^2]\mid C(a)}}
\lambda_1(a,d,f)^{-1},\]
where $\lambda_1(a,d,f)$ is the first successive minimum of $\sL(a,P^5[d,f^2])$.

Next we need to know how many $a\in \P^1(\Z/P^5[d,f^2]\Z)$
have $[d,f^2]\mid C(a)$. We tackle this using
multiplicativity. We note that $\P^1(\Z/P^5\Z)$ trivially has size $O(P^6)$,
say, while $f_C([d,f^2])\ll_\ep Y^\ep$ by Lemma \ref{fests}. The number of 
relevant $a$ is therefore $O_\ep(Y^\ep)$, whence the estimate above will be
\begin{equation}\label{da}
  \ll_\ep Y^\ep(Y^{3/2}+Y^2 F^{-1})+
  Y^{1+\ep}\sum_{d\le Y^{3/2}}\sum_{f\le F}\lambda_1^{-1},
  \end{equation}
where 
\[\lambda_1=
\min_{\substack{a \in \P^1(\Z/P^5[d,f^2]\Z)\\  [d,f^2]\mid C(a)}}
\lambda(a,d,f).\]

We divide the available range for
$\lambda_1$ into dyadic intervals $(L,2L]$, so that it suffices to
consider the range making the largest contribution, and we ask how many pairs
  $d,f$ have $\lambda_1\in(L,2L]$. We have a trivial bound
$O(Y^{3/2}F)$. 

We now come to the key place that irreducibility of $C$ is used, to
deduce that the values $\lambda_1$ cannot be small too often. 
If $\y$ is a vector with  
$||\y||_2=\lambda_1\in(L,2L]$ and $C(\y)\not=0$, then there are $O_\ep(Y^\ep)$
  pairs $d,f$ with $[d,f^2]\mid C(\y)$.
Hence the total number of pairs
 $d,f$ for which the corresponding $\lambda_1$ is in the range
 $(L,2L]$ will be  
 $O_\ep(Y^\ep L^2)$, provided that $C(\y)$ never vanishes for 
 $\y\not=\mathbf{0}$. The bound (\ref{da}) then becomes
 \begin{eqnarray*}
   &\ll_\ep & Y^\ep(Y^{3/2}+Y^2 F^{-1})+
 Y^{1+\ep}\min\left(Y^{3/2}F\,,\,L^2\right)L^{-1}\\
&\ll_\ep & Y^\ep(Y^{3/2}+Y^2 F^{-1})+
 Y^{1+\ep}(Y^{3/2}F)^{1/2}(L^2)^{1/2}L^{-1}\\
 &=&Y^\ep(Y^{3/2}+Y^2 F^{-1})+Y^{7/4+\ep}F^{1/2}.
 \end{eqnarray*}
This suffices for the lemma, on choosing $F=Y^{1/6}$ and taking $\ep$
sufficiently small.
\end{proof}

We next consider the summations over $f$ occurring in Lemma
\ref{apply}.
\begin{lemma}\label{fs}
  When $d\ll Y^{3/2}$ is square-free and coprime to $P$ we have
  \begin{eqnarray*}
  \lefteqn{\sum_{\substack{\gcd(f,P)=1\\ f\le Y^{1/6}}}\mu(f)
  \frac{A(d,f)}{[d,f^2]}}\\
  &=& A_1(P,\beta)\frac{\Pi}{\Pi_{d}}\sum_{h\mid d}\frac{\mu(h)}{dh}
\sum_{\substack{a\in \P^1(\Z/dh\Z) \\ dh\mid C(a)}}
\chi(d;a)+O(d^{-1}Y^{-1/8}),
\end{eqnarray*}
with
\[A_1(P,\beta)=\prod_{p\mid P}\left(\frac{p}{p+1}\right)
\sum_{\substack{a\in \P^1(\Z/P^5\Z)\\ \bad_1(a)=\beta}}\sigma^*(P^5;a),\]
\[\Pi=\prod_{p>N}\left(1-\frac{f_C(p)}{p(p+1)}\right)
\;\;\;\mbox{and}\;\;\;
\Pi_d=\prod_{p\mid d}\left(1+\frac{p-f_C(p)}{p^2}\right).\]
\end{lemma}

\begin{proof}
We split up the sum over $f$ according to the value of $(d,f)=h$, say,
and we write $f=hg$, whence
\[\sum_{\substack{\gcd(f,P)=1\\ f\le Y^{1/6}}}\mu(f)\frac{A(d,f)}{[d,f^2]}
=\sum_{h\mid d}\mu(h)\sum_{\substack{\gcd(g,dP)=1\\ g\le Y^{1/6}/h}}\mu(g)
\frac{A(d,hg)}{[d,h^2g^2]}.\]
The Chinese Remainder Theorem allows us to view $\P^1(\Z/P^5[d,f^2]\Z)$
as a product $\P^1(\Z/P^5\Z)\times \P^1(\Z/dh\Z)\times \P^1(\Z/g^2\Z)$, whence
$A(d,f)=A_1(P,\beta)A_2(d,h)A_3(g)$ with
\[A_2(d,h)=\prod_{p\mid d}\left(\frac{p}{p+1}\right)
\sum_{\substack{a\in \P^1(\Z/dh\Z) \\ dh\mid C(a)}}\chi(d;a),\]
and
\[A_3(g)=\prod_{p\mid g}\left(\frac{p}{p+1}\right)
\sum_{\substack{a\in \P^1(\Z/g^2\Z)\\ g^2\mid C(a)}}1.\]
We therefore find that
\beql{mtis}
\sum_{\substack{\gcd(f,P)=1\\ f\le Y^{1/6}}}\mu(f)\frac{A(d,f)}{[d,f^2]}
=A_1(P,\beta)\sum_{h\mid d}\mu(h)\frac{A_2(d,h)}{dh}
\sum_{\substack{\gcd(g,dP)=1\\ g\le Y^{1/6}/h}}\mu(g)\frac{A_3(g)}{g^2}.
\eeq
When $\gcd(g,P)=1$ the function $A_3(g)$ is multiplicative, with
\[A_3(p)=\frac{p}{p+1}\sum_{\substack{a\in \P^1(\Z/p^2\Z)\\ p^2\mid C(a)}}1
=\frac{p}{p+1}f_C(p^2)=\frac{p}{p+1}f_C(p),\]
by Lemma \ref{fests}, since $p>N$. Lemma \ref{fests}
then shows that $A_3(g)\ll_\ep g^{\ep}$ for any fixed $\ep>0$, so that
\beql{etwb}
\sum_{\substack{\gcd(g,dP)=1\\ g\le Y^{1/6}/h}}\mu(g)\frac{A_3(g)}{g^2}=
\sum_{\substack{g=1\\ \gcd(g,dP)=1}}^\infty \mu(g)\frac{A_3(g)}{g^2}
+O_\ep(hY^{\ep-1/6}).
\eeq
The infinite sum factorizes as
\[\prod_{\substack{p>N\\ p\nmid d}}\left(1-\frac{f_C(p)}{p(p+1)}\right).\]
The main term in our expression for  (\ref{mtis}) is therefore
\[A_1(P,\beta)\prod_{p\mid d}\left(\frac{p}{p+1}\right)
\prod_{\substack{p>N\\ p\nmid d}}\left(1-\frac{f_C(p)}{p(p+1)}\right)
\sum_{h\mid d}\frac{\mu(h)}{dh}
\sum_{\substack{a\in \P^1(\Z/dh\Z) \\ dh\mid C(a)}}\chi(d;a).\]
The two products give us $\Pi/\Pi_d$, so that we obtain the leading term in
Lemma \ref{fs}. 

The error term in (\ref{etwb}) contributes a total
\[\ll_\ep A_1(P,\beta)\sum_{h\mid d}\frac{|A_2(d,h)|}{dh}hY^{\ep-1/6},\]
in which the bounds (\ref{n1}) and (\ref{n2}) show that
$A_1(P,\beta)\ll_\ep Y^\ep$, and
\[A_2(d,h)\ll_\ep Y^\ep f_C(dh)\ll_\ep Y^\ep.\]
It follows that the error term is $O_\ep(d^{-1}Y^{\ep-1/6})$, and the
lemma follows on choosing $\ep$ appropriately.
\end{proof}

We can now evaluate $S^{(1)}(\cl{V},\beta)$ and
$S^{(2)}(\cl{V},\beta)$. In the course of the argument we will
use our hypothesis that $|C(\y)|\ge\eta Y^3$ throughout the square
$\cl{V}$. We will continue to assume that $\cl{M}$ and $\cl{C}$ have no
rational points.
\begin{lemma}\label{hds}
Suppose that $\beta\mid P$. Then
\[S^{(1)}(\cl{V},\beta)=
\frac{6}{\pi^2}\frac{A_1(P,\beta)}{P^5}\Pi V_1
\left\{1+O((\log N)^{-1})\right\}+O(Y^{15/8}\log Y)\]
and
\[S^{(2)}(\cl{V},\beta)=
\frac{6}{\pi^2}\frac{A_1(P,\beta)}{P^5}\Pi V_1
\left\{1+O((\log N)^{-1})\right\}+O_\eta(Y^{15/8}\log Y).\]
\end{lemma}

\begin{proof}
If we set $e=d/h$ we will have
  \[\sum_{\substack{a\in \P^1(\Z/dh\Z)\\ dh\mid C(a)}}\chi(d;a)=
\left\{\sum_{\substack{a\in \P^1(\Z/e\Z)\\ e\mid C(a)}}\chi(e;a)\right\}
\left\{\sum_{\substack{a\in \P^1(\Z/h^2\Z)\\ h^2\mid C(a)}}\chi(h;a)\right\}\]
by multiplicativity. We now observe that
\[\sum_{\substack{a\in \P^1(\Z/h^2\Z)\\ h^2\mid C(a)}}\chi(h;a)=
\sum_{\substack{a\in \P^1(\Z/h\Z)\\ h\mid C(a)}}\chi(h;a),\]
for square-free $h$ coprime to $P$, since each $a\in \P^1(\Z/h\Z)$ with 
$h\mid C(a)$ lifts to a unique point $a'\in \P^1(\Z/h^2\Z)$ with 
$h^2\mid C(a')$. It follows that
\[\left\{\sum_{\substack{a\in \P^1(\Z/e\Z)\\ e\mid C(a)}}\chi(e;a)\right\}
\left\{\sum_{\substack{a\in \P^1(\Z/h\Z)\\ h\mid C(a)}}\chi(h;a)\right\}=
\sum_{\substack{a\in \P^1(\Z/d\Z)\\ d\mid C(a)}}\chi(d;a),\]
whence
\[\sum_{h|d}\frac{\mu(h)}{dh}
\sum_{\substack{a\in \P^1(\Z/dh\Z)\\ dh\mid C(a)}}\chi(d;a)=
\frac{\phi(d)}{d^2}
\sum_{\substack{a\in \P^1(\Z/d\Z)\\ d\mid C(a)}}\chi(d;a).\]
By Lemma \ref{LC} we have 
\[\sum_{\substack{a\in \P^1(\Z/p\Z)\\ p\mid C(a)}}\chi(p;a)=f_M(p)-1,\]
so that
\begin{eqnarray*}
\lefteqn{\sum_{\substack{d\le Y^{3/2}\\ \gcd(d,P)=1,\, \mu^2(d)=1}}V_1\Pi_{d}^{-1}
\sum_{h|d}\frac{\mu(h)}{dh}
\sum_{\substack{a\in \P^1(\Z/dh\Z)\\ dh\mid C(a)}}\chi(d;a)}\hspace{2cm}\\
&=&
\sum_{\substack{d\le Y^{3/2}\\ \gcd(d,P)=1,\, \mu^2(d)=1}}V_1\Pi_{d}^{-1}
\frac{\phi(d)}{d^2}\prod_{p\mid d}(f_M(p)-1)\\
&=& \sum_{d\le Y^{3/2}}V_1f(d;N)/d
\end{eqnarray*}
in the notation of Lemma \ref{earlier}, and similarly for the sum with
$V_1$ replaced by $V_2(d)$.  For the sum involving $V_1$ we can 
now apply Lemma \ref{earlier}, giving
\[\sum_{\substack{d\le Y^{3/2}\\ \gcd(d,P)=1,\, \mu^2(d)=1}}V_1\Pi_{d}^{-1}
\frac{\phi(d)}{d^2}\prod_{p\mid d}(f_M(p)-1)=
\left(1+O((\log N)^{-1})\right)V_1.\]

When $V_1$ is replaced by $V_2(d)$ we use summation by parts coupled
with Lemma \ref{earlier}. We have
\[\sum_{d\le D}f(d;N)/d=1+O((\log N)^{-1})\]
for $N\le(\log D)^2$.
We write $D_0=Y^{3/2}P^{-1}\eta$ for convenience, and note that if
$d\le D_0$ then  
\[|C(\y)|\ge\eta Y^3\ge d\beta Y^{3/2}\]
on $\cl{V}$, so that $V_2(d)=V_1$. Since $N$, given by
(\ref{Ndef}), will satisfy $N\le(\log D_0)^2$ when $Y\gg_\eta 1$, it
follows that 
\[\sum_{\substack{d\le D_0\\ \gcd(d,P)=1,\, \mu^2(d)=1}}V_2(d)\Pi_{d}^{-1}
\frac{\phi(d)}{d^2}\prod_{p\mid d}(f_M(p)-1)=
V_1+O(V_1(\log N)^{-1}).\]
Moreover,  by partial summation we have
\[\sum_{\substack{D_0<d\le c_1Y^{3/2}\\ \gcd(d,P)=1,\, \mu^2(d)=1}}V_2(d)\Pi_{d}^{-1}
\frac{\phi(d)}{d^2}\prod_{p\mid d}(f_M(p)-1)\ll
V_1(\log N)^{-1},\]
and we conclude that
\[\sum_{\substack{d\le c_1Y^{3/2}\\ \gcd(d,P)=1,\, \mu^2(d)=1}}V_2(d)\Pi_{d}^{-1}
\frac{\phi(d)}{d^2}\prod_{p\mid d}(f_M(p)-1)=\left\{1+O((\log N)^{-1})\right\}V_1.\]
The lemma now follows. 
\end{proof}

Our next move is to sum over $\beta\mid P$.
\begin{lemma}\label{sbp}
  We have
  \[\sum_{\beta\mid P}S(\cl{V},\beta)=
  2H^2\mathfrak{S}_S+O_{\eta}\left(H^2(\log\log 3Y)^{-1/2}\right)
  +O_\eta(Y^{19/10}).\]  
\end{lemma}
\begin{proof}
Since
\[\Pi=\prod_{p>N}\left(1-\frac{f_C(p)}{p(p+1)}\right)=1+O(N^{-1})
=1+O\left((\log\log 3Y)^{-1/2}\right).\]
we deduce from Lemmas \ref{ba} and \ref{hds} that
\[\sum_{\beta\mid P}S(\cl{V},\beta)=
\left\{1+O_\eta\left((\log\log 3Y)^{-1/2}\right)\right\}
\frac{12}{\pi^2}V_1 S_P+O_\eta(Y^{19/10})\]
with
\[S_P=\sum_{\beta\mid P}\frac{A_1(P,\beta)}{P^5}.\]
It follows from the definition of $A_1(P,\beta)$ in Lemma \ref{fs}  that
\[S_P=P^{-5}\prod_{p\mid P}\left(\frac{p}{p+1}\right)
\sum_{\substack{a\in \P^1(\Z/P^5\Z)\\ \bad_1(a)\mid P}}\sigma^*(P^5;a).\]
The condition
$\bad_1(a)\mid P$ is equivalent to the requirement that
$p^{e+1}\nmid C(a)$ for every $p\le N$, where $e=e_p=[(\log N)/(\log p)]$.
We may therefore factor the expression on the right as
\[\prod_{p\le N} \left(p^{-5e}\frac{p}{p+1}
\sum_{\substack{a\in \P^1(\Z/p^{5e}\Z)\\ p^{e+1}\nmid C(a)}}\sigma^*(p^{5e};a)
\right)\]
We now wish to remove the condition $p^{e+1}\nmid C(a)$. According
to Lemmas \ref{gb} and \ref{fests} we have
\begin{eqnarray*}
\sum_{\substack{a\in \P^1(\Z/p^{5e}\Z)\\ p^{e+1}\mid C(a)}}\sigma^*(p^{5e};a)
&\ll& e\card\{a\in \P^1(\Z/p^{5e}\Z): p^{e+1}\mid C(a)\} \\
&\ll &  e p^{4e-1}f_C(p^{e+1})\\
&\ll&  ep^{4e-1}.
\end{eqnarray*}
It follows that
\[\sum_{\substack{a\in \P^1(\Z/p^{5e}\Z)\\ p^{e+1}\nmid C(a)}}\sigma^*(p^{5e};a)
=\sum_{a\in \P^1(\Z/p^{5e}\Z)}\sigma^*(p^{5e};a)+O(ep^{4e-1}),\]
whence
\begin{eqnarray*}
\lefteqn{p^{-5e}\frac{p}{p+1}
\sum_{\substack{a\in \P^1(\Z/p^{5e}\Z)\\ p^{e+1}\nmid C(a)}}\sigma^*(p^{5e};a)
}\hspace{2cm}\\
&=&\frac{1}{p^{5e}\phi(p^{5e})}\frac{p}{p+1}
\sum_{\substack{\y\in(\ZZ/p^{5e}\ZZ)^2\\ \gcd(\y,p)=1}}\sigma^*(p^{5e};\y)
+O(ep^{-e-1})\\
&=&\frac{p-1}{p+1}\frac{f_S(p^{5e})}{(p^{5e})^2}+O(ep^{-e-1}),
\end{eqnarray*}
in view of (\ref{ssd}) and (\ref{af1}). Lemma \ref{tl} then shows that
\[p^{-5e}\frac{p}{p+1}
\sum_{\substack{a\in \P^1(\Z/p^{5e}\Z)\\ p^{e+1}\nmid C(a)}}\sigma^*(p^{5e};a)
=\frac{p-1}{p+1}\varpi_p\left(1+O(ep^{-e-1})\right).\]
We now have
\begin{eqnarray*}
\lefteqn{\sum_{\beta\mid P}S(\cl{V},\beta)}\\
&=&
\left\{1+O_{\eta}\left((\log\log 3Y)^{-1/2}\right)\right\}
\frac{12}{\pi^2}V_1
\prod_{p\le N}\frac{1-1/p}{1+1/p}\varpi_p\left(1+O(ep^{-e-1})\right)\\
&&\hspace{1cm}+O_{\eta}(Y^{15/8})
\end{eqnarray*}
with $e=e_p=[(\log N)/(\log p)]$ as
before. We now observe that
\[\prod_{p\le N}\frac{1}{1-1/p^2}=
\frac{\pi^2}{6}\left(1+O(N^{-1})\right),\]
and that
\begin{eqnarray*}
\prod_{p\le N}\left(1+O(e_pp^{-e_p-1})\right)&=&
\exp\left\{\sum_{p\le N} O\left(e_p p^{-e_p-1}\right)\right\}\\
&=&\exp\left\{\sum_{p\le N} O\left(\frac{\log N}{\log p} N^{-1}\right)\right\}\\
&=&\exp\left\{O((\log N)^{-1})\right\}\\
&=&1+O((\log N)^{-1}).
\end{eqnarray*}
It follows that the leading term is
\[2V_1\left\{1+O_\eta\left((\log\log 3Y)^{-1/2}\right)\right\}
 \prod_{p\le N}\left(1-p^{-1}\right)^2\varpi_p.\]
However we have
\[\prod_{p\le N}\left(1-p^{-1}\right)^2\varpi_p=
\prod_{p\le N}\tau_p
=\{1+O((\log N)^{-1})\}\mathfrak{S}_S\]
by Lemma \ref{lem:f2}, and since $V_1=H^2$ the main term becomes
\[ 2H^2\mathfrak{S}_S+ O_\eta\left(H^2(\log\log 3Y)^{-1/2}\right).\]
The lemma then follows.
\end{proof}

Finally we handle the case in which $\bad(\y)$ does not divide $P$.
\begin{lemma}\label{second}
For $Y(\log\log 3Y)^{-1}\le H\le Y$ we have
  \[\sum_{\substack{\y\in\cl{U}\\ \bad(\y)\, \nmid\, P}}
\mathfrak{S}(Q_{\y})\ll H^2(\log Y)^{-1/5}.\]
\end{lemma}
\begin{proof}
  If $\bad(\y)\,\nmid\, P$ there is either a prime $p$ such that 
$p^{1+[(\log N)/(\log p)]}\mid C(\y)$ with $p\le N$, or a prime $p>N$
  for which
$p^2\mid C(\y)$. According to Lemma \ref{mfSb} and Proposition \ref{kAL} we have
\[\sum_{\substack{\y\in\cl{U}\\ p^e|C(\y)}}\mathfrak{S}(Q_{\y})\ll 
\sum_{\substack{\y\in\cl{U}\\ p^e|C(\y)}}\\ \kappa(\y)\ll
H^2p^{-3e/4}\]
so long as  $p^e\le Y^{1/5760}$. Thus if we sum for $p\le N$ with
$e=1+[(\log N)/(\log p)]$ we get a 
contribution $O(H^2N^{-1/4})$, while if we take $N<p\le Y^{1/11520}$
with $e=2$ we get a total contribution $O(H^2N^{-1/2})$.  These two
error terms are 
satisfactory for the lemma.

For the alternative case in which $Y^{1/5760}\le p^2\ll Y^3$ we have
$\mathfrak{S}(Q_\y)\ll_\ep Y^\ep$ for any $\ep>0$, by Lemma
\ref{mfSb}. Moreover the primitive 
integer vectors $\y$ for which $p^2|C(\y)$ belong to one of $f_C(p^2)$
lattices $\sL(a,p^2)$ of determinant $p^2$. Each such lattice contains
$O(Y^2p^{-2}+1)$ 
primitive vectors in $\mathcal{V}$, by Heath-Brown \cite[Lemma
  2]{HBsqf}. Using the bound  
$f_C(p^2)\ll 1$ from Lemma \ref{fests} we therefore have the 
estimate
\[\sum_{\substack{\y\in\cl{U}\\ p^2|C(\y)}}\mathfrak{S}(Q_{\y})
\ll_\ep Y^{2+\ep}p^{-2}+Y^\ep,\]
and the lemma follows, on summing over the relevant primes $p$ and 
choosing $\ep$ small enough.
\end{proof}

\subsection{Proofs of Proposition \ref{PMB} and Proposition \ref{ultgoal}}

Everything is now in place to complete the proof of Proposition \ref{PMB}. 
According to (\ref{dded1}) it follows from Lemmas \ref{sbp} and
\ref{second} that 
\[\sum_{\y\in\cl{U}}\mathfrak{S}(Q_\y)=2H^2\mathfrak{S}_S+
  O_{\eta}\left(H^2(\log\log 3Y)^{-1/2}\right)\]  
when $Y(\log\log 3Y)^{-1}\le H\le Y$.  We now apply Lemma \ref{sib}, noting 
that $\sigma_\infty(Q_\y)\ll_\eta Y^{-1} $
  by Lemma \ref{pdsi}. This yields
\begin{eqnarray*}
  \lefteqn{\sum_{\y\in\cl{U}}W_2(Y^{-1}\y)\Main(X,Q_\y)}\hspace{1cm}\\
  &=&
\tfrac12 X\sum_{\y\in\cl{U}}W_2(Y^{-1}\y)\sigma_\infty(Q_\y)\mathfrak{S}(Q_\y)\\
&=& X
\left\{\int_{\cl{V}}W_2(Y^{-1}\y)\sigma_\infty(Q_\y)\d y_0 \d
y_1\right\}
\mathfrak{S}_S\\
&&\hspace{1cm}\mbox{}
+O_\eta\left(XY^{-1}H^2\{(\log\log 3Y)^{-1/2}+(H/Y)^{1/5}\}\right),
\end{eqnarray*}
which suffices for Proposition \ref{PMB}.  \qed
\bigskip

We begin the proof of Proposition \ref{ultgoal} by recalling that
$W_2(Y^{-1}\y)$ is supported in the region
$\tfrac12 Y\le||\y||_\infty\le\tfrac52Y$,
in other words, in   
$[-\tfrac52 Y, \tfrac52 Y]^2\setminus [-\tfrac12 Y, \tfrac12 Y]^2$. We
proceed to cover this set with squares
$\cl{V}=(Y_0,Y_0+H]\times(Y_1,Y_1+H]$, with $H=Y/n$ for some positive
integer chosen so that 
\[\frac{Y}{\log\log 3Y}\ll H\ll \frac{Y}{\log\log 3Y}.\]
There will be $O(Y^2 H^{-2})$ such squares, which we describe as
`good' if $|C(\y)|\ge\eta Y^3$ throughout $\cl{V}$, and `bad'
otherwise. It follows from Propositions \ref{cdiff} and  \ref{PMB}
that 
\begin{eqnarray*}
\sum_{\cl{V}\,\mathrm{good}}\;\;\sum_{\y\in\cl{U}}W_2(Y^{-1}\y)N(X,Q_\y)
&=& X\left\{\sum_{\cl{V}\,\mathrm{good}}\;
\int_{\cl{V}}W_2(Y^{-1}\y)\sigma_\infty(Q_\y)\d y_0 \d y_1\right\}   
\mathfrak{S}_S
 \\&&\hspace{1cm}\mbox{}
+O_\eta\left(XY(\log\log 3Y)^{-1/5}\right),
\end{eqnarray*}
If $\cl{V}$ is a `bad' square, then $|C(\y_0)|\le\eta Y^3$ for some
$\y_0\in\cl{V}$, whence $C(\y)\ll \eta Y^3+HY^2$ throughout
$\cl{V}$. We may therefore estimate the contribution from all the bad
squares by taking $\delta\ll \eta+HY^{-1}$ in Lemma \ref{BPCP+}. We
then deduce that 
\[\sum_{\cl{V}\,\mathrm{bad}}\;\;\sum_{\y\in\cl{U}}W_2(Y^{-1}\y)N(X,Q_\y)
\ll Y^2+XY\LL^{-1}+\eta^{2/3}XY+XY(\log\log 3Y)^{-2/3},\]
since $\rho=2$ now.  We are assuming that $Y\le X\LL^{-1}$ for
Proposition \ref{ultgoal}, 
whence we may conclude that 
\begin{eqnarray*}
\sum_{\mathrm{all}\,\cl{V}}\;\;\sum_{\y\in\cl{U}}W_2(Y^{-1}\y)N(X,Q_\y)
&=& X\left\{\sum_{\cl{V}\,\mathrm{good}}\;\;
\int_{\cl{V}}W_2(Y^{-1}\y)\sigma_\infty(Q_\y)\d y_0 \d y_1\right\}  
\mathfrak{S}_S
 \\&&\hspace{3mm}\mbox{}
+O_\eta\left(XY(\log\log 3Y)^{-1/5}\right)+O(\eta^{2/3}XY).
\end{eqnarray*}
The sum on the left is just $S(X,Y)$, and we proceed to show that
\beql{RB}
\sum_{\cl{V}\,\mathrm{bad}}\;\;
\int_{\cl{V}}W_2(Y^{-1}\y)\sigma_\infty(Q_\y)\d y_0\d y_1\ll
\{(H/Y)^{2/3}+\eta^{2/3}\}Y.
\eeq
To do this we note as above the the bad squares $\cl{V}$ are contained
in the region  
where $|C(\y)|\ll (\eta+H/Y)Y^3$.  Moreover, as in the proof of
(\ref{ndd}),  the set where 
$\tfrac12 Y\le||\y||_\infty\le \tfrac52 Y$ and $\tfrac12
C_0\le|C(\y)|\le C_0$ has measure $O(C_0 Y^{-1})$. By Lemma \ref{pdsi}
we will have  
\[\sigma_\infty(Q_\y)\ll Y^{-1}\log (2+Y^3/C_0)\ll Y^{-1}(Y^3/C_0)^{1/3},\]
say, on such a set. Thus the range
$\tfrac12 C_0\le|C(\y)|\le C_0$ contributes 
\[\ll C_0 Y^{-1}.Y^{-1}(Y^3/C_0)^{1/3}=C_0^{2/3}Y^{-1}.\]
We can  then sum for dyadic values of $C_0\ll(\eta+H/Y)Y^3$, producing
the required bound (\ref{RB}). The error terms are now acceptable for
Proposition \ref{ultgoal}, and it remains to observe that 
\[\int_{\R^2}W_2(Y^{-1}\y)\sigma_\infty(Q_\y)\d y_0\d y_1=
Y\int_{\R^2}W_2(\z)\sigma_\infty(Q_\z)\d z_0\d z_1\]
in view of the fact that
$\sigma_\infty(Q_{Y\z})=Y^{-1}\sigma_\infty(Q_\z)$. This completes the
proof of Proposition \ref{ultgoal}.

\subsection{Assembling the pieces}\label{Assem}

In this section we will complete the proof of Theorem \ref{x+}, and hence of
Theorem \ref{thm:main}. We continue with our assumption that neither $\cl{M}$ 
nor $\cl{C}$ has rational points.

We begin by giving a quantitative form of (\ref{quant}).
\begin{lemma}\label{newdensity}
 Let $T>0$ and
 \[K_T(t)=T\max\{1-T|t|\,,\,0\}.\] 
Then  
  \[\sigma_\infty(Q_\y;W_3)=
  \int_{\R^3}K_T(Q_\y(\x))W_3(\x)\d x_0\d x_1\d x_2+O(|C(\y)|^{-1/2}T^{-1/2}),\]
  with an implied constant that may depend on $W_3$.
\end{lemma}
\begin{proof}
 Since
  \[\widehat{K_T}(\theta)=\left(\frac{\sin\pi\theta/T}{\pi\theta/T}\right)^2\]
  we have
  \begin{eqnarray*}
  \lefteqn{ \int_{\R^3}K_T(Q_\y(\x))W_3(\x)\d x_0\d x_1\d x_2}\\
  &=&\int_{\R^3}W_3(\x)
  \int_{-\infty}^\infty\left(\frac{\sin\pi\theta/T}{\pi\theta/T}\right)^2
  e(-\theta Q_\y(\x))\d \theta\, \d x_0\d x_1\d x_2.
  \end{eqnarray*}
  The repeated integral converges absolutely so we may switch the
  order of integration, producing
  \[ \int_{\R^3}K_T(Q_\y(\x))W_3(\x)\d x_0\d x_1\d x_2=
  \int_{-\infty}^\infty\left(\frac{\sin\pi\theta/T}{\pi\theta/T}\right)^2
  J(\theta,Q_\y;W_3)\d \theta,\]
  where
  \[J(\theta,Q_\y;W_3)=\int_{\R^3}W_3(\x)e(-\theta Q_\y(\x))\d x_0\d x_1\d x_2\]
  in the notation (\ref{Jdef}).   However
  \[\left(\frac{\sin\pi\theta/T}{\pi\theta/T}\right)^2=
  1+O(\min\{1,T^{-2}\theta^2\}),\]
  while (\ref{JB}) yields
  \[J(\theta,Q_\y;W_3)\ll |\theta|^{-3/2}|C(\y)|^{-1/2},\]
  where the implied constant may depend on $W_3$. We therefore deduce that
  \begin{eqnarray*}
  \lefteqn{\int_{-\infty}^\infty
  \left|\left(\frac{\sin\pi\theta/T}{\pi\theta/T}\right)^2-1\right|.\,
  |J(\theta,Q_\y;W_3)|\d \theta}\\
  &\ll& |C(\y)|^{-1/2}\int_{-\infty}^\infty 
  \min\{1,T^{-2}\theta^2\}|\theta|^{-3/2}\d \theta\\
  &\ll& |C(\y)|^{-1/2}T^{-1/2},
  \end{eqnarray*}
  and the result follows.
\end{proof}

We can now deduce the following corollary.
\begin{lemma}\label{nd2}
We have
\begin{eqnarray*}
\lefteqn{\int_{\R^2} W_2^{(*)}(\y)\sigma_\infty(Q_\y;W_3)\d y_0\d y_1}\\
&=&\lim_{T\to\infty}
\int_{\R^2\times\R^3}K_T(Q_\y(\x))W_3(\x)W_2^{(*)}(\y)
\d y_0\d y_1\d x_0\d x_1\d x_2
\end{eqnarray*}
for any bounded measurable function $W_2^{(*)}$ supported in
$[-\tfrac52,\tfrac52]^2\setminus [-\tfrac12,\tfrac12]^2$.
\end{lemma}
\begin{proof}
Since $C(\y)$ has no repeated factor we have
\[\int_{\tfrac{1}{2}\le||\y||_\infty\le \tfrac52}
\frac{\d y_0\d y_1}{|C(\y)|^{1/2}}\ll 1.\]
It follows from Lemma \ref{newdensity} that
\begin{eqnarray*}
\lefteqn{\int_{\R^2} W_2^{(*)}(\y)\sigma_\infty(Q_\y;W_3)\d y_0\d y_1}\\
&=&
\int_{\R^2}W_2^{(*)}(\y) \int_{\R^3}K_T(Q_\y(\x))W_3(\x)
\d x_0\d x_1\d x_2\d y_0\d y_1
+O(T^{-1/2}).
\end{eqnarray*}
The repeated integral on the right is absolutely convergent and may
therefore be written as a double integral, and the result follows. 
\end{proof}

We are now ready to remove the weight $W_2(\y)$, replacing it by the
characteristic function of the set on which $1\le||\y||_\infty\le
1+\zeta$, for some $\zeta\in(0,\tfrac12)$. 
\begin{lemma}\label{remW2}
Let $\zeta\in(0,\tfrac12)$ be given, and write
\[S^{(0)}(X,Y)=\sum_{\x\in\Z^3_{\mathrm{prim}}}
\;\;\sum_{\substack{\y\in\Z^2_{\mathrm{prim}}\\ y_0Q_0(\x)+y_1Q_1(\x)=0\\
Y< ||\y||_\infty\le (1+\zeta)Y}}W_3(X^{-1}\x),\]
similarly to the definition in Lemma \ref{rw2}. Assume that
neither $\cl{M}$ nor $\cl{C}$ has rational points.
Then if $Y\le X\mathcal{E}^{-8}$ we have
\begin{eqnarray*}
S^{(0)}(X,Y)&=& XY
\left\{\int_{1\le||\y||_\infty\le 1+\zeta}\sigma_\infty(Q_\y;W_3)\d y_0\d y_1\right\}
\mathfrak{S}_S\nonumber\\
&&\hspace{1cm}\mbox{}+O(\zeta^2 XY)+
O_{\zeta,\eta}\left(XY(\log\log 3Y)^{-1/5}\right)\\
&&\hspace{2cm}\mbox{}+O_\zeta(\eta^{2/3}XY).
\end{eqnarray*}
\end{lemma}

\begin{proof}
The procedure will be very close to that which we
used for the proof of Lemma \ref{rw2} so we will be brief. We start
from Proposition \ref{ultgoal}, which we apply with weights
$W_2^{(+)}(\u)$ and $W_2^{(-)}(\u)$ supported on
$1-\zeta^2\le||\u||_\infty\le 1+\zeta+\zeta^2$ and
$1\le||\u||_\infty\le 1+\zeta$ respectively. These will be fixed once
$\zeta$ is chosen. The reader will recall that the error terms in
Proposition \ref{ultgoal} have order constants that may depend on
the weights $W_2$
and $W_3$. Hence, once we take $W_2$ as either $W_2^{(+)}$ or
$W_2^{(-)}$, the order constants will depend on $\zeta$ and $W_3$ (as
well as other parameters). We then see that it is enough to show that 
\beql{zb}
\int_{\R^2} \{W_2^{(+)}(\y)-W_2^{(-)}(\y)\}\sigma_\infty(Q_\y;W_3)\d y_0\d y_1
\ll \zeta^2.
\eeq
According to Lemma \ref{nd2} the range
$1-\zeta^2\le||\y||_\infty\le 1$ makes a contribution 
\[\ll \limsup_{T\to\infty} \;\; T\mathrm{meas}(S_T)\]
where $S_T$ is the set of pairs $(\x,\y)\in\R^3\times\R^2$ satisfying 
$\tfrac12\le||\x||_\infty\le \tfrac52$ and
$1-\zeta^2\le||\y||_\infty\le 1$, for which
$|y_0Q_0(\x)+y_1Q_1(\x)|\le T^{-1}$. It therefore suffices to show that
\beql{m+}
\mathrm{meas}(S_T)\ll \zeta^2/T.
\eeq
We investigate the measure of the
available $\y$ for each given vector $\x$. Without loss of generality we let
$|y_0|\ge|y_1|$, say, whence  $|y_0|\in[1-\zeta^2,1]$. 
Suppose firstly that $|Q_1(\x)|\le\tfrac14 |Q_0(\x)|$. Then 
\[T^{-1}\ge |y_0Q_0(\x)+y_1Q_1(\x)|\ge (1-\zeta^2)|Q_0(\x)|-\tfrac14 |Q_0(\x)|\]
\[\gg |Q_0(\x)|=||(Q_0(\x),Q_1(\x))||_\infty=h_Q(\x).\]
In this case we merely observe firstly that the measure of the
relevant vectors $\y$ is $O(\zeta^2)$, and then that 
\[\zeta^2\ll\frac{\zeta^2}{Th_Q(\x)}.\]
In the alternative case one has $|Q_1(\x)|\ge\tfrac14 |Q_0(\x)|$. This
time we observe that $|y_0|$ is restricted to an interval
$[1-\zeta^2,1]$ of length $\zeta^2$, after which  the constraint 
$|y_0Q_0(\x)+y_1Q_1(\x)|\le T^{-1}$ confines $y_1$ to an interval of length 
\[\ll T^{-1}|Q_1(\x)|^{-1}\ll T^{-1}h_Q(\x)^{-1}.\]
It follows that $\y$ is restricted to a set of measure $O(\zeta^2/Th_Q(\x))$
whether or not $|Q_1(\x)|\le\tfrac14 |Q_0(\x)|$. Of course we have the
same conclusion when $|y_0|\le|y_1|$. We then see from Lemma \ref{QB}
that the set $S_T$ has measure $O(\zeta^2/T)$, as required for
(\ref{m+}). Naturally, the range  
$1+\zeta-\zeta^2\le||\y||_\infty\le 1+\zeta$ may be treated in exactly
the same way, and the claim  (\ref{zb}) follows. 
\end{proof}

We plan to use Lemma \ref{nd2} to give an alternative formulation
of Lemma \ref{remW2}, but first it is convenient to establish the
following result. 
\begin{lemma}\label{useful}
Define
\[J_T(a_0,a_1)=\int_{1\le||\y||_\infty\le 1+\zeta}K_T(a_0y_0+a_1y_1) \d y_0\d y_1.\]
Then if $T\ge 1$ and $\max(|a_0|,|a_1|)=|a_1|$ we have
\[J_T(a_0,a_1)=
\left\{\begin{array}{cc} 2|a_1|^{-1}\zeta, & |a_1| > |a_0|+T^{-1},\\ 
O(|a_1|^{-1}\zeta), & |a_0|\le|a_1|\le|a_0|+T^{-1}.\end{array}\right. \]
\end{lemma}
\begin{proof}
Suppose firstly that $|a_1|>|a_0|+T^{-1}$. Then when $|y_0|\le|y_1|$ and 
$|y_1|=||\y||_\infty\ge 1$ we have
\[T|a_0y_0+a_1y_1|=T|y_1|.|a_0y_0/y_1+a_1|\ge T|a_0y_0/y_1+a_1|\]
\[\ge T(|a_1|-|a_0|.|y_0/y_1|)\ge T(|a_1|-|a_0|)> 1.\]
Thus
\begin{eqnarray*}
  \lefteqn{\{(y_0,y_1):1\le\max(|y_0|,|y_1|)\le 1+\zeta,\,
    T|a_0y_0+a_1y_1|\le 1\}}\\
  &=& \{(y_0,y_1):1\le\max(|y_0|,|y_1|)=|y_0|\le 1+\zeta,\,
  T|a_0y_0+a_1y_1|\le 1\}\\
  &=& \{(y_0,y_1):1\le|y_0|\le 1+\zeta,\,T|a_0y_0+a_1y_1|\le 1\},
\end{eqnarray*}
whence
\[J_T(a_0,a_1)=\int_{1\le|y_0|\le 1+\zeta}\int_{-\infty}^\infty
K_T(a_0y_0+a_1y_1) \d y_1\d y_0.\]
However
\[\int_{-\infty}^\infty K_T(a_0y_0+a_1y_1) \d y_1=|a_1|^{-1},\]
proving the lemma in the first case.

Suppose on the other hand, that $|a_0|\le|a_1|\le|a_0|+T^{-1}$. Then
\begin{eqnarray*}
  J_T(a_0,a_1)&\le&\int_{1\le|y_0|\le 1+\zeta}\int_{-\infty}^\infty
  K_T(a_0y_0+a_1y_1) \d y_1\d y_0\\
  &&\hspace{1cm}+ \int_{1\le|y_1|\le 1+\zeta}\int_{-\infty}^\infty
  K_T(a_0y_0+a_1y_1) \d y_0\d y_1\\
  &=& 2\zeta/|a_1|+2\zeta/|a_0|\\
  &\le& 4\zeta/|a_0|.
  \end{eqnarray*}
This is enough for the second part of the lemma unless
$|a_0|\le\tfrac12 |a_1|$, for example. In this latter case, the
condition $|a_0|\le|a_1|\le|a_0|+T^{-1}$ yields $|a_1|\le 2T^{-1}$,
whence the trivial bound $K_T(x)\le T$ produces an estimate
\[J_T(a_0,a_1)\le 4\zeta(1+\zeta)T\ll \zeta/|a_1|\]
as required. The second case of the lemma now
follows. 
\end{proof}

We can now give our alternative version of Lemma \ref{remW2}.
\begin{lemma}\label{remW2a}
Under the conditions of Lemma \ref{remW2} we have
\begin{eqnarray*}
S^{(0)}(X,Y)&=&2XY\zeta\mathfrak{S}_S
\int_{\R^3}\frac{W_3(\x)}{h_Q(\x)}\d x_0\d x_1\d x_2+O(\zeta^2 XY)\\
&&\hspace{1cm}\mbox{}+O_{\zeta,\eta}\left(XY(\log\log 3Y)^{-1/5}\right)
+O_\zeta(\eta^{2/3}XY).
\end{eqnarray*}
\end{lemma}
\begin{proof}
It follows from Lemmas \ref{nd2} and \ref{useful} that
\begin{eqnarray*}
\lefteqn{\int_{1\le||\y||_\infty\le 1+\zeta}\sigma_\infty(Q_\y;W_3)\d y_0\d y_1}\\
&=&\lim_{T\to\infty}\int_{\R^3}W_3(\x)J_T(Q_0(\x),Q_1(\x))\d x_0\d x_1\d x_2\\
&=& 2\zeta\int_{\R^3}\frac{W_3(\x)}{h_Q(\x)}\d x_0\d x_1\d x_2\\
&&+O\left(\lim_{T\to\infty}\zeta
\int_{\substack{\x\in\R^3\\ |Q_0(\x)\pm Q_1(\x)|\le T^{-1}}}
\frac{W_3(\x)}{h_Q(\x)}\d x_0\d x_1\d x_2\right).
\end{eqnarray*}
The integral in the main term is finite by Lemma \ref{QB}, and the
integral in the error term tends to zero as $T\to\infty$, by the
Dominated Convergence Theorem. The result then follows. 
\end{proof}

Our plan now is to sum $S^{(0)}(X,Y)$ for $Y=BX^{-1}(1+\zeta)^{-n}$,
obtaining the following result. We recall that we defined $\mathcal{E}$
in (\ref{cEd}) in terms of $\LL=3+\log XY$. As a substitute for this we 
now introduce
\[\mathcal{E}_0=\exp\left\{\frac{3+\log B}{\sqrt{\log(3+\log B)}}\right\},\]
and note that $\mathcal{E}\le \mathcal{E}_0$ when $XY\le B$.
\begin{lemma}\label{267}
Let $\eta,\zeta\in(0,\tfrac12)$ be given. Let 
\[S^{(4)}(B;X)=\sum_{\x\in\Z^3_{\mathrm{prim}}}
\;\; \sum_{\substack{\y\in\Z^2_{\mathrm{prim}}\\  y_0Q_0(\x)+y_1Q_1(\x)=0\\
||\y||_\infty\le B/X}}W_3(X^{-1}\x),\]
and suppose that  
$B^{1/2}\mathcal{E}_0^4\le X\le B\mathcal{E}_0^{-1}$. 
Assume that neither $\cl{M}$ nor $\cl{C}$ has rational points.  Then
\begin{eqnarray*}
  S^{(4)}(B;X)&=& 2\mathfrak{S}_S B
  \int_{\R^3}\frac{W_3(\x)}{h_Q(\x)}\d x_0\d x_1\d x_2+O(\zeta B)\\
&&\hspace{1cm}\mbox{}+
O_{\zeta,\eta}(B(\log\log B)^{-1/5})+O_{\zeta}(\eta^{2/3}B).
\end{eqnarray*}
\end{lemma}
\begin{proof}
We begin by observing that if $X\ge B^{1/2}\mathcal{E}_0^4$ and $Y\le
BX^{-1}$ then we automatically have $Y\le X\mathcal{E}_0^{-8}$. Thus
Lemmas \ref{remW2} and \ref{remW2a} apply. When we sum for  
$Y=BX^{-1}(1+\zeta)^{-n}$, terms with $||\y||_\infty\ll 1$ contribute $O(X)$,
and the sum of the error terms $O_{\zeta,\eta}(XY(\log\log 3Y)^{-1/5})$ will be
$O_{\zeta,\eta}(B(\log\log B)^{-1/5})$. We then find that
\begin{eqnarray*}
  S^{(4)}(B;X)&=&\sum_{1\le n\le (\log B/X)/(\log(1+\zeta))}
  S^{(0)}(X,BX^{-1}(1+\zeta)^{-n})+O(X)\\
&=&2\mathfrak{S}_S B\Sigma\int_{\R^3}\frac{W_3(\x)}{h_Q(\x)}\d x_0\d x_1\d x_2
+O(\zeta B)\\
&&\hspace{1cm}\mbox{}+
O_{\zeta,\eta}(B(\log\log B)^{-1/5})+O_{\zeta}(\eta^{2/3}B),
\end{eqnarray*}
where we have temporarily written
\[\Sigma=\sum_{1\le n\le(\log B/X)/(\log(1+\zeta))}\frac{\zeta}{(1+\zeta)^n}.\]
The lemma then follows, since $\Sigma=1+O(\zeta)+O(XB^{-1})$.
\end{proof}

We next prove the following companion to Lemma \ref{s31}.
\begin{lemma}\label{s32}
  Suppose that $B^{1/2}\mathcal{E}_0^4\le X\le B\mathcal{E}_0^{-1}$,
  and define
  \[S^{(5)}(B;X)=\sum_{\substack{\x\in\Z^3_{\mathrm{prim}}\\ X<||\x||_\infty\le 2X}}
\;\; \sum_{\substack{\y\in\Z^2_{\mathrm{prim}}\\  y_0Q_0(\x)+y_1Q_1(\x)=0\\
||\y||_\infty\le B/||\x||_\infty\\ ||\y||_\infty< ||\x||_\infty}}1.\]
Then  
 \begin{eqnarray}\label{S32*}
    S^{(5)}(B;X)&=&
(4\log 2)\tau_\infty\mathfrak{S}_S B+O(K^{-1}B)+O_K(\zeta B)\nonumber\\
&&\hspace{1cm}\mbox{}+
O_{\zeta,\eta,K}(B(\log\log B)^{-1/5})+O_{\zeta,K}(\eta^{2/3}B)
\end{eqnarray}
for any positive integer $K$ and any $\eta,\zeta>0$.
\end{lemma}
The reader should note that the condition $||\y||_\infty< ||\x||_\infty$ in the
definition of $S^{(5)}(B;X)$ is redundant for the above range of $X$, since
we will automatically have 
\[||\y||_\infty\le B/||\x||_\infty\le B/X< X\le ||\x||_\infty.\]
However we will later use $S^{(5)}(B;X)$ for other values of $X$.

Roughly speaking, for the situation in which $\cl{M}(\Q)$ and $\cl{C}(\Q)$ are 
both empty, Lemmas \ref{s31} and \ref{s32} tell us that
\[S^{(5)}(B;X)\sim(4\log 2)\tau_\infty\mathfrak{S}_S B\]
for the two ranges and $B^{1/3}\mathcal{E}_0\le X\le\mathcal{E}_0^{-1}$ and
$B^{1/2}\mathcal{E}_0^4\le X\le B\mathcal{E}_0^{-1}$. One may speculate that 
the same asymptotic behaviour holds when $X$ is in the intermediate range; 
but of course it breaks down near $X=B^{1/3}$, since
$S^{(5)}(B;X)=0$ for $X\ll B^{1/3}$.

\begin{proof}[Proof of Lemma \ref{s32}]
We begin by defining
 \[S^{(6)}(B;X)=\sum_{\substack{\x\in\Z^3_{\mathrm{prim}}\\ X<||\x||_\infty\le (1+\xi)X}}
\;\; \sum_{\substack{\y\in\Z^2_{\mathrm{prim}}\\  y_0Q_0(\x)+y_1Q_1(\x)=0\\
||\y||_\infty\le B/X}}1.\]
Then, following the argument we used for Lemma \ref{rw3},
we find from Lemma \ref{267} that 
\begin{eqnarray*}
S^{(6)}(B;X)&=& 4\xi\tau_\infty\mathfrak{S}_S B+O(\xi^2 B)+O_\xi(\zeta B)\\
&&\hspace{1cm}\mbox{}+
O_{\zeta,\eta,\xi}(B(\log\log B)^{-1/5})+O_{\zeta,\xi}(\eta^{2/3}B)
\end{eqnarray*}
for any $\xi\in(0,\tfrac12)$. The reasoning that leads to (\ref{S3})
then yields (\ref{S32*}).
\end{proof}

We are now ready to complete the proof of Theorem \ref{x+}..
\begin{proof}[Proof of Theorem \ref{x+}]
Our argument will use the universal bound (\ref{sxy}), and we note that when 
$\mathcal{M}$ and $\mathcal{C}$ have no rational points we can include cases
in which $X$ or $Y$ are $O(1)$. On summing $S^*(X,Y)$ for dyadic
values of $Y\ll \min(B/X,X)$ 
we deduce that
\beql{OW}
S^{(5)}(B;X)\ll\min(B,X^2).
\eeq
We now observe that
\[N_2(U,B)=\tfrac14\sum_{X=2^n\le B}S^{(5)}(B;X).\]
We estimate terms with $X< B^{1/2}\mathcal{E}_0^4$ or 
$X> B\mathcal{E}_0^{-1}$ via (\ref{OW}), giving a total contribution
\[\ll B+B\log\mathcal{E}_0\ll B(\log B)(\log\log B)^{-1/2}.\]
For the remaining terms we apply Lemma \ref{s32}.
The number of terms for which $B^{1/2}\mathcal{E}_0^4\le X\le
B\mathcal{E}_0^{-1}$ 
will be $(\log B)/(2\log 2)+O(\log\mathcal{E}_0)$, so that the leading
term in Lemma 
\ref{s32} contributes
\[\tfrac12\tau_\infty\mathfrak{S}_S B\log B +O(B(\log B)(\log\log B)^{-1/2})\]
to $N_2(U,B)$. Finally, the error terms produce a total
\[O(K^{-1} B\log B)+O_K(\zeta B\log B)\hspace{4cm}\]
\[\hspace{2cm}\mbox{}+O_{\zeta,\eta,K}(B(\log B)(\log\log B)^{-1/5})
+O_{\zeta,K}(\eta^{2/3}B\log B).\]
We can now complete the argument in much the same way as we did for
Theorem~\ref{extra} at the end of \S \ref{pt13}. We begin by choosing
$K$ so as to make the 
first error term sufficiently small, and then select $\zeta$ to ensure
that the second error term is acceptable. Next we pick $\eta$ so as to
make the fourth error term small enough, and finally we see that the
third error term is acceptable as soon as $B$ is sufficiently large. 
\end{proof}

\section{Equidistribution}
\label{sec:equi}
  
The aim of this section is to establish Theorem \ref{thm:equi}. 

\subsection{Set-up}
We begin by putting Peyre's
equidistribution framework from \cite[\S3.3]{Pey95} into the
setting for our Theorem \ref{thm:main}.
Let $H$ be the height \eqref{eqn:Height}, which we view as
induced by an adelic metric on the anticanonical bundle. Peyre's
Tamagawa measure is $\tau_H = \tau_{H,\infty} \prod_p (1 - 1/p)^2
\tau_{H,p}$ for local Tamagawa measures $\tau_{H,v}$ (Peyre uses
slightly different convergence factors which give an absolutely
convergent product). Then we say that \textit{equidistribution holds
on $U$} if for all continuity sets $W \subset S(\Adele_\Q)$ we have 
\begin{equation} \label{def:equi}
	\lim_{B \to \infty} \frac{\#\{(\x,\y)\in U(\Q) :H(\x,\y)\leq
          B, (\x,\y) \in W\} }{N(U,B)}  
	 = \frac{\tau_H(W)}{\tau_H(S(\Adele_\Q))}.
\end{equation}
Recall that a continuity set is a set whose boundary has
measure $0$. 
We similarly say that \textit{equidistribution holds over $\R$}
if \eqref{def:equi} holds for continuity sets of the form 
\begin{equation} \label{eqn:W_infty}
W = W_\infty \times  \prod_{v \neq \infty} S(\Q_v).
\end{equation}
The proof of Theorem \ref{thm:equi} uses two principles. The first is that equidistribution holds over $\R$; and the second is 
the fact that the
collection of surfaces in Theorem \ref{thm:main} is invariant under
linear changes of variables. Our approach is applicable in considerable
generality, in situations where these two principles apply. Nonetheless, given the following example, it is not clear what the correct generality is for such results; this
in part explains the delicate nature of our proof.

\begin{example}
	Let $X$ be a smooth cubic surface with $\Br X/ \Br\Q \cong \Z/2\Z$ and representative
	$\alpha \in \Br X$, such that there is a prime $p$ for which $\alpha$
	takes two values on $X(\Q_p)$ but evaluates trivially away
	from $p$. 
	Assume that equidistribution holds over $\R$ for $X$ and all other surfaces obtained after linear
	change of variables. Note that in this case our assumptions even imply that  
	$\frac{\tau_H(W)}{\tau_H(X(\Adele_\Q))} = \frac{\tau_H(W^{\Br})}{\tau_H(X(\Adele_\Q)^{\Br})}$
	for all $W$ of the form \eqref{eqn:W_infty}. 	Nonetheless \eqref{def:equi} does not hold with respect
	to all continuity sets due to a Brauer--Manin obstruction to weak approximation at $p$.
	When there is a Brauer--Manin obstruction then \eqref{def:equi} needs to be modified to take this
	into account.
\end{example}

\subsection{Proof of Theorem \ref{thm:equi}}

We will re-use various ideas from earlier in the
paper, and we shall therefore be brief.
Throughout this section we will think of $\x, \y$ and so on, as either
being projective rational points, or, by abuse of notation, as being the corresponding primitive
integer vectors. 
One may check that, whenever we impose a condition on such a vector 
$\x$, the condition holds equally for $-\x$.

\begin{lemma} \label{lem:equi_real}
Under the conditions of Theorem \ref{thm:main}, the rational points of $U$
are equidistributed over $\R$.
\end{lemma}
\begin{proof}
One sees from 	the analysis of Section \ref{trd} that the local 
Tamagawa measure for $\R$ is absolutely continuous with respect
to Lebesgue measure. This allows us to ignore sets of Lebesgue measure
zero. Consequently it suffices to consider sets $W_\infty$ whose closure 
avoids points where neither of $Q_1(\x)$, or $x_0$ vanish. For such sets
the equation for $S$ determines $\y$, given $\x$, so that
it is enough to
examine sets where there is a constraint on $\x$ but not on $\y$. Thus we may
take the set $W_\infty$ to have the shape
\beql{WG}
W_\infty=\{(\x,\y)\in U(\R): x_0^{-1}(x_1,x_2)\in \Gamma\}
\eeq
for some set $\Gamma\subseteq \R^2$. Our goal is then to show
that
\begin{equation} \label{equi2}
\lim_{B \to \infty} 
\frac{\#\{(\x,\y)\in U(\Q) :H(\x,\y)\leq B, \, x_0^{-1}(x_1,x_2)\in \Gamma\} }{N(U,B)}  
	 = \frac{\tau_\infty(\Gamma)}{\tau_\infty}.
\end{equation}	
with
\beql{tif}
\tau_\infty(\Gamma)=
\int_{(u_1,u_2)\in \Gamma}\frac{\d u_1\d u_2}{h_Q(\widehat{\u})||\widehat{\u}||_\infty}
\eeq
where we introduce the notation $\widehat{\x}=(1,x_1,x_2)$ for a vector
$\x=(x_1,x_2)\in\R^2$. Similarly we will later use the notation
$\widehat{\y}=(1,y)$ when $y\in\R$.

To prove (\ref{equi2}) we adapt the arguments used in Sections \ref{pt13} and \ref{Assem}.
In passing from Lemma \ref{8.2} to Lemma \ref{rw3} we use weights 
$W_3(\mathbf{z})$
which approximate the characteristic function of the set
\[\{(z_0,z_1,z_2): 1<|z_0|<1+\xi, z_0^{-1}(z_1,z_2)\in \Gamma\}.\]
We can thus establish an analogue of Theorem \ref{extra} for
points in $W_\infty$ given by (\ref{WG}), in which
$\tau_\infty$ is replaced by $\tau_\infty(\Gamma)$.  If we use the same choice for 
$W_3(\mathbf{z})$ in passing from Lemma \ref{remW2a} to a version of
Lemma \ref{267} we end up with an analogue of Theorem \ref{x+} in which
again $\tau_\infty$ is replaced by $\tau_\infty(\Gamma)$.
This suffices for the lemma.
\end{proof}

To handle equidistribution for more general continuity sets
$W\subseteq\prod_v S(\Q_v)$ we begin by giving a more down-to-earth formulation.
Suppose we are given a set $\Gamma\subseteq \R^2$, an integer $q\ge 1$ and 
vectors $\mathbf{a}\in\Z_{\mathrm{prim}}^3$ and
$\mathbf{b}\in\Z_{\mathrm{prim}}^2$. We then define
\begin{eqnarray}\label{NBD}
N(B)&=&N(B;\Gamma,q,\mathbf{a},\mathbf{b})\nonumber\\
&\hspace{-4cm}=& \hspace{-2cm}\card\left\{(\x,\y)\in\Z^3_{\mathrm{prim}}\times\Z^2_{\mathrm{prim}}:
\begin{array}{c} (\x,\y)\in W_\infty,\,  ||\x||_\infty ||\y||_\infty\le B,\\
\exists \lambda,\mu,\;\; \x\equiv\lambda\mathbf{a},\, \y\equiv\mu\mathbf{b}\Mod{q}
\end{array}\right\},
\end{eqnarray}
where $W_\infty$ is defined by the condition (\ref{WG}).  We also introduce 
$p$-adic densities
\begin{eqnarray}\label{vpd1}
\lefteqn{\varpi_p(q,\mathbf{a},\mathbf{b})}\nonumber\\
&\hspace{-2cm}=&\hspace{-1cm}\lim_{n\to\infty}p^{-2n}\phi(p^n)^{-2}
\card\left\{ \x,\y\Mod{p^n }: 
\begin{array}{c} \gcd(\x,p)=\gcd(\y,p)=1\\
\exists\lambda,\,\x\equiv\lambda\mathbf{a}\Mod{q_p}, \\
\exists\mu,\, \y\equiv\mu\mathbf{b}\Mod{q_p}, \\
Q_\y(\x)\equiv 0\Mod{p^n}\end{array}\right\},
\end{eqnarray}
where the notation $m_p$ denotes the $p$-part of a positive integer $m$.

We then have the following result.
\begin{lemma}\label{equiSC}
To prove equidistribution for $S$, subject to the assumptions of Theorem \ref{thm:main}, it is
enough to show that
\[N(B)\sim 4\times\left\{\frac12\tau_\infty(\Gamma)
\prod_p(1-p^{-1})^2\varpi_p(q,\mathbf{a},\mathbf{b})\right\} B\log B\]
as $B\to\infty$, for every fixed closed square $\Gamma\subset\R^2$
such that $W_\infty$ avoids points where $Q_0$, $Q_1$, or $C$ vanish, and for 
every fixed admissible choice of $q$, $\mathbf{a}$, and $\mathbf{b}$.
\end{lemma}
The case $q=1$ reduces to Lemma \ref{lem:equi_real}, the factor 4 accounting for the fact 
that each projective point $([\x],[\y])$ corresponds to four pairs $(\x,\y)$.

\begin{proof}
We begin by observing that
it is enough to handle sets $W$ which are products
$\prod_v W_v$ of open sets $W_v$, with $W_v=S(\Q_v)$ for all but
finitely many places $v$, because continuity sets
can be approximated both from below and above by suitable
finite unions of such product sets. As before it suffices to consider 
sets $W_\infty$ of the shape \eqref{WG}, avoiding points where $Q_1$, or $x_0$ vanishes,
and indeed we can take $\Gamma$ to be a closed square, since more general $\Gamma$
can be approximated from above and below by suitable finite unions of such sets.

By the same reasoning it
suffices to take $\prod_{v \neq\infty} W_v$ to be a set
$W(\mathbf{a},\mathbf{b};q)$ for some
$(\mathbf{a},\mathbf{b})\in\Z^3_{\mathrm{prim}}\times\Z^2_{\mathrm{prim}}$
and $q\in\N$, where $W(\mathbf{a},\mathbf{b};q)$ is defined as the set of
$(\x,\y)\in\Z^3_{\mathrm{prim}}\times\Z^2_{\mathrm{prim}}$ such that
$\x\equiv\lambda\mathbf{a}\Mod{q}$ and $\y\equiv\mu\mathbf{b}\Mod{q}$, 
for some choices of  $\lambda,\mu\in\Z$. 
\end{proof}

We proceed to show how the congruence conditions modulo $q$ may be removed.
We emphasize that $\Gamma$ and $q$ will be fixed throughout the argument,
as will $\mathbf{a}$ and $\mathbf{b}$.  It will be convenient to decompose $\Gamma$ 
into $K^2$ smaller squares $\Gamma^{(k)}$, each of side $O(K^{-1})$, and to write 
$W_\infty^{(k)}$ for the corresponding set of pairs $(\x,\y)\in W_\infty$. In due 
course we will let $K$ go to infinity.  
\begin{lemma}\label{remove}
For divisors $r$ and $s$ of $q$ there are matrices $C=C(r;q,\mathbf{a})\in\mathrm{M}_3(\Z)$
and $D=D(s;q,\mathbf{b})\in\mathrm{M}_2(\Z)$ of determinant $(q/r)^2$ and $q/s$ respectively,
such that
\[N(B)=\sum_{k\le K^2}\sum_{r,s\mid q}\mu(r)\mu(s)\sum_{e,f\mid q^\infty}N_k(B;r,s,e,f),\]
where
\[N_k(B;r,s,e,f)=\card\left\{(\bg,\bd)\in \Z^3_{\mathrm{prim}}\times\Z^2_{\mathrm{prim}}:
\begin{array}{c}  (C\bg,D\bd)\in W_\infty^{(k)},\\  
||C\bg||_\infty ||D\bd||_\infty\le B/rsef\end{array}\right\}.\]
\end{lemma}

\begin{proof}
Our first step is to replace the condition that $\x$ should be primitive by two other
constraints. Firstly we will require that $\gcd(x_0,x_1,x_2,q)=1$, and secondly we will suppose 
that every prime factor of $\gcd(x_0,x_1,x_2)$ also divides $q$. It will be convenient to write
\[S_3(q)=\{\x\in\Z^3-\{\mathbf{0}\}: p \mid \gcd(x_0,x_1,x_2)\Rightarrow p|q\}\]
for the set of $\x$ satisfying this second condition. We treat the primitivity condition for $\y$ in a
similar fashion, using a set $S_2(q)\subseteq\Z^2$. We then deduce that
\[N(B)=\sum_{r,s\mid q}\mu(r)\mu(s) N(B;r,s),\]
where
\[N(B;r,s)=\card\left\{(\x,\y)\in S_3(q)\times S_2(q):
\begin{array}{c} r\mid\x,\, s\mid\y, \\ (\x,\y)\in W_\infty,\,  ||\x||_\infty ||\y||_\infty\le B,\\
\exists \lambda,\mu, \; \x\equiv\lambda\mathbf{a},\, \y\equiv\mu\mathbf{b}\Mod{q}
\end{array}\right\}.\]
If we now write $\x=r\u$ and $\y=s\v$ the conditions then become $\u\in S_3(q)$, $\v\in S_2(q)$,
$(\u,\v)\in W_\infty$, $||\u||_\infty ||\v||_\infty\le B/rs$, and $u_0^{-1}(u_1,u_2)\in \Gamma$,
along with the congruences $\u\equiv\lambda\mathbf{a}\Mod{q/r}$ and 
$\v\equiv\mu\mathbf{a}\Mod{q/s}$, for suitable $\lambda,\mu$.

The constraint that $\u\equiv\lambda\mathbf{a}\Mod{q/r}$ for some $\lambda$
is equivalent to $\u$ lying on a certain lattice of determinant
$(q/r)^2$. We take $\mathbf{g}_0,\mathbf{g}_1,\mathbf{g}_2$ to be a basis
for the lattice, and write $\u=C\mathbf{c}$ where 
$C=(\mathbf{g}_0|\mathbf{g}_1|\mathbf{g}_2)$, so that $\det(C)=(q/r)^2$. Similarly
the condition $\v\equiv\mu\mathbf{b}\Mod{q/s}$ becomes
$\v=D\mathbf{d}$ for an appropriate $2\times 2$ integer matrix $D$
of determinant $q/s$. Since the determinants of $C$ and $D$ only contain primes which divide
$q$ the conditions $\u\in S_3(q)$ and $\v\in S_2(q)$ are equivalent to $\mathbf{c}\in S_3(q)$
and $\mathbf{d}\in S_2(q)$ respectively. It then follows that
\[N(B;r,s)=\card\left\{(\mathbf{c},\mathbf{d})\in S_3(q)\times S_2(q):
\begin{array}{c}  (C\mathbf{c},D\mathbf{d})\in W_\infty, \\  ||C\mathbf{c}||_\infty ||D\mathbf{d}||_\infty\le B/rs
\end{array}\right\}.\]
A vector $\mathbf{c}\in S_3(q)$ may be written as $e\bg$ with $e|q^\infty$ and 
$\bg\in\Z^3_{\mathrm{prim}}$, and similarly for $\mathbf{d}$. We therefore have
\[N(B;r,s)=\sum_{e,f\mid q^\infty}N(B;r,s,e,f)\]
with
\[N(B;r,s,e,f)=\card\left\{(\bg,\bd)\in \Z^3_{\mathrm{prim}}\times\Z^2_{\mathrm{prim}}:
\begin{array}{c}  (C\bg,D\bd)\in W_\infty,\\  
||C\bg||_\infty ||D\bd||_\infty\le B/rsef\end{array}\right\}.\]
The lemma then follows on splitting $W_\infty$ into the sets $W_{\infty}^{(k)}$.
\end{proof}

The condition $||C\bg||_\infty ||D\bd||_\infty\le B/rsef$ is not of the form $||\bg||_\infty||\bd||\infty\le B'$,
but when the squares $\Gamma^{(k)}$ are sufficiently small we can produce a condition
which is approximately of the right shape.
\begin{lemma}\label{ARS}
Let $\ba^{(k)}=(\alpha^{(k)}_1,\alpha^{(k)}_2)$ be the centre of $\Gamma^{(k)}$.
Then
\beql{TH}
\frac{||\bg||_\infty ||\bd||_\infty}{||C\bg||_\infty ||D\bd||_\infty}=\Omega_k+O_{\Gamma,q}(1/K),
\eeq
 for $(C\bg,D\bd)\in W_\infty^{(k)}$, where
\beql{THF}
\Omega_k=\frac{||C^{-1}\widehat{\ba}^{(k)}||_\infty\;\left|\left|D^{T}\left(Q_0(\widehat{\ba}^{(k)}),Q_1(\widehat{\ba}^{(k)})\right)\right|\right|_\infty}
{\det(D)\;||\widehat{\ba}^{(k)}||_\infty\;h_Q(\widehat{\ba}^{(k)})}.
\eeq
\end{lemma}
We remind the reader that $\widehat{\u}=(1,u_1,u_2)$ for $\u=(u_1,u_2)\in\R^2$.
\begin{proof}
When $(\x,\y)\in W_\infty^{(k)}$ we will have
\[x_0^{-1}(x_1,x_2)=\ba^{(k)}+O_{\Gamma}(1/K).\]
We recall that $|Q_1(\widehat{\u})|$ must be  bounded away from zero as $\u$ runs over $\Gamma$, 
and we therefore deduce that $y_0^{-1}y_1=\beta^{(k)}+O_{\Gamma}(1/K)$ for 
points $(\x,\y)\in W^{(k)}_\infty$, with 
\beql{btd}
\beta^{(k)}=-\frac{Q_0(\widehat{\ba}^{(k)})}{Q_1(\widehat{\ba}^{(k)})}.
\eeq
We now see that if $(C\bg,D\bd)\in W_\infty^{(k)}$ then
\[C\bg=(C\bg)_0\{\widehat{\ba}^{(k)}+O_{\Gamma,q}(1/K)\}\]
and
\[ D\bd=(D\bd)_0\{\widehat{\bb}^{(k)}+O_{\Gamma,q}(1/K)\},\]
whence
\[\bg=(C\bg)_0\{C^{-1}\widehat{\ba}^{(k)}+O_{\Gamma,q}(1/K)\}\]
and
\[\bd=(D\bd)_0\{D^{-1}\widehat{\bb}^{(k)}+O_{\Gamma,q}(1/K)\}.\]
The implied constants above will depend on $C$ and $D$, but these run over a fixed set of 
possibilities once $q$ is given.
The estimate (\ref{TH}) then follows, but with
\[\Omega_k=\frac{||C^{-1}\widehat{\ba}^{(k)}||_\infty\;||D^{-1}\widehat{\bb}^{(k)}||_\infty}
{||\widehat{\ba}^{(k)}||_\infty\;||\widehat{\bb}^{(k)}||_\infty}.\]
However we see from (\ref{btd}) that 
\[||D^{-1}\widehat{\bb}^{(k)}||_\infty=
\frac{\left|\left|D^T\left(Q_0(\widehat{\ba}^{(k)}),Q_1(\widehat{\ba}^{(k)})\right)\right|\right|_\infty}{\det(D)\,|Q_1(\widehat{\ba}^{(k)})|}\]
and that 
\[||\widehat{\bb}^{(k)}||_\infty=
\frac{\left|\left|\left(Q_0(\widehat{\ba}^{(k)}),Q_1(\widehat{\ba}^{(k)})\right)\right|\right|_\infty}{|Q_1(\widehat{\ba}^{(k)})|}
=\frac{h_Q(\widehat{\ba}^{(k)})}{|Q_1(\widehat{\ba}^{(k)})|}. \]
This produces the expression (\ref{THF}), and the lemma follows.
\end{proof}

Our plan next is to combine Lemmas \ref{lem:equi_real}, and  \ref{ARS} to produce asymptotic formulae for
$N_k(B;r,s,e,f)$ and for 
\beql{NBKD}
N_k(B;r,s)=\sum_{e,f\mid q^\infty}N_k(B;r,s,e,f).
\eeq
Let $f=f_C:\Gamma^{(k)}\to\R^2$ be the map given by
\beql{mapf}
f(\u)=\left(\frac{(C^{-1}\widehat{\u})_1}{(C^{-1}\widehat{\u})_0}\,,\,\frac{(C^{-1}\widehat{\u})_2}{(C^{-1}\widehat{\u})_0}\right),
\eeq
and let $\widetilde{\Gamma}^{(k)}$ be the image of $f$. Let $\widetilde{W}^{(k)}$ be the set of pairs $(\bg,\bd)$ for which 
$(C\bg,D\bd)\in W^{(k)}$. Thus $(\bg,\bd)\in\widetilde{W}^{(k)}$ if and only if 
$\gamma_0^{-1}(\gamma_1,\gamma_2)\in\widetilde{\Gamma}^{(k)}$ and $Q_{D\bd}(C\bg)=0$. This last condition defines
a surface $\widetilde{S}$, depending on $C$ and $D$, taking the form $\delta_0\widetilde{Q}_0(\bg)+\delta_1 \widetilde{Q}_1(\bg)=0$, 
say, where
\beql{QTD}
\left(\begin{array}{c} \widetilde{Q}_0(\x) \\ \widetilde{Q}_1(\x)\end{array}\right)=D^T\left(\begin{array}{c}
Q_0(C\x) \\ Q_1(C\x)\end{array}\right).
\eeq
We let $\tau(\widetilde{\Gamma}^{(k)};C,D)=\tau(\widetilde{\Gamma}^{(k)})$ and $\widetilde{\varpi}_p(C,D)=\widetilde{\varpi}_p$ be the 
local densities for the surface $\widetilde{S}$. We will often suppress the dependency on $C$ and $D$. 
The real density is given by (\ref{tif}), and the $p$-adic densities will be
\begin{eqnarray}\label{vpd2}
\lefteqn{\widetilde{\varpi}_p(C,D)}\nonumber\\
&\hspace{-2cm}=&\hspace{-1cm}\lim_{n\to\infty}p^{-2n}\phi(p^n)^{-2}
\card\left\{ \x,\y\Mod{p^n }: \begin{array}{c} \gcd(\x,p)=\gcd(\y,p)=1\\
Q_{D\y}(C\x)\equiv 0\Mod{p^n}\end{array}\right\}.
\end{eqnarray}
We can now give the following asymptotic formulae.
\begin{lemma}\label{twoNs}
We have
\[N_k(B;r,s,e,f)=2\left(\tau_\infty(\widetilde{\Gamma}^{(k)})\prod_p(1-p^{-1})^2\widetilde{\varpi}_p\right)
\{\Omega_k+O_{\Gamma,q}(1/K)\} \frac{B}{rsef}\log \frac{B}{rsef}\]
as soon as $B/rsef \gg_{K,\Gamma,q} 1$.  Moreover
\[N_k(B;r,s)=2\left(\tau_\infty(\widetilde{\Gamma}^{(k)})\frac{q^2}{\phi(q)^2}\prod_p (1-p^{-1})^2\widetilde{\varpi}_p\right)
\{\Omega_k+O_{\Gamma,q}(1/K)\} \frac{B}{rs}\log B\]
as soon as $B\gg_{K,\Gamma,q} 1$.
\end{lemma}
\begin{proof}
We use the temporary notation
\[N_k(B)=\card\{(\bg,\bd)\in \Z^3_{\mathrm{prim}}\times\Z^2_{\mathrm{prim}}:
 ||\bg||_\infty ||\bd||_\infty\leq B, \, (\bg,\bd)\in \widetilde{W}^{(k)}_\infty\}.\]
 Since $\widetilde{S}$ satisfies the conditions 
for Theorem \ref{thm:main} we may deduce from Lemma \ref{lem:equi_real} that $N_k(B)\sim \tilde{c}B\log B$ with
\[\tilde{c}=4\times\frac{1}{2}\left(\tau_\infty(\widetilde{\Gamma}^{(k)})\prod_p(1-p^{-1})^2\widetilde{\varpi}_p\right),\]
where the factor 4 allows for the fact that each projective point $([\bg],[\bd])$ corresponds to 4 pairs $(\bg,\bd)$. 
Thus
\[|N_k(B)-\tilde{c}B\log B|\le K^{-1}\tilde{c}B\log B\]
if $B\gg_{K,\Gamma,q}1$.

According to Lemma \ref{ARS} we have
\[N_k\left((\Omega_k-c_{\Gamma,q}/K)B/rsef\right)\le N_k(B;r,s,e,f)\le 
N_k\left((\Omega_k+c_{\Gamma,q}/K)B/rsef\right),\]
for a suitable constant $c_{\Gamma,q}$,
and the first assertion of Lemma \ref{twoNs} then follows. Note that the
implied constant in the $O_{\Gamma,q}(1/K)$ notation is independent of $e$ and $f$. There is
a potential dependence on $r$ and $s$, but these are divisors of $q$, which is fixed.

When $B/rsef\ll_{K,\Gamma,q} 1$ we have 
$N_k(B;r,s,e,f)\ll_{K,\Gamma,q} 1$, and indeed the function $N_k(B;r,s,e,f)$ will vanish if $B/rsef<1$. 
It follows that
\begin{align*}
N_k(B;r,s) &=\sum_{e,f\mid q^\infty}N_k(B;r,s,e,f)\hspace{6cm} \\
&=\tilde{c}\{\Omega_k+O_{\Gamma,q}(1/K)\}\sum_{\substack{e,f\mid q^\infty \\ ef\ll_{K,\Gamma,q} B/rs}}
\frac{B}{rsef}\log \frac{B}{rsef}+O_{K,\Gamma,q}\left(\sum_{\substack{e,f\mid q^\infty \\ ef\le B/rs}}1\right).
\end{align*}
Standard estimations using ``Rankin's Trick''  show that
\[\sum_{e,f\mid q^{\infty}}\frac{\log ef}{ef}\ll_q 1,\;\;\;
\sum_{\substack{e,f\mid q^\infty \\ ef\ge R}}\frac{1}{ef}\ll_{q} R^{-1/2},\;\;\;
\mbox{and}\;\;\;\sum_{\substack{e,f\mid q^\infty \\ ef\le R}}1\ll_{q} R^{1/2},\]
say. Since
\[\sum_{e,f\mid q^\infty}e^{-1}f^{-1}=q^2\phi(q)^{-2}\]
we conclude that
\[N_k(B;r,s)=\tilde{c}\{\Omega_k+O_{\Gamma,q}(1/K)\}\frac{B}{rs}+O_{K,\Gamma,q}(B),\]
and the lemma follows.
\end{proof}

We next relate $\tau_\infty(\widetilde{\Gamma}^{(k)})\Omega_k$ to $\tau(\Gamma^{(k)})$.
\begin{lemma}\label{tom}
We have
\[\tau_\infty(\widetilde{\Gamma}^{(k)})\Omega_k=\{1+O(1/K)\}\frac{1}{\det(CD)}\tau_\infty(\Gamma^{(k)}).\]
\end{lemma}
\begin{proof}
According to (\ref{tif}) and (\ref{THF}) we have
\[\tau_\infty(\widetilde{\Gamma}^{(k)})\Omega_k=
\frac{||C^{-1}\widehat{\ba}^{(k)}||_\infty\;\left|\left|D^{T}\left(Q_0(\widehat{\ba}^{(k)}),Q_1(\widehat{\ba}^{(k)})\right)\right|\right|_\infty}
{\det(D)\;||\widehat{\ba}^{(k)}||_\infty\;h_Q(\widehat{\ba}^{(k)})}
\int_{\w\in \widetilde{\Gamma}^{(k)}}
\frac{\d w_1\d w_2}{h_{\widetilde{Q}}(\widehat{\w})||\widehat{\w}||_\infty}.\]
We will use the map $f$ given by (\ref{mapf}), for which one may compute that
\[\mathrm{Jac}(f)(\u)=\frac{\det(C^{-1})}{{(C^{-1}\widehat{\u})_0}^3}.\]
Substituting $\w=f(\u)$ we then find that
\[||\widehat{\w}||_\infty=||C^{-1}\widehat{\u}||_\infty/|(C^{-1}\widehat{\u})_0|,\]
whence
\[\tau_\infty(\widetilde{\Gamma}^{(k)})\Omega_k=\frac{1}{\det(CD)}\int_{\u\in\Gamma^{(k)}}
\frac{||C^{-1}\widehat{\ba}^{(k)}||_\infty}{||C^{-1}\widehat{\u}||_\infty}
\frac{\Psi(\u)\d u_1\d u_2}{h_Q(\widehat{\ba}^{(k)})||\widehat{\ba}^{(k)}||_\infty},\]
with
\[\Psi(\u)=\frac{\left|\left|D^{T}\left(Q_0(\widehat{\ba}^{(k)}),Q_1(\widehat{\ba}^{(k)})\right)\right|\right|_\infty}
{h_{\widetilde{Q}}(\widehat{f(\u)})|(C^{-1}\widehat{\u})_0|^2}.\]

At this point we recall that $\Gamma^{(k)}$ is a square of side $O(K^{-1})$, centred at $(\alpha_1,\alpha_2)$, so that
\[\frac{||C^{-1}\widehat{\ba}^{(k)}||_\infty}{||C^{-1}\widehat{\u}||_\infty}=1+O(1/K)\]
and
\[\frac{1}{h_Q(\widehat{\ba}^{(k)})||\widehat{\ba}^{(k)}||_\infty}=\{1+O(1/K)\}\frac{1}{h_Q(\widehat{\u})||\widehat{\u}||_\infty}.\]
The forms $\widetilde{Q}_0,\widetilde{Q}_1$ defining $\widetilde{S}$ are given by (\ref{QTD}), and
\[C\widehat{f(\u)}=\frac{1}{(C^{-1}\widehat{\u})_0}\widehat{\u},\]
whence
\[\Psi(\u)=\frac{\left|\left|D^{T}\left(Q_0(\widehat{\ba}^{(k)}),Q_1(\widehat{\ba}^{(k)})\right)\right|\right|_\infty}
{\left|\left|D^{T}\left(Q_0(\widehat{\u}),Q_1(\widehat{\u})\right)\right|\right|_\infty}.\]
Thus $\Psi(\u)=1+O(1/K)$ on $\Gamma^{(k)}$ and we conclude that
\[\tau_\infty(\widetilde{\Gamma}^{(k)})\Omega_k=\frac{1+O(1/K)}{\det(CD)}\int_{\u\in\Gamma^{(k)}}
\frac{\d u_1\d u_2}{h_Q(\widehat{\u})||\widehat{\u}||_\infty}=\frac{1+O(1/K)}{\det(CD)}\tau_\infty(\Gamma^{(k)})\]
as required.
\end{proof}

We can now sum over the small squares $\Gamma^{(k)}$ for $k\le K^2$ to obtain the following asymptotic formula.
We recall that $N(B)$ is given by (\ref{NBD}).
\begin{lemma}\label{Nass}
We have
\[N(B)\sim 2\left(\tau_\infty(\Gamma)\frac{q^2}{\phi(q)^2}\sum_{r,s\mid q}\frac{\mu(r)\mu(s)}{rs\det(CD)}
\prod_p (1-p^{-1})^2\widetilde{\varpi}_p(C,D)\right)B\log B.\]
\end{lemma}
We remind the reader that in fact $\det(C)=(q/r)^2$ and $\det(D)=q/s$.

\begin{proof}
In view of Lemma \ref{remove} and (\ref{NBKD}), it will be enough to show that
\[\sum_{k\le K^2}N_k(B;r,s)\sim \kappa(r,s)B\log B\]
for each fixed pair of divisors $r,s$ of $q$, with
\[\kappa(r,s)=\frac{2}{rs}\left(\tau_\infty(\Gamma)\frac{q^2}{\phi(q)^2}\frac{1}{\det(CD)}
\prod_p (1-p^{-1})^2\widetilde{\varpi}_p(C,D)\right).\]
We will use Lemma \ref{twoNs} for this.  It is clear that the $p$-adic densities
$\widetilde{\varpi}_p(C,D)$ are independent of $k$, and Lemma \ref{tom} shows that
\[\sum_{k\le K^2}\tau_\infty(\widetilde{\Gamma}^{(k)})\Omega_k=\{1+O(1/K)\}\frac{1}{\det(CD)}\tau_\infty(\Gamma).\]
We then find that
\[\sum_{k\le K^2}N_k(B;r,s)=\{1+O(1/K)\}\kappa(r,s)B\log B\]
as soon as $B\gg_{K,\Gamma,q}1$. It follows that
\[\limsup_{B\to\infty}\frac{\sum_{k\le K^2}N_k(B;r,s)}{\kappa(r,s)B\log B}\le 1+O(1/K)\]
for any fixed $K$, whence the limsup must be at most 1. Similarly the liminf is at least 1,
and the lemma follows.
\end{proof}

We now examine the $p$-adic densities $\widetilde{\varpi}_p(C,D)$, with a view to determining how
they are related to the corresponding densities $\varpi_p(q,\mathbf{a},\mathbf{b})$ in Lemma \ref{equiSC}.
We remind the reader that $m_p$ denotes the $p$-part of a positive integer $m$. 
\begin{lemma}\label{varpis}
We have 
\[\widetilde{\varpi}_p(C,D)=\varpi_p(q,\mathbf{a},\mathbf{b})\]
when $p\nmid q$, and 
\[\widetilde{\varpi}_p(C,D)=\frac{q_p^3}{r_p}\lim_{n\to\infty}p^{-4n}\card\left\{ \x,\y\Mod{p^n }: 
\begin{array}{c} \exists\lambda,\,\x\equiv\lambda\mathbf{a}\Mod{q_p},\, r_p|\x \\
\exists\mu,\, \y\equiv\mu\mathbf{b}\Mod{q_p},\, s_p|\y \\
Q_\y(\x)\equiv 0\Mod{p^n}\end{array}\right\}\]
when $p\mid q$.
\end{lemma}
The reader should recall that $\varpi_p(q,\mathbf{a},\mathbf{b})$ and $\widetilde{\varpi}_p(C,D)$ are given by
(\ref{vpd1}) and (\ref{vpd2}) respectively.

\begin{proof}
The result is easy for $p\nmid q$, since the matrices $C$ and $D$ are invertible modulo $p$ in this case.
For primes $p \mid q$ we see, according to Lemma  \ref{varpialt}, that
\[\widetilde{\varpi}_p(C,D)=\lim_{n\to\infty}p^{-4n}\card\left\{(\bg,\bd)\Mod{p^n}:  Q_{D\bd}(C\bg)\equiv 0\Mod{p^n}\right\}.\]
To examine the above limit we consider the set occurring 
on the right, and assume henceforth that the exponent $n$ is large enough that $q_p^2 \mid p^n$.
The map $C:(\Z/p^n\Z)^3\to(\Z/p^n\Z)^3$ has kernel of size $\gcd(\det(C),p^n)=\gcd((q/r)^2,p^n)=(q/r)_p^2$. Moreover the image of the above map consists of the cosets $\mathbf{u}+(\Z/p^n\Z)^3$ for which there is a 
value of $\lambda$ such that $\u\equiv\lambda\mathbf{a}\Mod{(q/r)_p}$.
We may argue similarly with $D:(\Z/p^n\Z)^2\to(\Z/p^n\Z)^2$, and deduce that
\begin{eqnarray*}
\lefteqn{\card\left\{(\bg,\bd)\Mod{p^n}:  Q_{D\bd}(C\bg)\equiv 0\Mod{p^n}\right\}}\hspace{1cm}\\
&=&
(q/r)_p^2(q/s)_p\card\left\{(\mathbf{u},\mathbf{v})\Mod{p^n}: 
\begin{array}{c} \exists\lambda,\,\u\equiv\lambda\mathbf{a}\Mod{(q/r)_p},\\
\exists\mu,\, \v\equiv\mu\mathbf{b}\Mod{(q/s)_p},\\
Q_\v(\u)\equiv 0\Mod{p^n}\end{array}\right\}.
\end{eqnarray*}
We now let $\x\equiv r_p\u\Mod{p^n r_p}$ and $\y\equiv s_p\v\Mod{p^n s_p}$, whence
\begin{eqnarray*}
\lefteqn{\card\left\{(\bg,\bd)\Mod{p^n}:  Q_{D\bd}(C\bg)\equiv 0\Mod{p^n}\right\}}\hspace{1cm}\\
&=&
(q/r)_p^2(q/s)_p\card\left\{\begin{array}{c} \x\Mod{p^n r_p}\\ \y\Mod{p^n s_p} \end{array}: 
\begin{array}{c} \exists\lambda,\,\x\equiv\lambda\mathbf{a}\Mod{q_p},\, r_p \mid \x \\
\exists\mu,\, \y\equiv\mu\mathbf{b}\Mod{q_p},\, s_p \mid \y \\
Q_\y(\x)\equiv 0\Mod{p^n r_p^2 s_p}\end{array}\right\}.
\end{eqnarray*}
If $\x$ and $\y$ satisfy the conditions above, and 
\[\x'\equiv\x\Mod{p^n r_p},\;\;\;\mbox{and}\;\;\; \y'\equiv\y\Mod{p^n s_p},\]
then $\x'$ and $\y'$ satisfy the same conditions, including the requirement that 
$Q_{\y'}(\x')$ is divisible by $p^n r_p^2 s_p$. It follows that
\begin{eqnarray*}
\lefteqn{\card\left\{\begin{array}{c} \x\Mod{p^n r_p}\\ \y\Mod{p^n s_p} \end{array}: 
\begin{array}{c} \exists\lambda,\,\x\equiv\lambda\mathbf{a}\Mod{q_p},\, r_p \mid \x \\
\exists\mu,\, \y\equiv\mu\mathbf{b}\Mod{q_p},\, s_p \mid \y \\
Q_\y(\x)\equiv 0\Mod{p^n r_p^2 s_p}\end{array}\right\}}\\
&=&(r_p s_p)^{-3}(r_p^2)^{-2}
\card\left\{\begin{array}{c} \x\Mod{p^n r_p^2s_p}\\ \y\Mod{p^n r_p^2 s_p} \end{array}: 
\begin{array}{c} \exists\lambda,\,\x\equiv\lambda\mathbf{a}\Mod{q_p},\, r_p \mid \x \\
\exists\mu,\, \y\equiv\mu\mathbf{b}\Mod{q_p},\, s_p \mid \y \\
Q_\y(\x)\equiv 0\Mod{p^n r_p^2 s_p}\end{array}\right\},
\end{eqnarray*}
and hence that
\begin{eqnarray*}
\lefteqn{p^{-4n}\card\left\{(\bg,\bd)\Mod{p^n}:  Q_{D\bd}(C\bg)\equiv 0\Mod{p^n}\right\}}\\
&=&\frac{q_p^3}{r_p}(p^nr_p^2s_p)^{-4}
\card\left\{ \x,\y\Mod{p^n r_p^2s_p}: 
\begin{array}{c} \exists\lambda,\,\x\equiv\lambda\mathbf{a}\Mod{q_p},\, r_p \mid \x \\
\exists\mu,\, \y\equiv\mu\mathbf{b}\Mod{q_p},\, s_p \mid \y \\
Q_\y(\x)\equiv 0\Mod{p^n r_p^2 s_p}\end{array}\right\}.
\end{eqnarray*}
We therefore deduce that
\[\widetilde{\varpi}_p(C,D)=\frac{q_p^3}{r_p}\lim_{n\to\infty}p^{-4n}\card\left\{ \x,\y\Mod{p^n }: 
\begin{array}{c} \exists\lambda,\,\x\equiv\lambda\mathbf{a}\Mod{q_p},\, r_p \mid \x \\
\exists\mu,\, \y\equiv\mu\mathbf{b}\Mod{q_p},\, s_p \mid \y \\
Q_\y(\x)\equiv 0\Mod{p^n}\end{array}\right\},\]
as required.
\end{proof}

We will write
\[M_{r,s}(p^n)=\card\left\{ \x,\y\Mod{p^n}: 
\begin{array}{c} \exists\lambda,\,\x\equiv\lambda\mathbf{a}\Mod{q_p},\, r_p \mid \x \\
\exists\mu,\, \y\equiv\mu\mathbf{b}\Mod{q_p},\, s_p \mid \y \\
Q_\y(\x)\equiv 0\Mod{p^n}\end{array}\right\}.\]
Then, recalling that $\det(C)=(q/r)^2$ and $\det(D)=q/s$, we see that the leading constant
for the asymptotic formula in Lemma \ref{Nass} takes the shape
\[2\tau_\infty(\Gamma)\kappa_q\prod_{p\nmid q}(1-p^{-1})^2\varpi_p(q,\mathbf{a},\mathbf{b}),\]
with
\begin{eqnarray*}
\kappa_q&=&\frac{q^2}{\phi(q)^2}\sum_{r,s\mid q}\mu(r)\mu(s)\prod_{p\mid q}(1-p^{-1})^2
\left(\lim_{n\to\infty}p^{-4n}M_{r,s}(p^n)\right)\\
&=&\sum_{r,s\mid q}\mu(r)\mu(s)\prod_{p\mid q}
\left(\lim_{n\to\infty}p^{-4n}M_{r,s}(p^n)\right).
\end{eqnarray*}
Our final lemma evaluates $\kappa_q$ in terms of the $p$-adic densities $\varpi_p(q,\mathbf{a},\mathbf{b})$.
\begin{lemma}\label{last}
We have
\[\kappa_q=\prod_{p\mid q}(1-p^{-1})^2\varpi_p(q,\mathbf{a},\mathbf{b}),\]
\end{lemma}
\begin{proof}
Let $q=\prod p^{e(p)}$ with exponents $e(p)\ge 0$. Then, by the Chinese Remainder Theorem we have
\[\prod_{p\mid q}M_{r,s}(p^{ne(p)})=\card\left\{ \x,\y\Mod{q^n}: 
\begin{array}{c} \exists\lambda,\,\x\equiv\lambda\mathbf{a}\Mod{q},\, r \mid \x \\
\exists\mu,\, \y\equiv\mu\mathbf{b}\Mod{q},\, s \mid \y \\
Q_\y(\x)\equiv 0\Mod{q^n}\end{array}\right\}.\]
It follows that
\begin{eqnarray*}
\lefteqn{\sum_{r,s\mid q}\mu(r)\mu(s)\prod_{p\mid q}M_{r,s}(p^{ne(p)})}\\
&=&\card\left\{ \x,\y\Mod{q^n}: 
\begin{array}{c} \gcd(\x,q)=\gcd(\y,q)=1,\\ \exists\lambda,\,\x\equiv\lambda\mathbf{a}\Mod{q}, \\
\exists\mu,\, \y\equiv\mu\mathbf{b}\Mod{q}, \\ Q_\y(\x)\equiv 0\Mod{q^n}\end{array}\right\}.
\end{eqnarray*}
We then see that
\begin{eqnarray*}
\kappa_q&=&\lim_{n\to\infty}q^{-4n}\card\left\{ \x,\y\Mod{q^n}: 
\begin{array}{c} \gcd(\x,q)=\gcd(\y,q)=1,\\ \exists\lambda,\,\x\equiv\lambda\mathbf{a}\Mod{q}, \\
\exists\mu,\, \y\equiv\mu\mathbf{b}\Mod{q}, \\
Q_\y(\x)\equiv 0\Mod{q^n}\end{array}\right\}\\
&=&\prod_{p\mid q}\kappa_p,
\end{eqnarray*}
with
\[\kappa_p=\lim_{n\to\infty}p^{-4n}\card\left\{ \x,\y\Mod{p^n}: 
\begin{array}{c} \gcd(\x,p)=\gcd(\y,p)=1,\\ \exists\lambda,\,\x\equiv\lambda\mathbf{a}\Mod{q_p}, \\
\exists\mu,\, \y\equiv\mu\mathbf{b}\Mod{q_p}, \\
Q_\y(\x)\equiv 0\Mod{p^n}\end{array}\right\}.\]
Since $p^{-4n}=(1-p^{-1})^2p^{-2n}\phi(p^n)^{-2}$ we have
\[\kappa_p=(1-p^{-1})^2\varpi_p(q,\mathbf{a},\mathbf{b}),\]
and the lemma follows.
\end{proof}
\medskip

Theorem \ref{thm:equi} now follows from the criterion in Lemma \ref{equiSC},
using Lemmas \ref{Nass}, \ref{varpis} and \ref{last}.

\end{document}